\documentclass[reqno,english,11pt]{amsart}

\usepackage{float}
\usepackage{times}

\usepackage{amsmath,amsfonts,amssymb,graphicx,amsthm,enumerate,url}
\usepackage[noadjust]{cite}
\usepackage{stmaryrd}
\usepackage{mathrsfs,booktabs}
\usepackage{xifthen,xcolor,tikz,setspace}
\usetikzlibrary{decorations.pathmorphing,patterns,shapes,calc,decorations}
\usetikzlibrary{decorations.pathreplacing}
\usepackage{mathtools,upgreek}
\usepackage[final]{showkeys} 

\definecolor{refkey}{gray}{.75}
\definecolor{labelkey}{gray}{.5}

\usepackage[colorlinks=true]{hyperref}
\usepackage{verbatim}

\usepackage[letterpaper]{geometry}
\makeatletter
\numberwithin{equation}{section}
\numberwithin{figure}{section}

\newtheorem{theorem}{Theorem}[section]
\newtheorem*{theorem*}{Theorem}
\newtheorem{lemma}[theorem]{Lemma}

\newtheorem{claim}[theorem]{Claim}
\newtheorem{proposition}[theorem]{Proposition}

\newtheorem{observation}[theorem]{Observation}
\newtheorem{fact}[theorem]{Fact}
\newtheorem{corollary}[theorem]{Corollary}
\newtheorem{cor}[theorem]{Corollary}

\theoremstyle{definition}{

\newtheorem{definition}[theorem]{Definition}

\newtheorem*{definition*}{Definition}

\newtheorem{remark}[theorem]{Remark}

}

\renewcommand{\P}{\mathbb P}

\newcommand{\cA}{\ensuremath{\mathcal A}}

\newcommand{\cC}{\ensuremath{\mathcal C}}

\newcommand{\cE}{\ensuremath{\mathcal E}}
\newcommand{\cF}{\ensuremath{\mathcal F}}
\newcommand{\cG}{\ensuremath{\mathcal G}}

\newcommand{\cN}{\ensuremath{\mathcal N}}
\newcommand{\cO}{\ensuremath{\mathcal O}}
\newcommand{\cP}{\ensuremath{\mathcal P}}

\newcommand{\cT}{\ensuremath{\mathcal T}}

\newcommand{\cV}{\ensuremath{\mathcal V}}

\newcommand{\sT}{{\ensuremath{\mathscr T}}}

\newcommand{\ps}{{\hat{p}}}
\newcommand{\bbP}{{\mathbb{P}}}
\newcommand{\df}{{\mathcal E}}

\newcommand{\T}{{\mathcal{T}}}

\newcommand{\treelike}{{\mathsf{Treelike}}}
\newcommand{\sparse}{{\mathsf{Sparse}}}

\newcommand{\Prrg}{{\mathbf{P}}_{\textsc{rrg}}}
\newcommand{\Pcm}{{\mathbf{P}}_{\textsc{cm}}}

\newcommand{\Ecm}{{\mathbf{E}}_{\textsc{cm}}}

\newcommand{\Tburn}{T_{\textsc{burn}}}

\renewcommand{\epsilon}{{\varepsilon}}

\DeclareMathOperator{\ber}{\mathrm{Ber}}
\DeclareMathOperator{\bin}{\mathrm{Bin}}

\newcommand{\tv}{{\textsc{tv}}}
\newcommand{\tmix}{{t_{\textsc{mix}}}}

\makeatother

\begin{document}
\title[Random-cluster dynamics on random regular graphs]{Random-cluster dynamics on random regular graphs in \\ tree uniqueness}
\author{Antonio Blanca}
\address{A.\ Blanca \hfill\break
Department of CSE \\ Pennsylvania State University}
\email{ablanca@cse.psu.edu}

\author{Reza Gheissari}
\address{R.\ Gheissari\hfill\break
Departments of Statistics and EECS \\ University of California, Berkeley}
\email{gheissari@berkeley.edu}

\maketitle
\date{\today}

\vspace{-1cm}
\begin{abstract}
        We establish rapid mixing of the random-cluster Glauber dynamics on random $\Delta$-regular graphs for all $q\ge 1$ and $p<p_u(q,\Delta)$, where the threshold $p_u(q,\Delta)$ corresponds to a uniqueness/non-uniqueness phase transition for the random-cluster model on the (infinite) $\Delta$-regular tree. It is expected that this threshold is sharp, and for $q>2$ the Glauber dynamics on random $\Delta$-regular graphs undergoes an exponential slowdown at $p_u(q,\Delta)$. 
        
        More precisely, we show that for every $q\ge 1$, $\Delta\ge 3$, and $p<p_u(q,\Delta)$, with probability $1-o(1)$ over the choice of a random $\Delta$-regular graph on $n$ vertices, the Glauber dynamics for the random-cluster model has $\Theta(n \log n)$ mixing time.
        As a corollary, we deduce fast mixing of the Swendsen--Wang  dynamics for the Potts model on random $\Delta$-regular graphs for every $q\ge 2$, 
    in the tree uniqueness region. Our proof relies on a sharp bound on the ``shattering time'', i.e., the number of steps required to break up any configuration into $O(\log n)$ sized clusters. This is established 
    by analyzing a delicate and novel
    iterative scheme to simultaneously reveal the underlying random graph with clusters of the Glauber dynamics configuration on it, at a given time. 
    \end{abstract}

\vspace{-.25cm}
\section{Introduction}
The random-cluster model is a random graph model, unifying the study of electrical networks, independent bond percolation, and the ferromagnetic Ising/Potts model from statistical physics~\cite{FK,Grimmett}. 
It is defined on a graph~$G=(V,E)$ and parametrized
by an edge probability
$p\in(0,1)$ and cluster weight $q>0$.
Each configuration consists of a subset of edges $\omega \subseteq E$ 
(equivalently $\omega \in \{0,1\}^E$)
 and is assigned probability
\begin{equation}\label{eq:rcmeasure}
\pi_{G,p,q}(\omega) = \frac{1}{Z_{G,p,q}}{p^{|\omega|}(1-p)^{|E|-|\omega|} q^{c(\omega)}}\,, 
\end{equation}
where $c(\omega)$ is the number of connected components in~$(V,\omega)$ and $Z_{G,p,q}$ is a normalizing constant.

Aside from its inherent interest as a model of random networks, the random-cluster model provides an elegant class of Markov Chain Monte Carlo (MCMC) algorithms for sampling from the Ising/Potts model.
For integer $q \ge 2$, a sample $\omega$ from~\eqref{eq:rcmeasure} can be transformed into one for the $q$-state ferromagnetic Potts model by independently assigning a random spin from $\{1,\dots,q\}$ to each connected component of $(V,\omega)$; see, e.g.,~\cite{ES,Grimmett}.
Random-cluster based sampling algorithms, which include the popular Swendsen--Wang algorithm~\cite{SW}, are a widely-used  alternative to the standard Ising/Potts Markov chains since 
the former are often efficient at ``low-temperatures'' (large $p$)
where the latter suffer exponential slowdowns (see~\cite{BS,GuoJer}).

Our focus here is on the Glauber dynamics of the random-cluster model.
Specifically, we consider
the following discrete-time Glauber dynamics chain, which we refer to as the {\em FK-dynamics}.
From a configuration $\omega_t\subseteq E$, one step of the FK-dynamics transitions to
a new configuration $\omega_{t+1}\subseteq E$ as follows:
\begin{enumerate}
	\item Choose an edge $e_t\in E$ uniformly at random;
	\item Set $\omega_{t+1} = \omega_t \cup \{e_t\}$ with probability
	$\left\{\begin{array}{ll}
	\ps := \frac{p}{q(1-p)+p} & \mbox{if $e_t$ is a ``cut-edge'' in $(V,\omega_t)$;} \\
	p & \mbox{otherwise;}
	\end{array}\right.$
	\item Otherwise set $\omega_{t+1} = \omega_t \setminus \{e_t\}$.
\end{enumerate}
We say $e$ is a {\it cut-edge} in $(V,\omega_t)$ 
if changing the state of $e_t$ changes the number of connected components $c(\omega_t)$ in $(V,\omega_t)$. This chain is, by design, reversible with respect to $\pi_{G,p,q}$.

A central question in the study of Markov chains is how the \textit{mixing time}---defined as the number of steps until the Markov chain is close to stationarity starting from the worst possible initial configuration---grows as the size of the graph $G$ increases. Of particular interest in the context of random-cluster and Ising/Potts dynamics is the relation of mixing times to the rich equilibrium phase transitions of the model.

We consider this question when $G$ is a random $\Delta$-regular graph on $n$ vertices.
The study of spin systems and their dynamics on random graphs is quite active~\cite{MS09,DeMo10,MS,DMS,DMSS,GSVY,Eft1,Eft2,BGP}.
Random $\Delta$-regular graphs are a canonical example of graphs having exponential volume growth,
with a non-trivial geometry,
making them an attractive alternative to lattices or trees. 
More generally, the study of spin systems on random graphs
yields insight into hard instances of the classical computational problems
of sampling, counting, learning and testing~\cite{Sly,SlySun,GSVuniq,BBCSV} and features in the study of random constraint satisfaction problems~\cite{GMcol,ZK}. 

The phase transition of the random-cluster model on random $\Delta$-regular graphs is expected to involve three critical points~\cite{BGJ,luczak2006phase,Haggstrom,Jonasson}. 
Most relevant to us would be the critical threshold $p_u(q,\Delta)$ corresponding to a \emph{uniqueness/non-uniqueness} phase transition for the random-cluster model on the infinite $\Delta$-regular wired tree (in which the leaves are externally wired to be in the same connected component). Roughly speaking, the uniqueness/non-uniqueness phase transition captures whether the wired boundary has an effect or not on the configuration near the root of the tree (in the limit as the height of the tree grows). It is believed that the mixing time slows down at $p_u(q,\Delta)$, either polynomially  or exponentially depending on $q\le 2$ or $q>2$.

In this paper, we establish optimal mixing 
for the FK-dynamics on random $\Delta$-regular graphs throughout the uniqueness region $p < p_u(q,\Delta)$ for all real $q\ge 1$ and all $\Delta\ge 3$.

\begin{theorem}\label{thm:main-fk}
	Fix any $q\ge 1,\Delta \ge 3$, and $p < p_u(q,\Delta)$.
	Consider the FK-dynamics on a uniformly random $\Delta$-regular graph on $n$ vertices. With probability $1-o(1)$ over the choice of the random graph $\cG$, the mixing time of the FK-dynamics on $\cG$ is $\Theta\big( n \log n\big)$.
\end{theorem} 
The FK-dynamics are known to be resistant to sharp analysis with the known techniques for Markov chains for spin systems. 
This is due, in part, to the fact that the random-cluster model presents highly non-local interactions:
an update on an edge $e_t$ depends on the entire configuration $\omega_t(E\setminus \{e_t\})$. 
Indeed, the only other setting where the speed of convergence of FK-dynamics
is well-understood via direct analysis is in square subsets of $\mathbb Z^2$~\cite{BS,GL1,GL2,GL3,BGVfull,PS}.
Other bounds to date have been obtained either indirectly, via comparison with global Markov chains using the results of~\cite{Ullrich1,Ullrich2}
(and as a result, these bounds are off by polynomial factors),
or by taking either $p$ very small (e.g., under a Dobrushin-type condition) or very large, or $q$ large.
This is the state of affairs even on the (geometrically trivial) complete graph~\cite{Huber,BS,GLP}.

Our results are tight in the sense that the FK-dynamics is expected to
undergo a slowdown at $p_u(q,\Delta)$, as we describe next. 
The equilibrium phase transition of the random-cluster model on random $\Delta$-regular graphs should qualitatively resemble those on the $\Delta$-regular tree 
and the complete graph. 
Based on this relation, and understandings of those phase diagrams~\cite{BGJ,luczak2006phase,Haggstrom,Jonasson}, it is expected to involve three critical points 
$p_u(q,\Delta) \le p_c(q,\Delta) \le p_u^*(q,\Delta)$. 
The tree uniqueness/non-uniqueness
phase transition at $p_u(q,\Delta)$ manifests on 
the finite $\Delta$-regular tree in the form of existence/non-existence of root-to-leaf paths  under wired boundary conditions. 
%
%
The threshold $p_u^*(q,\Delta)$ corresponds to a (conjectured) second non-uniqueness/uniqueness transition; above this point even the $\Delta$-regular tree under free boundary conditions has root-to-leaf connections
(see~\cite{Haggstrom,Jonasson,GSVY,HeJePe20} for more details). The threshold $p_c(q,\Delta)$, on the other hand, corresponds to
an order-disorder transition
captured by the emergence of a ``giant component'' of linear size
on the random graph 
(which, roughly, imposes ``typical" boundary conditions on its treelike balls). 

When $q\in (1,2]$ the phase transition is of second-order and these three thresholds coincide; namely $p_u(q,\Delta) = p_c(q,\Delta) = p_u^*(q,\Delta)$.
On the other hand when $q>2$, 
the phase transition on random $\Delta$-regular graphs is conjectured to be of \emph{first-order} and $p_u(q,\Delta)<p_c(q,\Delta)<p_u^*(q,\Delta)$.
Here, the uniqueness threshold $p_u(q,\Delta)$ should mark the onset of the \emph{metastability}
phenomenon, and that should persist up to $p_u^*(q,\Delta)$. 
Metastability has
been linked to an exponential slowdown for both random-cluster and Potts Glauber dynamics on the complete graph~\cite{CDLLPS,BS-MF,GSV,GLP},
and the same slowdown is expected to occur on random $\Delta$-regular graphs.
Namely, in the window $(p_u(q,\Delta), p_u^*(q,\Delta))$, the ordered and disordered phases should each be ``metastable" behaving locally (on treelike balls) like the configurations on wired and free trees, respectively. 
The coexistence of these metastable phases 
with exponentially small boundaries, facilitates states from which reversible Markov chains cannot easily escape (i.e., these sets have bad conductance). 
It is thus expected that on random $\Delta$-regular graphs, for every $q>2$, the FK-dynamics mixes exponentially slowly throughout  $(p_u(q,\Delta),p_u^*(q,\Delta))$.
For $q$ sufficiently large, such slowdown was established in~\cite{GSVY} at 
$p = p_c(q,\Delta) \in (p_u(q,\Delta), p_u^*(q,\Delta))$. 

From Theorem~\ref{thm:main-fk} we obtain an efficient MCMC sampling algorithm, for both the random-cluster model and the ferromagnetic Ising/Potts model on 
random $\Delta$-regular graphs in the uniqueness regime. 

\begin{cor}
	\label{cor:sampling}
	Fix any $q\ge 1,\Delta \ge 3$, $p < p_u(q,\Delta)$ and any accuracy parameter $\delta \in (0,1)$. Then, with probability $1-o(1)$ over the choice of the random $\Delta$-regular $n$-vertex graph $\cG$,
    there is a sampling algorithm which, given the graph $\cG$, 
     outputs a random-cluster configuration $\omega$
    whose distribution is within total variation distance $\delta$ 
    of $\pi_{\cG,p,q}$.
    The running time of the algorithm is $O(n (\log n)^3 \log (1/\delta))$.
\end{cor}
The extra $O((\log n)^2)$ factor in the running time of the algorithm comes 
from the (amortized) computational cost of checking
whether the chosen edge is a cut-edge in each step of the FK-dynamics.
This is equivalent to the fully dynamic connectivity problem which has been
thoroughly studied (see, e.g.,~\cite{Thorup,Michigan}).

For integer $q$, the algorithm in Corollary~\ref{cor:sampling} combined with the $O(n)$ cost of translating between the random-cluster and Potts configurations mentioned earlier
yields a sampling algorithm for the ferromagnetic $q$-state Potts
model on random regular graphs up to the Potts uniqueness threshold (the uniqueness thresholds of both these models coincide).
This improves on the best previously known sampling algorithm for both these models in~\cite{BGGSVY}, which runs in 
$\tilde O(n^{6/5})$
time, and it is a ``weak sampler''
in the sense that it outputs samples that are close in total variation distance
to the target distribution but with a \emph{fixed} accuracy. (See also the recent work of~\cite{HeJePe20} for a $\mathrm{poly}(n)$ sampler for \emph{all} $p \in (0,1)$ but provided $q$ is sufficiently large.)

 As another important corollary of Theorem~\ref{thm:main-fk}, we deduce fast mixing of the standard Swendsen-Wang (SW) algorithm for the ferromagnetic $q$-state Potts model~\cite{SW}.
This is an extensively-used global-update Markov chain.
The dynamics starts from a Potts configuration $\sigma_t\in \{1,\ldots,q\}^V$, moves to a ``joint" spin/random-cluster 
configuration $(\sigma_t,\omega_t)$ by including each monochromatic edge independently with probability $p$
and then assigns to each connected component of $(V,\omega_{t})$ a uniform at random spin from $\{1,...,q\}$ to obtain a new Potts configuration $\sigma_{t+1}$ (see~\cite{SW,ES}). 

\begin{corollary}\label{cor:main-SW}
		Fix any integer $q\ge 2$ and $\Delta\ge 3$, and let $p < p_u(q,\Delta)$. Consider the Swendsen-Wang dynamics on a uniformly random $\Delta$-regular graph on $n$ vertices. With probability $1-o(1)$ over the choice of the random graph $\cG$, the mixing time of the Swendsen--Wang dynamics  on $\cG$ is $O\big( n^2 \log n\big)$.
\end{corollary}

Corollary~\ref{cor:main-SW} follows immediately from Theorem~\ref{thm:main-fk} and the comparison results of Ullrich~\cite{Ullrich1,Ullrich2}. Previously, our understanding of the speed of convergence of the SW dynamics on random $\Delta$-regular graphs was very limited.
For the special case of $q=2$, which corresponds to the Ising model, it was established in~\cite{BCV20} that the spectral gap of the SW dynamics is $\Omega(1)$ for all $p<p_u(2,\Delta)$; this implies an $O(n)$ mixing time bound.
In addition, Guo and Jerrum~\cite{GuoJer} established an $O(n^{10})$
mixing time bound for the SW dynamics that applies
to any graph and any $p \in (0,1)$. 
The methods in both of these works are specific to the Ising model ($q=2$) and do not generalize to other values of $q$. Beyond the special case of $q=2$, no sub-exponential bound was previously known for either the FK-dynamics or the SW dynamics throughout the uniqueness regime $p<p_u(q,\Delta)$.  

\subsection*{Proof ideas}
We comment briefly on the techniques and main innovations in our analysis next: for more details and an extended proof sketch, we refer the reader to Section~\ref{sec:proof-strategy}. 
The main ingredient in our proof is an $O(n\log n)$ bound on the ``shattering time'' of the FK-dynamics (Theorem~\ref{thm:pi-exponential-decay}); 
this is the number of steps the chain requires to break up any configuration into connected components of size at most  $O(\log n)$. The bound on the shattering time uses a novel and delicate  
iterative scheme to simultaneously reveal the underlying random graph and the connected components of the FK-dynamics configuration on it at a given time: see Definition~\ref{def:revealing-FK-on-graph} and Figures~\ref{fig:revealing-fig1}--\ref{fig:revealing-fig2}.
While revealing procedures are a standard tool in the study of both random graphs 
and of the random-cluster model, their combined analysis is highly non-trivial, as the law of the random-cluster configuration at an edge depends on the global geometry of the graph. 
To our knowledge, this the first direct upper bound for the shattering time of the FK-dynamics in any setting. In fact, understanding the shattering time is usually the main obstacle for proving rapid mixing of the FK-dynamics on other graphs: e.g., on the complete graph, the shattering time is not known and 
only loose mixing time bounds (off by $\Theta(n^2)$ factors) can be derived~\cite{BS-MF}. 

Once the dynamics has shattered, 
we use standard methods (i.e., censoring~\cite{PWcensoring})
to reduce the analysis of the FK-dynamics to localized dynamics in balls of radius $o(\sqrt n)$ centered at each vertex, but with random boundary conditions induced by the current state outside the ball. In random $\Delta$-regular graphs, these balls are ``treelike" 
and, after shattering, their boundary conditions are ``almost free'', in  that only $O(1)$ vertices in their boundaries are connected through the external configuration.
This implies that the FK-dynamics mix quickly and satisfy a log-Sobolev inequality akin to a product measure in each of these balls. 
The last ingredient in our proof is an exponential decay of correlation property (sometimes called spatial mixing) between the root and boundary of such balls. A delicate point is that since these balls have radius $\Theta(\log n)$, we need exact control on the rate of this exponential decay  to sustain the union bound over the $n$ balls. 

\begin{remark}
We expect our methods for the analysis of the shattering phase
to have applications to other locally tree-like graphs, e.g., wired trees and Erd\H{o}s--R\'{e}nyi random graphs. 
In the latter case, however, the possibility of having a small number of vertices of large degree poses technical obstructions to direct extension of our methods. Whereas this should not affect the equilibrium phase diagram of the model, interestingly, in the case of the Glauber dynamics for the Ising model on an Erd\H{o}s--R\'{e}nyi random graph, the high maximum degree is known to slow down the high-temperature mixing time to $n^{1+\Omega(\frac{1}{\log \log n})}$~\cite{MS09}. 
\end{remark}



\subsection*{Organization of paper}
The rest of the paper is organized as follows. In Section~\ref{sec:prelim}, we provide a number of preliminary definitions and notations we will use. In Section~\ref{sec:proof-strategy}, we give a detailed proof overview highlighting some of the key novelties in our arguments. Our revealing procedures to bound the shattering time are the focus of Section~\ref{sec:shattering}. In Section~\ref{sec:correlation-decay-treelike} we establish the sharp rate of spatial mixing on treelike graphs with sparse boundary conditions. We combine these to conclude the proof of the upper bound of Theorem~\ref{thm:main-fk} in 
 Section~\ref{sec:proof-main-theorem}. We prove the matching lower bound on the mixing time in Section~\ref{sec:proof-lower-bound}.

\section{Preliminaries}\label{sec:prelim}
In this section, we collect some standard definitions and properties that are necessary to present our proofs, and to which the reader can refer throughout. See the standard texts~\cite{BollobasBook}, \cite{Grimmett}, and~\cite{LP} for more details on random graphs, the random-cluster model, and Markov chain mixing times, respectively. 

\subsection{Random $\Delta$-regular graphs}
We begin by considering the underlying geometry we work on. Fix $\Delta\ge 3$ and consider the uniform distribution $\Prrg$ over $\Delta$-regular graphs on $n$ vertices.
(Let us always assume $n$ is such that $\Delta n$ is even, so that such a graph exists.)
We identify the vertices $V(\cG)$ with the set $\{1,...,n\}$, and the randomness of $\Prrg$ will be over the edge-subset of $\{ij = ji: 1\le i,j\le n \}$. 
Throughout this paper, we set $d:=\Delta-1$ for convenience.

\subsubsection*{Random graphs are treelike}A key ingredient in our proof is the fact that random $\Delta$-regular graphs are locally treelike. While this can be formalized in various ways, we use a notion that is most relevant to this paper, and applies uniformly to all vertices (as opposed to a vertex chosen uniformly at random).  

For a graph $\cG = (V(\cG),E(\cG))$ and a vertex $v \in V(\cG)$, 
we define the ball of radius $R$ around $v$ as: 
$$B_R(v):= \{w \in V(\cG): d(w,v) \le R\}\,,$$
where $d(w,v)$ is the graph distance. 
For a vertex set, $B\subset V(\cG)$, define $E(B) = \{v,w\in B: vw\in E(\cG)\}$. 

\begin{definition}\label{def:k-treelike}
We say that a graph $G = (V,E)$ is $L$-$\treelike$ if there is a set $H\subset E$ with $|H|\le L$ such that the graph $(V,E\setminus H)$ is a tree. 
\end{definition}

\begin{definition}\label{def:(L,R)-treelike}
	We say that a $\Delta$-regular graph $\cG = (V(\cG),E(\cG))$ is $(L,R)$-$\treelike$ 
	if for every $v \in V(\cG)$ the subgraph $(B_R(v), E(B_R(v))$ is $L$-$\treelike$.
\end{definition}

\begin{fact}\label{fact:random-graph-treelike}
Fix any $\Delta\ge 3$. For every $\delta>0$, there exists  $L(\delta,\Delta)$ such that if $R = (\frac 12 - \delta)\log_d n$,  we have 
\begin{align*}
    \Prrg \big(\cG \mbox{ is } (L,R)\mbox{-}\treelike\big) = 1-o(n^{-1})\,.
\end{align*}
\end{fact}
We include a short proof of Fact~\ref{fact:random-graph-treelike} after introducing the configuration model in Section~\ref{subsec:proof-sketch-revealing}. 
It is known that when ${R} > \frac 12 {\log_d n}$, the number of cycles in every ball $B_R(v)$ goes to $\infty$ with $n$.

\subsection{The random-cluster model}
For a graph $G  = (V,E)$, recall the definition of the random-cluster model from~\eqref{eq:rcmeasure}. We say an edge $e\in E$ is \emph{open} or \emph{wired} if $\omega(e)=1$ and \emph{closed} or \emph{free} if $\omega(e)=0$. We say two vertices are \emph{connected in $\omega$} if they are in the same connected component of the sub-graph $(V,\{e\in E: \omega(e)= 1\})$. For a vertex set $\cV\subset V$, denote by $\cC_\cV(\omega)$ the union of connected components (clusters) containing $v\in \cV$ in this sub-graph. For a configuration $\omega$ and edge set $A\subset E$, we use $\omega(A)$ for the restriction of $\omega$ to $A$.

\subsubsection*{Boundary conditions}To help study the random-cluster measure, we introduce boundary conditions. 

\begin{definition}
     A random-cluster \emph{boundary condition} $\xi$ on $G=(V,E)$ is a partition of $V$, such that the vertices in each element of the partition are identified with one another. The random-cluster measure with boundary conditions $\xi$, denoted $\pi^{\xi}_{G,p,q}$, is the same as in~\eqref{eq:rcmeasure} except the number of connected components $c(\omega)= c(\omega;\xi)$ would be counted with this vertex identification, i.e., if $v,w$ are in the same element of $\xi$, they are always counted as being in the same connected component of $\omega$ in~\eqref{eq:rcmeasure}. 
     In this manner, the boundary condition can alternatively be seen as ghost ``wirings" of the vertices in the same element of $\xi$.
\end{definition} 

\noindent
     The \emph{free} boundary condition, $\xi = 0$, is the one whose partition consists only of singletons.
     For a subset $\partial V \subset V$, the \emph{wired} boundary condition on $\partial V$, denoted $\xi = 1$, is the one whose partition has all vertices of $\partial V$ in the same element and all vertices of $V\setminus \partial V$ as singletons; i.e., $\xi = \{\partial V\} \cup \bigcup \{v: v\in V\setminus \partial V\}$. For boundary conditions $\xi,\xi'$ we say that $\xi \le \xi'$ if $\xi$ is a finer partition than $\xi'$. 
     We have the following important monotonicity in boundary conditions: for any two boundary conditions $\xi$ and $\xi'$ with $\xi\ge \xi'$, 
 we have $\pi_{G,p,q}^{\xi} \succcurlyeq \pi_{G,p,q}^{\xi'}$ where $\succcurlyeq$ denotes stochastic domination.

 \subsubsection*{Uniqueness/non-uniqueness transition on the $\Delta$-regular tree.}
 As the geometry of the random graph is locally treelike, its dynamical transition point should be inherited from a  transition on the $\Delta$-regular tree. Throughout this paper, we denote by $\T_h := \T_{h,\Delta}= (V(\cT_h),E(\cT_h))$ the rooted (at $\rho$) $\Delta$-regular complete tree of depth $h$ (the root has $\Delta$ children, and all other vertices have $\Delta- 1$ children and one parent). Since the tree has depth $h<\infty$, evidently it is not actually $\Delta$-regular, and has leaves $\partial \cT_h= \{w\in V(\T_h): d(\rho,w) =h\}$ (where $d(\cdot,\cdot)$ denotes graph distance); observe that
\begin{align}\label{eq:T-h-vertices-edges}
    |V(\cT_h)| & = 1+  \Delta \sum_{i=1}^{h} d^{i-1} \le 2\Delta d^{h}\,,\qquad \mbox{and} \qquad|E(\cT_h)| = \Delta \sum_{i=1}^{h} d^{i-1} \le 2\Delta d^{h}\,,
\end{align}
and $|\partial \cT_h|= \Delta d^{h-1}$. 
The wired boundary condition ``1" is  the one that wires all vertices of $\partial \cT_h$ together.

For every $\Delta \ge 3$ and $q\ge 1$, the random-cluster measure $\pi_{\cT_h,p,q}^{1}$ undergoes a transition at $p_u(q,\Delta)$: when $p<p_u(q,\Delta)$ the probability that $\rho$ is connected to $\partial \T_h$ in $\omega$ goes to $0$ as $h\to\infty$, whereas when $p>p_u(q,\Delta)$ it stays bounded away from zero~\cite{Haggstrom}. (While in general $p_u(q,\Delta)$ does not have a closed form, it can be expressed as the root of an explicit formula: see~\cite{Haggstrom,BGGSVY}.) 
A key fact (see~\cite[Theorem 1.5]{Haggstrom}) we will use is that whenever  $p<p_u(q,\Delta)$ we have that $\hat p$ (the probability of a cut-edge being open) satisfies
\begin{align}\label{eq:hat-p-inequality}
    \hat p:= \frac{p}{q(1-p)+p} < \frac 1d\,, \qquad \text{where}\qquad d:=\Delta-1\,.
\end{align}



\subsection{Markov chain mixing times.}
Consider a (discrete-time) Markov chain with transition matrix $P$ on a finite state space $\Omega$, reversible with respect to an invariant distribution $\pi$; denote the chain initialized from $x_0$ by $(X_t^{x_0})_{t\ge 0}$. Its \emph{mixing time} is given by 
\begin{align*}
    \tmix = \tmix(1/4)\,,\qquad \mbox{where}\qquad \tmix(\epsilon) = \min \{t: \max_{x_0 \in \Omega} \| P(X_t^{x_0} \in \cdot ) - \pi\|_\tv \le \epsilon\}\,,
\end{align*}
where the total-variation distance between $\mu$ and $\nu$ is given by 
$$\|\mu - \nu\|_\tv = \frac 12 \|\mu - \nu\|_{1} = \inf_{(U,V)\sim \mathbb P:  U\sim \mu, V\sim \nu} \mathbb P(U\ne V)\,.$$
Here the infimum runs over all couplings of $\mu,\nu$. 
By this definition, to bound the mixing time, it suffices to bound the \emph{coupling time} of the dynamics; i.e., if we construct a coupling $\mathbb P$ of the steps of the chain such that for each $x_0,y_0 \in \Omega$, we have 
$\mathbb P(X_T^{x_0} \neq X_T^{y_0}) \le 1/4$, then $\tmix \le T$. It is a standard fact that $\tmix(\delta) \le \tmix \log (2\delta^{-1})$. See chapters 4--5 of~\cite{LP} for more details. 

\subsubsection*{A coupling for the FK-dynamics.}
Recall the definition of the FK-dynamics from the introduction. Note that in the presence of boundary conditions $\xi$, the only change is that in step (2) of the FK-dynamics transitions, the status of $e$ being a cut-edge is dictated by whether its presence changes $c(\omega_t; \xi)$. 

For the FK-dynamics, there is a canonical choice of coupling known as the \emph{identity coupling}. This is the coupling that couples the evolution of two copies of the FK-dynamics, $(X_t^{x_0})$ and $(X_t^{y_0})$, by using the same random edge $e_t$ and the same uniform random number $U_{e_t,t}$ to decide whether to add or remove $e_t$. When $q \ge  1$, the identity coupling is a \emph{monotone coupling}, in the sense that if $X_t^{x_0} \le X_t^{y_0}$ then $X_{t+1}^{x_0} \le X_{t+1}^{y_0}$ with probability $1$. The identity coupling can also be extended to a simultaneous coupling of all the Markov chains $(X_t^{x_0})$ indexed by their initial configuration $x_0\in \{0,1\}^E$ (i.e., a \emph{a grand coupling}), so that if $x_0\le y_0$ we have $X_t^{x_0}\le X_t^{y_0}$ for all $t\ge 0$. As a consequence, the coupling time starting from any pair of configurations is bounded by the coupling time starting from the free $x_0 = 0$ and wired $y_0 = 1$ configurations.


\section{Extended proof sketch}\label{sec:proof-strategy}
In this section, we provide a detailed sketch of our proof of Theorem~\ref{thm:main-fk}, outlining the structure of the argument and highlighting some of the key technical difficulties we encountered.
Most of the paper is dedicated to upper bounding the mixing time of the FK-dynamics by $O(n \log n)$, so the sequel is dedicating to sketching that proof. The matching lower bound follows from coupling a certain projection of the FK-dynamics to a product chain and is derived in Section~\ref{sec:proof-lower-bound}.

\subsection{Proof outline}

Let $\cG = (V(\cG),E(\cG))$ be an $n$-vertex graph. 
Let $(X_{t}^1)_{t \ge 0}$ and $(X_{t}^0)_{t \ge 0}$ be two realizations of the FK-dynamics started from the all-wired 
and all-free configurations, respectively, and coupled via the identity coupling as defined in Section~\ref{sec:prelim}.

Our goal is to show that there exists $T = O(n \log n)$ such that for every vertex $v \in V(\cG)$, with probability $1-o(n^{-1})$, the configurations $X_{T}^1$ and $X_{T}^0$ agree on the $\Delta$ edges incident $v$, denoted
\begin{align}\label{eq:N-v}    \cN_v := \{e\in E(\cG): v\in e\}\,.
\end{align}
A union bound over the $n$ vertices would then imply that under the identity coupling $X_{T}^1 = X_{T}^0$ 
with probability $1-o(1)$. By the monotonicity of the FK-dynamics under the identity coupling, this would show that the mixing time of the FK-dynamics is at most $T = O(n \log n)$. 

There are two key stages to establishing this coupling, each of which we describe next.

\bigskip
\noindent\textbf{Stage I.}
    In the first stage of the coupling, we show that after
    an initial \emph{burn-in period} of $T= O(n\log n)$ steps, 
    the configuration $X_{T}^1$ 
    is \emph{shattered}.
    That is, its connected components have 
    constant size in expectation, and every component is of size $O(\log n)$ with high probability;
    more precisely, we show that the size of the connected components have 
    exponential tails. Since $X_{T}^1 \ge X_{T}^0$, the same holds for $X_{T}^0$.
    
    The intuition behind our proof of shattering after $T$ steps goes as follows.     
    Consider the balls $B_r(v)$ for $v \in V(\cG)$ where $r=O(1)$ is a sufficiently large constant.
    For each $v \in V(\cG)$,
    let $(Z_t^v)$ denote the chain on $B_r(v)$ with a fixed \emph{wired} boundary condition outside of $B_r(v)$. 
    The evolution of $(Z_t^v)$ and $(X_t^1)$ can be coupled with the same identity coupling
    by ignoring the updates outside of $B_r(v)$ for $(Z_t^v)$; by monotonicity, for every $v \in V(\cG)$ and $t \ge 0$, we have
    $Z_t^v \ge X_t^1(B_r(v))$. Consequently, we can even take the \emph{minimum} (intersection) of the chains $Z_t^v$ over $v\in V(\cG)$, to obtain a configuration $\omega_t$ by setting $\omega_t(e) = 1$ if $Z_t^v(e) = 1$ for \emph{all} $v\in V(\cG)$. Since all these chains are coupled using the same randomness, we maintain the domination $\omega_t \ge X_t^1$ for all $t\ge 0$.

    
%

    We thus focus on showing the shattering property for $\omega_T$. Notice that we can bound the connected component of a vertex $v$ in $\omega_T$ via an iterative exploration process.
    We initialize a set $A$ as the connected component $\mathcal C_v$ of $v$ in $Z_T^v(B_r(v))$ and initialize $\partial A$ to $\mathcal C_v \cap \partial B_r(v)$. For each $u\in \partial A$ (while $\partial A \neq \emptyset$), we
        \begin{enumerate}
            \item Add to $A$ the connected component $\mathcal C_u$ of $u$ in $Z_T^u(B_r(u))$
           \item   Add to $\partial A$ all vertices in $\mathcal C_u \cap \partial B_r(u)$. Remove $u$ from $\partial A$.
    \end{enumerate}
    The procedure ends when $\partial A$ is empty and outputs an edge set $A$ necessarily containing the component of $v$ in $\omega_T$. See the depiction in Figure~\ref{fig:overview}. 
    The nature of this exploration process lends itself naturally to comparison with a branching process in which the  ``children" of $u$ are the vertices connected to $u$ through $Z_T^u$. (It turns out that the revealing of these configurations can be done in such a way that although they are all coupled, the dependencies between configurations $Z_T^u(B_r(u))$ are negligible: we comment on this later.) We will show that with high probability over $\cG$, the resulting branching process is \emph{sub-critical}. 
    
    
    To see this, first note that since the mixing time on $B_r(v)$ is $O(1)$ ($r$ is constant), 
    after $T = \Theta(n\log n)$ steps, 
    enough updates have occurred in each ball $B_r(v)$ so that 
    the chains $(Z_t^v)$ have all mixed with high probability.
    Hence, up to a small error, 
    we can consider instead the branching process 
    where the number of children of $v$ is given by the number of connections
    of $v$ to the boundary of $B_r(v)$ in a sample from $\pi_{B_r(v)}^1$.
    
    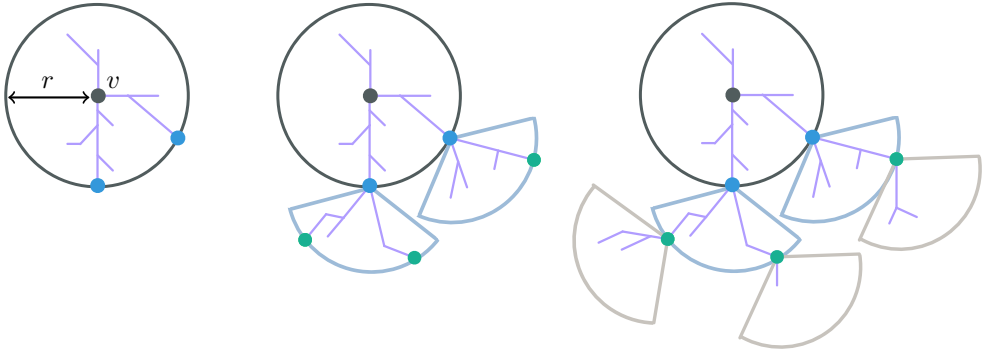
\begin{figure}
        \centering
        \begin{tikzpicture}
        \node at (0,0) {
        \includegraphics[width = .8\textwidth]{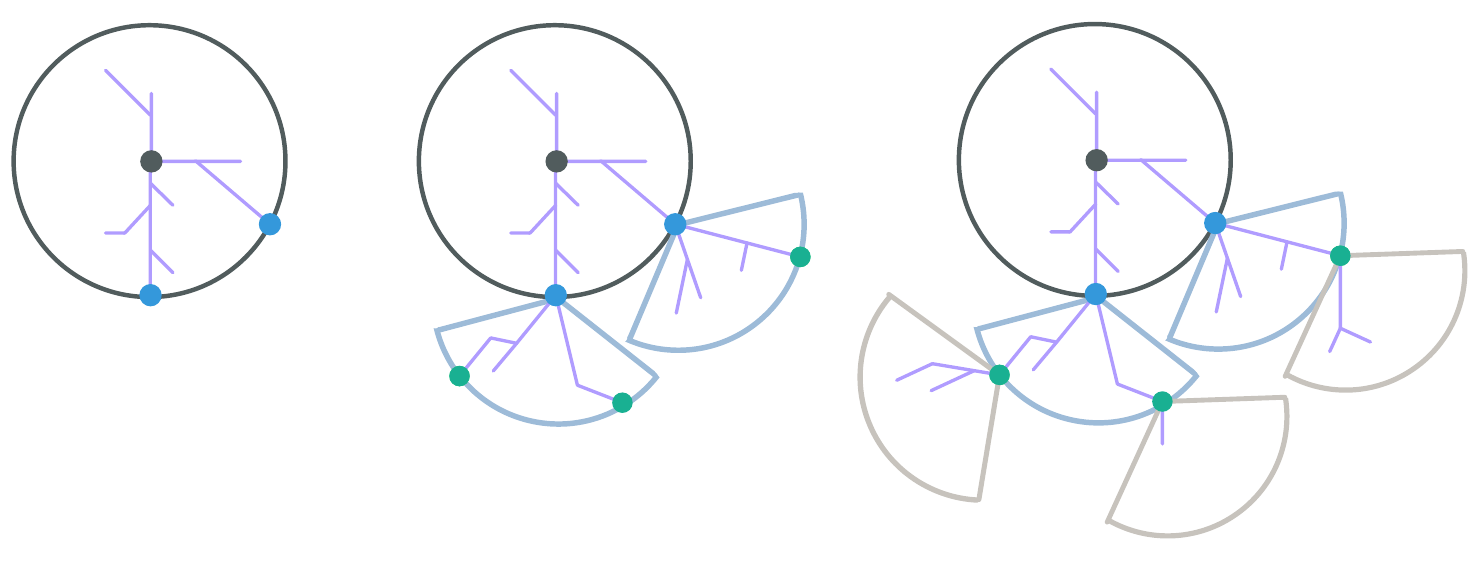}};
        \draw[<->, thick] (-6.45,1.05)--(-5.37,1.05);
        \node[font = \small] at (-5.925, 1.25) {$r$};
        \node[font = \small] at (-5.05, 1.25) {$v$};

        \end{tikzpicture}
        \vspace{-.35cm}
        \caption{Left: Upon exposing the localized FK-dynamics $Z_T^v$ on $B_r(v)$, the connected component of $v$ (purple) reaches two boundary points of $B_r(v)$ (blue). Middle: The revealing procedure then exposes their localized configurations, in their balls of radius $r$. Right: The procedure continues in that manner until these connected components die out.}
        \label{fig:overview}
    \end{figure}
    
   Now, most $O(1)$-sized balls in a random $\Delta$-regular graph are trees.
   A key characteristic of the uniqueness regime $p<p_u(q,\Delta)$ is that in the wired $\Delta$-regular tree $\cT_r$, the expected number of leaves connected to $v$ under $\pi_{\cT_r}^1$ is less than $1$ as long as $r$ is large. 
    As long as the role played by non-tree balls in $\cG$ is bounded, this would imply the desired sub-criticality of the dominating branching process. 
   We in fact need concentration bounds on the number of explored vertices in this branching process; towards this we show that $\hat p$ is the actual exponential decay rate of root-to-leaf connectivities on $\Delta$-regular (wired) trees.
    
    \begin{lemma}\label{lem:exp-decay-wired-tree-new}
    	Let $\T_h$ denote the rooted $\Delta$-regular complete tree of depth $h$ and let
    	$p<p_u(q,\Delta)$. 
    	Let $(1,\circlearrowleft)$ be the wired boundary condition on $\partial \cT_h$ that additionally wires the root of $\T_h$ to $\partial \cT_h$. 
    	There exists a constant $C = C(p,q,\Delta)$ such that for every $h$ and every leaf $u\in \partial \cT_h$, 
    	\begin{align*}
    	\pi_{\cT_h}^{(1,\circlearrowleft)} (u~\textrm{is connected to the root of }~\T_h) \le C \hat p^h\,.
    	\end{align*}
    \end{lemma}
    Since there are $O(d^r)$ leaves in $\cT_r$, 
    the lemma implies that the expected number of connections 
    from $v$ to the boundary is $O ((\hat p d)^r)$, which is less than one for $r$ large
    (as $\hat p d < 1$ when $p < p_u(q,\Delta)$).
    The reason we establish this decay for the boundary condition $(1,\circlearrowleft)$, instead of simply the wired one, is to eliminate the potential dependencies between the chains $(Z_t^u)$ through their roots. 
    
    To conclude our sketch of the  ideas in Stage I, 
    we mention two fundamental challenges to implementing the above approach. 
    First, since all the chains $(Z_t^v)$ are coupled via the identity coupling, revealing their configurations while maintaining some independence is delicate (see Lemma~\ref{lem:revealing-process-law}). We perform this revealing by additionally wiring the root to the boundary as hinted by the $(1,\circlearrowleft)$, and for each $u$, only revealing the new randomness needed to run the resulting chain on $B_R(u)$ up to time $T$.   
    Roughly speaking, the wired boundary conditions allow us to evolve the un-revealed configuration in $B_r(v)$ in isolation.
    
    Secondly, not every ball $B_r(v)$ in $\cG$ will be a tree, and there are strong correlations between the short cycles of the underlying graph and the places where the random-cluster configuration is more wired. A key contribution of our work is to construct a simultaneous revealing procedure for the random graph $\cG$ with the overlayed random-cluster configuration of $\omega_T$ in a manner that handles these dependencies and can be approximated by the above sub-critical branching process; see Definition~\ref{def:revealing-FK-on-graph}.  
    
    Putting all these ideas together, we establish the following exponential tail bound (shattering estimate) on cluster sizes of $X_{T}^1$ after a burn-in period of $O(n\log n)$ time.



\begin{theorem}\label{thm:pi-exponential-decay}
	Let $p<p_u(q,\Delta)$ and suppose that $\cG$ is sampled from $\Prrg$, the uniform distribution over $\Delta$-regular graphs on $n$ vertices.
	Then,
	 for every $v \in V(\cG)$, 
	 $k \ge 1$  and $T \ge C n\log n$, where $C > 0$ is a sufficiently large constant, 
	 with probability $1- \exp ( - \Omega(k)) - O(n^{-5})$, the random graph $\cG$ is such that 
	\[
	P(|\cC_v(X_{T}^1)|\ge k) \le \exp ( - \Omega(k)) + O(n^{-5})\,.
	\]
\end{theorem}

\noindent
(Recall that $\cC_v(X_{T}^1)$ denotes the component of vertex $v$ in $X_{T}^1$.)
By a union bound, Theorem~\ref{thm:pi-exponential-decay} implies that all components of $X_{T}^1$ are of size at most $O(\log n)$ with high probability. Theorem~\ref{thm:pi-exponential-decay} is proved in \S\ref{sec:shattering}.

    Using the above arguments, we can further show that
    for each $v \in  V(\cG)$
    the boundary condition induced on the ball $B_R(v)$ of radius $R= (\frac 12 - \delta) \log_d n$
    by the configuration of
    $X_{T}^1$ on the edges outside of $B_R(v)$ is typically $K$-\emph{sparse}, i.e., the boundary condition induces only $K = O(1)$ many connections on $\partial B_R(v)$. Theorem~\ref{thm:k-R-sparse-whp} establishes that this property holds for all $v\in V(\cG)$ simultaneously with high probability.

    \bigskip\noindent\textbf{Stage II.} \ After the initial $T = O(n \log n)$ steps of the burn-in phase, the configurations $X_{T}^1$ and $X_{T}^0$ shatter 
    and induce sparse boundary conditions (with up to $O(1)$ vertices wired through the boundary) on every ball $B_R(v)$ of radius $R= (\frac 12 - \delta) \log_d n$ with high probability.
    It remains to show that the copies of the FK-dynamics 
    will couple on $\cN_v$ except with probability $1-o(n^{-1})$ in
    an additional $O(n \log n)$ steps.
    
    Starting at time $T$, we
    consider localized copies of the FK-dynamics in each ball $B_R(v)$ with $v\in V(\cG)$.
    This is done by ignoring (or \emph{censoring}) the moves of the dynamics outside of $B_R(v)$
    which has the effect of ``freezing'' the two distinct boundary conditions induced by 
     $X_{T}^1(E(\cG)\setminus B_R(v))$ and $X_{T}^0(E(\cG)\setminus B_R(v))$ on the boundary of $B_R(v)$. With the sparse boundaries conditions frozen on $\partial B_R(v)$, the two coupled chains continue to run inside $B_R(v)$, and we can more easily analyze their configurations near $v$.
     The censoring technology of~\cite{PWcensoring} implies that if these censored chains are coupled on $\cN_v$, then so are the original chains.  
      
    In Lemma~\ref{lemma:main-tree-mixing}, we show that if $X_{T}^1(E(\cG)\setminus B_R(v))$  and $X_{T}^0(E(\cG)\setminus B_R(v))$ induce sparse boundary conditions on $B_R(v)$, and $\cG$ is $(L,R)$-$\treelike$ with $L=O(1)$ (see Definition~\ref{def:(L,R)-treelike}), 
	the mixing time of the FK-dynamics on $B_R(v)$ is $O( d^R \log (d^R))$. In fact, we can establish a tight bound on the log-Sobolev constant
					of the FK-dynamics on $B_r(v)$
					under sparse boundary conditions, showing that 
					$\tmix(\epsilon) = O(d^R\log (d^R/\varepsilon))$.
					This slightly stronger fact turns out to be crucial for deducing the tight $O(n \log n)$
					bound on the mixing time of the FK-dynamics on $\cG$, i.e., 
					without an additional $\textrm{polylog}(n)$ factor.
 
 With this optimal bound on the local mixing on treelike balls, 
 we know that the localized chains have all mixed after $O(n \log n)$ steps of the FK-dynamics. Therefore, the probability that two instances of the FK-dynamics on $B_R(v)$ with distinct sparse boundary conditions $\xi$ and $\xi'$ are not coupled on $\cN_v$ is given by the total variation distance between $\pi_{B_R(v)}^\xi$ and $\pi_{B_R(v)}^{\xi'}$ on $\cN_v$.
 We show that this distance is $O(\hat p^{2R})$. 
 
 \begin{proposition}\label{prop:influence-probability-new}
 	Consider a vertex $v$ in a $\Delta$-regular graph $\cG$. 
 	If $B_R(v)$ is $L$-$\treelike$ and $\xi,\xi'$ are any two $K$-\emph{sparse} boundary conditions on $\partial B_R(v)$, there exists a constant $C = C(p,q,\Delta,L,K)$ such that
 	\begin{align*}
 	\|\pi_{B_R(v)}^{\xi}(\omega(\cN_v)\in \cdot)- \pi_{B_R(v)}^{\xi'}(\omega(\cN_v)\in \cdot ) \|_\tv \le C \ps^{2R}\,.
 	\end{align*}
 \end{proposition}
\noindent
 (Recall that we say a boundary condition is $K$-sparse when there are only $K$ boundary wirings.)  We stress the importance of obtaining the sharp $\hat p^{2}$ decay rate here for the spatial mixing to support a union bound over $n$ vertices.
  Since
  $\hat p <1/d$ and $R =  (\frac 12 - \delta) \log_d n$, we have $\hat p^{2R} = o(n^{ - 1})$, but any 
  weaker bound on the decay rate would force us to choose a 
  larger $R$, which would cross the threshold at which point balls of $\cG$ are no longer $(L,R)$-$\treelike$ for $L = O(1)$, and we would lose control over the mixing time on $B_R(v)$. 
  
  The proof of this spatial mixing property is based on the fact that in order for information to travel from the boundary of $B_R(v)$ to $\cN_v$ there must be $\emph{two}$ disjoint open paths from $\cN_v$ to non-singleton elements of $\xi$ in $\partial B_R(v)$. 
  We contrast this to the more traditional bound on influence by the existence of a \emph{single} connection from the center of a ball to its boundary, which in our setting would only yield a bound of $\hat p^R$. (Such a bound by a single connectivity event is the one traditionally used on amenable graphs like $\mathbb Z^2$ to go from spatial mixing with \emph{any} positive rate of exponential decay to fast mixing: see~\cite{MOS,Alexander,BS}.)

\section{The FK-dynamics shatters quickly on random graphs}\label{sec:shattering}

Our first goal in this section
is to prove 
Theorem~\ref{thm:pi-exponential-decay} establishing existence of $\Tburn = O(n\log n)$ such that for $t\ge \Tburn$, the configuration $X_{\cG,t}^1$ is shattered. We will then use this to conclude that the boundary conditions $X_{\cG,t}^1$ induces on any ball of volume $o(\sqrt n)$ are $O(1)$-sparse. Let us now be more precise. 

	\begin{definition}\label{def:k-sparse-bc}
		A random-cluster boundary condition $\xi$ on an edge-subset $H\subset E(\cG)$ is said to be $K$-$\sparse$ if the number of vertices in non-trivial (non-singleton) boundary components of $\xi$ is at most $K$. 
	\end{definition}
	
	\begin{definition}\label{def:(k,r)-sparse}
		A random-cluster configuration $\omega$ on $\cG = (V(\cG),E(\cG))$ is \emph{$(K,R)$-$\sparse$} 
		if, for every $v \in V(\cG)$,
		the boundary conditions induced on $B_R(v)$ by $\omega(E(\cG)\setminus E(B_r(v)))$ are $K$-$\sparse$. 
	\end{definition}

	The following key result asserts that the boundary of every ball about a vertex is $O(1)$-$\sparse$ with high probability after an $O(n\log n)$ burn-in time: this is proven in Section~\ref{subsec:main-revealing-proof}.
	
	\begin{theorem}\label{thm:k-R-sparse-whp}
		Fix $p<p_u(q,\Delta)$. There exists $C(p,q,\Delta)$ such that for every $t\ge C n \log n$, 
		the following holds. For every $\delta>0$, if $R:= (\frac 12 -\delta)\log_d n$, there exists $K(p,q,\Delta,\delta)$ such that with $\Prrg$-probability $1-o(1)$, $\cG$ is such that 
		\begin{align}
		P\big(X_{\cG,t}^1 \mbox{ is $(K,R)$-$\sparse$}\big)\ge 1-O(n^{-2})\,.
		\end{align}
\end{theorem}

\begin{remark}
By monotonicity of the FK-dynamics, for every $\cG$, we have that $X_{\cG,t}^1\succeq \pi_{\cG}$, from which it follows that both Theorem~\ref{thm:k-R-sparse-whp}, and the exponential tails of Theorem~\ref{thm:pi-exponential-decay}, hold under $\pi_{\cG}$, i.e., if one replaces $X_{\cG,T}^1$ by an equilibrium configuration $\omega\sim\pi_{\cG}$. 
\end{remark}

In Section~\ref{subsec:proof-sketch-revealing}, we construct the relevant revealing procedures for FK-dynamics clusters on random graphs, and define the branching process we dominate it by. In Sections~\ref{subsec:properties-couplings-revealings}--\ref{subsec:analysis-of-branching-process}, we analyze these processes, and in Section~\ref{subsec:main-revealing-proof}, we complete the proofs of Theorems~\ref{thm:pi-exponential-decay} and~\ref{thm:k-R-sparse-whp}.


\subsection{Couplings and revealing schemes for the FK-dynamics on random graphs}\label{subsec:proof-sketch-revealing}
In this section, we summarize the key couplings and revealing schemes for the connected components of $X_{\cG,t}^1$. These are fundamental to the proof of shattering for $X_{\cG,t}^1$ in the uniqueness region after an $O(n\log n)$ burn-in time.  

\subsubsection{The configuration model}
%
%
The configuration model $\Pcm$ is a distribution over multigraphs on $n$ vertices and fixed degree distribution, which we take to be $\Delta$ for every vertex, defined as follows~\cite{BollobasConfig}. Give every vertex $v\in \{1,...,n\}$ $\Delta$-half-edges and select a matching on the $\Delta n$ many half-edges uniformly at random to form the $\Delta n/2$ edges of the graph. Let $\mathfrak M_n$ be the set of possible edges (the set of pairs of half-edges). 

The configuration model is a useful tool for studying the random $\Delta$-regular graph, as the distribution $\Prrg$ is equal to the distribution $\Pcm(\cdot \mid \cG \in \Gamma_{\textsc{rrg}})$ where $\Gamma_{\textsc{rrg}}$ is the event that the graph $\cG$ is simple (i.e., has no self-loops or multi-edges). In particular, it is standard (see e.g.,~\cite{BollobasConfig}) that $\Pcm(\Gamma_{\textsc{rrg}})>c$ for some $c(\Delta)>0$, and therefore for any event $\Gamma$,
\begin{align}\label{eq:CM-to-RRG}
    \Prrg(\Gamma) = \Pcm(\Gamma, \Gamma_{\textsc{rrg}})\big(\Pcm(\Gamma_{\textsc{rrg}})\big)^{-1} \le c^{-1}\Pcm(\Gamma)\,.
\end{align}
Refer to  the book~\cite{FriezeBook} for more on the configuration model. We will use~\eqref{eq:CM-to-RRG}, with an iterative revealing scheme of a matching of the $\Delta n$ half-edges, to analyze the random $\Delta$-regular graph.

The configuration model lends itself to revealing procedures. Towards introducing the joint revealing procedure for the random graph $\cG \sim \Pcm$ and the configuration $X_{\cG,t}^1$, let us first recall a standard revealing procedure for random $\Delta$-regular graphs according to $\Pcm$ on its own. This procedure is useful to proving random graph estimates for the configuration model and $\Delta$-regular random graph. It also serves as a building block for the revealing procedure of the random graph together with the FK-dynamics configuration.  

The following iterative algorithm is a way to sample from the configuration model for a given degree sequence. The fact that this gives a valid sample from $\Pcm$ is straightforward after naturally identifying samples from $\Pcm$ with samples from the uniform distribution over matchings on $\Delta n$.

			\begin{definition}\label{def:revealing-graph}
Assign every $v\in \{1,...,n\}$, $\Delta$ half-edges. Suppose $f$ is a (possibly random) function from edge-sets $\mathcal A\subset \mathfrak M_n$ (a set of matched pairs of half-edges) to a half-edge not matched in $\cA$.  
\begin{enumerate}
\item Initialize the set of exposed edges as $\mathcal A_0 = \emptyset$.
\item For every $1\le m\le \frac{\Delta n}{2}$, match $f(\mathcal A_{m-1})$ to a half-edge selected uniformly
at random from the remaining un-matched ones to form the edge $e_m$. Let $\mathcal A_m := \mathcal A_{m-1} \cup \{e_m\}$\,.
\end{enumerate}
\end{definition}

Observe, importantly, that the choice of next half-edge to match
(given by the function $f$) can be \emph{adaptive}, specifically, adapted to the filtration generated by $(\mathcal A_0,...,\mathcal A_{m-1})$.  

Definition~\ref{def:revealing-graph} provides an adaptive sampling method from the configuration model distribution~$\Pcm$ (see e.g.,~\cite{MW}) and can be used to prove myriad properties of random $\Delta$-regular graphs. In particular, it yields a simple proof of Fact~\ref{fact:random-graph-treelike} that $\cG\sim \Prrg$ is $(L,R)$-$\treelike$ for $R \le n^{\frac 12 - \delta}$ and $L=O(1)$; see Section~\ref{subsec:properties-couplings-revealings}.

 \begin{remark}
  The definitions of the random-cluster model~\eqref{eq:rcmeasure}, and the FK-dynamics extend naturally to multigraphs, where $G = (V,E)$ is such that $V$ is identified with $\{1,...,n\}$ and $E\subset \mathfrak M_n$ is a multiset. The random-cluster model and FK-dynamics then live over subsets of $E$, identified with~$\omega: E\to \{0,1\}$, and connectivity components of a configuration $\omega$ are understood naturally.
  \end{remark}

\subsubsection{A coupling of localized FK-dynamics chains}
Our goal is to simultaneously expose edges of $\cG\sim \Pcm$ while revealing the FK-dynamics configuration $X_{\cG,t}^1$ at time $t$ on $\cG$. We show that under their joint distribution 
the size of the connected components of $X_{\cG,t}^1$ have exponential tails; this in turn implies that the boundary condition on $B_r(v)$ 
is $O(1)$-$\sparse$ (see Definition~\ref{def:k-sparse-bc}).

Note that a ball of radius $O(\log n)$ about a vertex $v$ may have many cycles---indeed it may encompass the entire graph $\cG$---but a typical FK  cluster of \emph{size} $O(\log n)$ does not use most of these cycles. Thus, we expose the edges of $\Pcm$ guided by the revealing of the random-cluster component of a vertex $v$ in $X_{\cG,t}^1$; in this way, to expose the $\cC_{v}(X_{\cG,t}^1)$ we will not have to reveal much of the random graph.

There are two key difficulties to consider when constructing a joint revealing process for $(\cG, X_{\cG,t}^1)$: 
\begin{enumerate}
    \item Under either of $X_{\cG,t}^1$ or e.g., the random-cluster measure $\pi_{\cG}$, the value $\omega(e)$ on an edge $e$ shown to belong to $E(\cG)$, affects the distribution of the remainder of the underlying random graph. 
    
    \item Unlike the random-cluster measure $\pi_{\cG}$, the law of $X_{\cG,t}^1$ does not satisfy any domain Markov property. Indeed, the distribution of $X_{\cG,t}^{1}(e)$ conditionally on some $X_{\cG,t}^1(A)$ is quite difficult to analyze.
\end{enumerate}

The key to overcoming these obstructions will be to reveal the configurations of a family of FK-dynamics chains that are localized (in the sense that their distribution only depends on a small $O(1)$ sized subset of edges of the graph) and whose concatenation stochastically dominates the distribution of $X_{\cG,t}^1$. 
Let us be more precise next and explicitly construct a coupling of a family of localized FK-dynamics
chains.  

\begin{definition}\label{def:censored-Glauber-chains}
For a graph $\cG$ and edge subset $A\subset E(\cG)$, let $\partial A$ be the set of vertices in $V(A)$ that are adjacent to vertices of $V(\cG)\setminus V(A)$, and let $\pi_A^1$ be the random-cluster measure on $A$ with wired boundary conditions on $\partial A$. Let $(X_{A,t}^{1})_{t\ge 0}$ be the FK-dynamics chain that starts from all wired on $E(\cG)$, censors (ignores) all updates in $E(\cG)\setminus A$, and makes FK-dynamics updates w.r.t.\ $\pi_{A}^1$ when it updates edges in $A$. 
\end{definition}

Importantly, the wiring on $\partial A$ ensures that the law of $X_{A,t}^1$ does not depend on $E(\cG)\setminus A$.

\begin{definition}\label{def:coupling-Glauber-chains}
For any graph $\mathcal{G}$,
we can construct $(\mathcal X_{t}^1)_{t\ge 0}= \big\{(X_{A,t}^1)_{t \ge 0}\big\}_{A\subset E(\mathcal G)}$, a grand monotone coupling  of the ensemble of FK-dynamics
chains $(X_{A,t}^1)_{t\ge 0}$ for $A\subset E({\mathcal{G}})$ as follows: 
\begin{enumerate}
\item Initialize $X_{A,0}^1 \equiv 1$ for all $A\subset E({\mathcal G})$; i.e., the all wired configuration on $A$. 
\item Let $(e_{t})_{t\ge1}=(e_{1},e_{2},...)$ be drawn i.i.d. from $E({\mathcal G})$.
\item Let $(U_{e,t})_{e\in E({\mathcal G}),t\ge1}$ be a sequence of i.i.d. uniform
random variables on $[0,1]$. 
\end{enumerate}
For $A\subset E(\cG)$, construct $(X_{A,t}^{1})_{t\ge 1}$ as follows: for each $t\ge1$, set
$X_{A,t}^{1}(e)=X_{A,t-1}^{1}(e)$ for $e\ne e_{t}$ and 
\[X_{A,t}^{1}(e_{t})= 
\begin{cases}
1 & \mbox{if }e_{t}\notin A;\\
1 & \mbox{if }e_{t}\in A~\mbox{and } U_{e_t,t} \le \varrho; \\
0 & \mbox{if }e_{t}\in A~\mbox{and }U_{e_t,t} > \varrho;
\end{cases}
\]
for $\varrho = \pi_{A}^{1}\big(\omega(e_{t})=1\mid\omega(A\setminus \{e_t\})= X_{A,t-1}^1(A\setminus \{e_t\})\big)$; 
i.e., if $e_{t}\in A$, we resample $e_{t}$ given the
remainder of the configuration on $A$, together with the wired
boundary condition on $\partial A$, using the same uniform random variable $U_{e_t,t}$ for every $X_{A,t}^{1}$ such that $e_t\in A$. 
\end{definition}

As in the grand coupling for different initializations, this is a monotone coupling. In particular, we have $X_{\cG,t}^1\le X_{A,t}^1$ for all $A\subset E(\cG)$ and thus $X_{\cG,t}^1 \le \bigcap_{A\subset E(\cG)} X_{A,t}^1$. 

A key observation for our revealing process is that for every $A$, the configuration $X_{A,t}^{1}$ depends only on:
\begin{enumerate}
    \item the number of updates amongst $(e_{s})_{s\le t}$ that belong
to $A$, which we denote by $\kappa_{A,t}$;
    \item the choice of edges to be updated on $A$ on those $\kappa_{A,t}$ updates; we denote such set by $\mathcal O_{A,\kappa_{A,t}}$; and
    \item the family of uniform random variables on those
edges, $(U_{e,s})_{e\in A,s\le t}$.
\end{enumerate}
With this observation in hand, we can extend this to a coupling of $(X_{A,t}^1)$ averaged over $\cG\sim \Pcm$. 

\begin{definition}\label{def:coupling-Glauber-chains-average-graph}
Let $\mathbb{P}_{t}^{1}$ be the distribution over pairs $(\mathcal{G},\omega_t)$
where $\omega_{t}$ is a random-cluster configuration on $\mathcal{G}$ that results by first
drawing $\mathcal{G}\sim\Pcm$, then drawing $\omega_{t}\sim P(X_{\cG,t}^{1}\in\cdot)$.
Likewise, for every set $A\subset \mathfrak M_n$, let $\mathbb{P}_{A,t}^{1}$
be the distribution over pairs $(\mathcal{G},\omega_{A\cap E({\mathcal{G}}),t})$
where $\omega_{A,t}\sim P(X_{A\cap E(\mathcal{G}),t}^{1}\in\cdot)$.
Couple, under the distribution $\bbP$, the family of distributions $(\mathbb{P}_{A,t}^{1})_{A\subset \mathfrak M_n,t\ge 1}$
by selecting the same random graph $\mathcal{G}\sim\Pcm$
for all of them, then using the coupling of Definition~\ref{def:coupling-Glauber-chains} of the family  $(X_{A,t}^{1})_{A,t}$. 
\end{definition}

In this manner, we have constructed a monotone coupling of the family
$(\mathcal{G},(X_{A,t}^{1})_{t\ge1})_{A\subset \mathfrak M_n}$. Note that we use this coupling for sets $A$ which we know have $E(\cG)\cap A = A$, so that the averaging is only over the edges of $E(\cG)\setminus A$, which we earlier noted $X_{A,t}^1$ is independent of; thus the role of this coupling is only to put the random graphs with their random-cluster configurations on the same probability space. We defer detailed discussion of the properties of the coupling to Section~\ref{subsec:properties-couplings-revealings} (after constructing the revealing procedure in the sequel) but emphasize that by construction, if $A \cap B = \emptyset$, the only dependency of $X_{A,t}^1$ and $X_{B,t}^1$ is through the distributions of the binomial random variables  $\kappa_{A,t}$ and $\kappa_{B,t}$.

\subsubsection{The joint revealing procedure}
We now construct a revealing procedure for $\cG$ and a configuration $\tilde \omega_t$ on $\cG$
that stochastically dominates $X_{\cG,t}^{1}$. 
Fix $r$ to be chosen as a large constant (depending on $p,q,\Delta$) later. 

\begin{definition}\label{def:B-r-out}
Given an exposed set of edges $\mathcal A$ of the random graph $\cG\sim \Pcm$, we define $B_r^{out}(v)= B_r^{out}(v;\cA)$ as the ball of radius $r$ in $E(\mathcal G)\setminus \mathcal N_v(\mathcal A)$ where $\mathcal N_v(\mathcal A)$ is the set of edges in $\cA$ incident to $v$. 
We drop the $\cA$ from the notation when understood contextually.
\end{definition}

For an edge-set $\cA_0\subset \mathfrak M_n$ revealed to be part of $E(\cG)$ and a vertex set $\cV_0\subset V(\cG)$, we construct a joint iterative procedure to expose (a set containing) the connected components $\cC_{\cV_0}(X_{\cG,t}^1(E(\cG)\setminus \cA_0))$. The examples to have in mind are (1) $\cA_0 = \emptyset$ and $\cV_0 = \{v\}$ and (2) $\cA_0= E(B_R(v))$ and $\cV_0 = \partial B_R(v)$. 

Through this process we will keep track of the following variables at each step:
\begin{itemize}
    \item $\mathcal{A}_{m}$: the set of edges of the random graph
          that have been revealed by step $m$;
    \item $\mathcal{V}_k$: the set of vertices in the $k$-th generation we want to explore out of;
    \item $\tilde \omega_{m}$: the random-cluster configuration revealed up to step $m$;
    \item $\mathcal{F}_{m}$: elements of the filtration with respect to which the configuration $\tilde \omega_m$ on $\mathcal{A}_m$ is measurable.
\end{itemize}
The process is defined as follows (see Figures~\ref{fig:revealing-fig1} and~\ref{fig:revealing-fig2} for a depiction of several steps of this process):
\begin{center}
	\fbox{
		\parbox{0.98\textwidth}{
			\begin{definition}\label{def:revealing-FK-on-graph} 
			\textbf{Initialize:} \qquad $k = 0$;
            \qquad $m = 1$;
            \qquad $\mathcal{F}_{0}=\emptyset$;
            \qquad $\mathcal{A}_{0} \subset E(\cG)$;
            \qquad $\mathcal{V}_{0} \subset V(\cG)$;
            
            \medskip
            \textbf{for each}~$k \ge 0$~\textbf{while}~$\mathcal{V}_k \neq \emptyset$
            
           \quad \textbf{for each}~$v \in \mathcal{V}_k$
\begin{enumerate}[\quad\quad\quad~1.]\setlength{\itemsep}{3mm}
\item Set $v_m = v$. Conditionally on $\mathcal{A}_{m-1}$, reveal the edges of the random graph in $B_{r}^{out}(v_m)$ and set 
$$\mathcal{A}_{m}:=\mathcal{A}_{m-1}\cup E(B_{r}^{out}(v_m));$$ 
Let $A_m := \mathcal A_m \setminus \mathcal A_{m-1}$ be the set of new edges revealed in $\cG$.

\item Reveal the triplet $\mathcal H_{A_m,t} := \{\kappa_{A_m,t},\mathcal O_{A_m, \kappa_{A_m,t}},(U_{e,s})_{e\in A_m,s\le t}\}$ conditionally on $\mathcal{F}_{m-1}$. Recall that
$\kappa_{A_m,t}$ is the number of updates of the FK-dynamics in $A_m$,
$\mathcal O_{A_m, \kappa_{A_m,t}}$ is the sequence of edges to be updated in $A_m$,
and $(U_{e,s})_{e\in A_m,s\le t}$ is the family of uniform random variables used for the edges updates in $A_m$.

\item Construct the random-cluster configuration $X_{A_m,t}^{1}(A_m)$
from the triplet $\mathcal H_{A_m,t}$ by simulating the steps of the FK-dynamics on $A_m$ with wired boundary conditions.
Concatenate $X_{A_m,t}^{1}(A_m)$ with $\tilde \omega_{m-1}$ to obtain a new configuration $\tilde \omega_{m}$ on $\cA_m \setminus \cA_0$. 

\item Add to $\mathcal V_{k+1}$ all vertices of $\partial (\cA_{m}\setminus \cA_{0})$ that are in the component of $\cV_0$ in $\tilde \omega_{m}(\cA_m \setminus \cA_0)$ and are not in $\bigcup_{\ell \le k} \cV_\ell$. 

\item Let $\mathcal{F}_{m} = (\mathcal F_{m-1}, \mathcal H_{A_m,t})$ and increase $m$ by $1$.
\end{enumerate}
\end{definition} 
	}}
\end{center}

\bigskip
Let $\kappa_{\emptyset}$
be the first $k$ such that $\mathcal{V}_{k}=\emptyset$ and
let $\mathfrak m_{k_{\emptyset}} = \sum_{k=0}^{k_{\emptyset}} |\mathcal V_k|$. Let $\tilde \omega_t = \tilde \omega_t(\cA_{\mathfrak m_{k_\emptyset}} \setminus\cA_0)$ be the random-cluster configuration revealed when the process terminates. 
The key observation about the above process is that we can control the cluster of $\cV_0$ in $X_t^1(E(\cG)\setminus \cA_0)$ by the set $\cA_{\mathfrak m_{k_\emptyset}}$; the size of this set will then be approximately controlled by comparison to a sub-critical  branching process in the following subsection.
 
\begin{observation}\label{obs:key-observation}
The connected components of $\cV_0$ in $X_t^1(E(\cG)\setminus\cA_0)$ 
are a subset of the connected components of $\cV_0$ in $\tilde \omega_t (E(\cG)\setminus\cA_0)$. In particular, the number of vertices in  non-trivial (i.e., non-singleton) components 
of the boundary condition $X_t^1(E(\cG)\setminus \cA_0)$ induces on $\cA_0$ is less than the number of vertices in non-trivial components of the boundary condition $\tilde \omega_t(E(\cG)\setminus \cA_0)$ induces on $\cA_0$. The edges in both of these sets of connected components are subsets of the edge-set $\cA_{\mathfrak m_{k_\emptyset}}\setminus \cA_0$. 
\end{observation}

With Observation~\ref{obs:key-observation} in hand, we focus on obtaining the exponential tail bound of Theorem~\ref{thm:pi-exponential-decay} for $\cC_v(\tilde \omega_t)$ (the component of $v$ in $\tilde \omega_t$) and likewise, the sparsity bound of Proposition~\ref{prop:nontrivial-bdy-component-bound} for $\mathfrak V_{B_R(v)}(\tilde \omega_t)$.

\subsubsection{Constructing a dominating branching process}
Towards proving Proposition~\ref{prop:nontrivial-bdy-component-bound}, we construct a (non-Markovian, size-dependent) branching process which we will show stochastically dominates the sequence $(\cV_k)_{k\ge 0}$ of our joint reveleaing process.
This process $(Z_k)_{k\ge 0}$ will then be shown to be sub-critical and satisfy exponential tail bounds on its total population, implying the same for the cluster of $\cV_0$ in $\tilde \omega_t$. 

\begin{definition}\label{def:branching-process}
Initialize $Z_{0}=|\mathcal{V}_{0}|$, and let $(Z_{k})_{k\ge1}$ be the (size-dependent) branching process, which for each $k$, has progeny $(\chi_{i,k})_{i\le Z_k}$ drawn i.i.d.\ from the following distribution:
\begin{enumerate}
\item With probability $n^{-1/2}$, let  $\chi_{i,k} = |V(\cT_r)| \sum_{\ell \le k} Z_\ell$ and say the progeny number  $\chi_{i,k}$ is $\mathsf{Bad}$.
\item Otherwise, sample $\chi_{i,k}$ from the distribution of the
number of leaves in the connected component of the root under $\pi_{\mathcal{T}_{r}}^{(1,\circlearrowleft)}$ (the random-cluster measure on the $\Delta$-regular tree of depth $r$ with a wired boundary
condition and with the root also wired to $\partial \cT_r$). 
\end{enumerate}
Let $Z_{k+1} = \sum_{i\le Z_k} \chi_{i,k}$; that is, the $i$-th member of the $k$'th generation gets $\chi_{i,k}$ many children. 
\end{definition}

\begin{figure}
    \centering
    \begin{tikzpicture}
    \node at (-3.5,0) {
    \includegraphics[width = .33\textwidth] {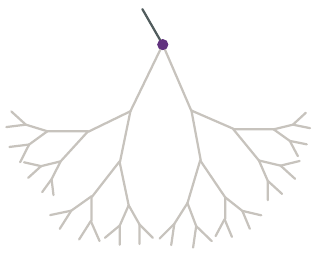}};
    \node at (-3.15,1.7) {$v_1$};

    \node at (3.5,0) {
    \includegraphics[width = .33\textwidth] {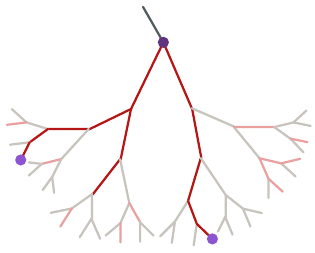}};
        \node at (3.85,1.7) {$v_1$};
    \end{tikzpicture}
    \caption{Left: We initialize the process with $r=5$ from $\cV_0= \{v_1\}$ (dark purple) and incident edge $\cA_0$ (black). The process begins by revealing $A_1 = B_r^{out}(v_1)$, depicted in gray. Right: The process then reveals the configuration $X_{A_1,t}^1$, given by $\kappa_{A_1,t}$ steps of FK-dynamics for $\pi_{A_1}^1$, to form $\tilde \omega_1$ (open edges in red/pink). Vertices in $\partial A_1$ shown to be connected to $v_1$ in $X_{A_1,t}^1(A_1)$ are added to $\cV_1$ (purple).}
    \label{fig:revealing-fig1}
    \centering
    \begin{tikzpicture}
    \node at (-4,0) {
    \includegraphics[width = .35\textwidth] {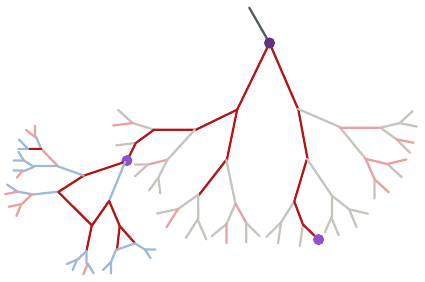}};
        \node[font = \tiny] at (-5.5,-.21) {$v_2$};
    \node at (3.5,-.65) {
    \includegraphics[width = .35\textwidth] {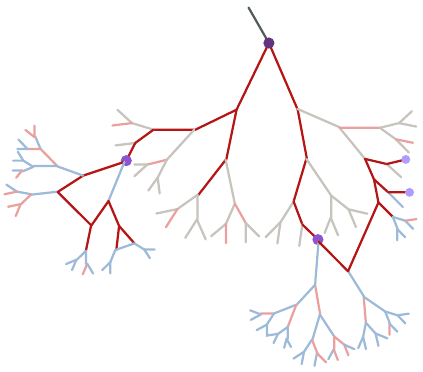}};
            \node[font = \tiny] at (4.75,-1.6) {$v_3$};
    \end{tikzpicture}
    \caption{Left: Proceeding from above, in the next generation, starting from  $v_2\in \cV_1$, reveal the edges of $B_t^{out}(v_2)$ in $\cG$; in this case, this is not a tree, but is disjoint from $\cA_1$, so that $A_2= B_r^{out}(v_2)$. The configuration $X_{A_2,t}^1$ is generated and concatenated with $\tilde \omega_1$ to form $\tilde \omega_2$. Right: For $v_3\in \cV_1$, $B_r^{out}(v_3)$ is a tree, but it intersects $\cA_2$. As such, $A_3 = B_r^{out}(v_3)\setminus \cA_2$. Running the FK-dynamics on $A_3$ with all-wired boundary conditions ensures that $X_{A_3,t}^1$ is nonetheless independent of the configuration we had revealed in $\tilde \omega_2$. The light purple vertices are connected to $\cV_0$ in $\tilde \omega_3$ and are added to $\cV_2$ to form the next generation.}
    \label{fig:revealing-fig2}
\end{figure}

Note that this is not a branching process in the traditional sense, since the progeny distribution is not i.i.d.\ and depends on the population up to that generation. Nonetheless, we will show good tail bounds on $(Z_k)_{k\ge 0}$ by dominating it by sub-critical branching processes between the $\mathsf{Bad}$ steps.

To justify the above construction, let us formalize the relation between $(Z_k)_k$ and the revealed vertices of the process in Definition~\ref{def:revealing-FK-on-graph}, $(\cV_k)$.  
Intuitively, we want to identify vertices $v_m\in \cV_k$ with those of generation $k$ in $(Z_k)$; the \emph{progeny} of $v_m$ will then be those vertices added to $\cV_{k+1}$ in step (4) of Definition~\ref{def:revealing-FK-on-graph}.  
Item (1) from the progeny distribution of Definition~\ref{def:branching-process} corresponds to situations where: 
\begin{enumerate}
    \item $B_{r}^{out}(v_m)$ intersects $\cA_{m-1}$;
    \item $B_r^{out}(v_m)$ is not a tree; or
    \item There are an insufficient number of updates on $B_r^{out}(v_m)$ for $X_{A_m,t}^1$ to mix.
\end{enumerate}
Examples of situations (1)--(2) were depicted in Figure~\ref{fig:revealing-fig2}. The $n^{-1/2}$ probability assigned to these bad situations by the dominating branching process comes from the fact that Theorem~\ref{thm:k-R-sparse-whp} requires us to consider $|\cA_0|$ of size $n^{-1/2 + \delta}$, and thus any edge has at least probability $n^{-\frac 12-\delta}$ of intersecting $\cA_0$. 
On the complement of situations (1)--(3) above, $X_{A_m,t}^1$ is mixed, and is comparable to the $(1,\circlearrowleft)$-tree of depth $r$.

\subsubsection{Comparing the revealing procedure and branching process} 
We conclude the section by stating the main two lemmas comparing the revealing procedure to the branching process defined above. Towards stating these, denote by $\tmix(\cT_r,(1,\circlearrowleft)$ the mixing time of FK-dynamics on $\cT_r$ with $(1,\circlearrowleft)$ boundary conditions, and define the burn-in time
\begin{align}\label{eq:T-burn}
\Tburn = \Tburn(C_0, r) & : =   C_0 n \log n \cdot   \frac{t_{\textsc{mix}}(\mathcal{T}_{r},(1,\circlearrowleft)) }{|E(\cT_{r})|}\,.
\end{align}
Recall, the definition of the update numbers $(\kappa_{A_m,t})_m$ and define, for every $t\ge \Tburn$, the event 
\begin{align}\label{eq:mathfrak-E}
\mathfrak E_t = \mathfrak E_t^\infty \qquad\mbox{where}\qquad \mathfrak E_t^m : = \bigg\{(e_s)_{s\le t}: \bigcap_{l= 1}^{m} \Big\{\kappa_{A_l,t} \le \frac{4|E(\cT_r)|}{\Delta n} t \Big\} \bigg\}\,.
\end{align}
Standard tail estimates for binomial random variables will imply that $\mathfrak E_t$ holds with high probability. 

Let $\mathfrak m_{0} = 0$ and for each $k\ge 0$, let $\mathfrak m_{k+1}=\mathfrak m_{k}+|\mathcal{V}_{k}|$, i.e., the total number of exposed vertices on the boundaries of explored balls, and in the same connected component as $\cV_0$ before the exploration for the $(k+1)$-th generation begins. This will be the quantity which we compare to the population of the branching process of Definition~\ref{def:branching-process}. 
More precisely, on the event $\mathfrak E_t$, by construction of $(Z_k)$, and the choice of $\Tburn$, we are able to show the following stochastic domination. 
\begin{lemma}\label{lem:branching-process-domination}
There exists $C_0(p,q,\Delta)$ in the definition of~\eqref{eq:T-burn} such that the following holds for every $t\ge \Tburn$. For every $\cA_0, \cV_0$ such that $|\cA_0|,|\cV_0|\le n^{\frac 12 - \delta}$ for $\delta>0$, every  $K>0$ fixed, and every $\ell\ge 1$, 
\begin{align*}
\big(|\mathcal{V}_{j}| \mathbf 1\{{\mathfrak E_t^{\mathfrak m_{j}}}\} \mathbf 1{\{\mathfrak m_{j-1} \le n^{1/2 - \delta/ 2} \}}\big)_{j\le \ell} & \preccurlyeq  (Z_{j})_{j\le \ell}\,.
\end{align*}
\end{lemma}
In this manner, we will have reduced the analysis of the set of exposed vertices through the revealing process of $(\cG, \tilde \omega_t)$, and thus, the clusters of $X_{\cG,t}^1$, to the analysis of the process $(Z_k)$, which except on some rare $\mathsf{Bad}$ increments, is a simple branching process with subcritical progeny distribution dictated by connectivity probabilities in the wired measure $\pi_{\cT_{r}}^{(1,\circlearrowleft)}$. 
We will establish the following tail estimate for $(Z_k)$.

\begin{lemma}\label{lem:modified-branching-process-tail-bound}
    Suppose $p<p_u(q,\Delta)$ and fix any $\delta>0$, any $M \ge 1$, and any $1\le Z_0\le n^{\frac{1}{2} -\delta}$. There exist $r_0(p,q,\Delta),C(p,q,\Delta,M),K_0(p,q,\Delta,M)$ such that for every $r \ge r_0$ fixed and every $0<\lambda \le n^{\frac 12 - \delta}$,
    \begin{align*}
        \mathbb P\Big(\sum_{k\ge 0} Z_k \ge K_0 Z_0 + \lambda  \Big) \le  C\exp ( - \lambda/C)+ C e^{M}n^{- \delta M}\,.
    \end{align*}
\end{lemma}

\subsubsection{Outline of remainder of section}
Having sketched the key revealing procedures and the way they fit together to provide the desired bounds on the clusters of $X_{\cG,t}^1$, let us prove the various relations and bounds claimed above. In Section~\ref{subsec:properties-couplings-revealings}, we prove various key properties of the configuration model revealing process of Definition~\ref{def:revealing-graph} and the coupling of Definition~\ref{def:coupling-Glauber-chains} that will be central to the analysis of the revealing procedure of Definition~\ref{def:revealing-FK-on-graph}. Then in Section~\ref{subsec:domination-by-the-branching-process}, we show that the size-dependent branching process $(Z_k)$ of Definition~\ref{def:branching-process} stochastically dominates the FK process $(\cV_k)$ of Definition~\ref{def:revealing-FK-on-graph} on a high-probability event, proving Lemma~\ref{lem:branching-process-domination}. In Section~\ref{subsec:analysis-of-branching-process}, we analyze the process $(Z_k)$ by comparing its population to the sum of $O(1)$ many sub-critical branching processes to deduce Lemma~\ref{lem:modified-branching-process-tail-bound}. In Section~\ref{subsec:main-revealing-proof}, we combine these ingredients to conclude Theorem~\ref{thm:pi-exponential-decay} and Proposition~\ref{prop:nontrivial-bdy-component-bound}, and from that Theorem~\ref{thm:k-R-sparse-whp}.

\subsection{Key properties of the revealing procedure for $(\cG, \tilde \omega_t)$}\label{subsec:properties-couplings-revealings}
In this section, we describe some of the key properties of the coupling constructed in Definition~\ref{def:coupling-Glauber-chains}, and the revealing procedure constructed for the clusters of $\cV_0$ in $\tilde \omega_t$ in Definition~\ref{def:revealing-FK-on-graph}. The following preliminary lemmas describe the law of the random graph edges and overlaying FK configurations through the revealing process. 

\subsubsection{Properties of the configuration model revealing procedure} 
We begin with the following lemma on the law of the random graph $\cG$ conditionally on a set $\cA_m$ which we have revealed to be a subset of $E(\cG)$. Recall the configuration model's revealing procedure from Definition~\ref{def:revealing-graph} and say a vertex is \emph{discovered }if at least
 one of its half-edges has been matched, and \emph{exhausted} if all
 of its half-edges have been matched. 

\begin{lemma}
\label{lem:ball-intersection-probability}
Let $\mathcal{A}$ be
any set of edges (pairing of half-edges) revealed to belong to $E(\mathcal{G})$. For every
$r$, 
\begin{align*}
\sup_{v\in\{1,...,n\}}\Pcm\big(B_{r}^{out}(v)\cap\mathcal{A}\ne\emptyset\,\mbox{ or }\,B_{r}^{out}(v) \mbox{ is not a tree}\mid\mathcal{A} \big)\le & \frac{2\Delta d^{r}(|V(\mathcal{A})|+ d^r)}{n-(|V(\mathcal{A})|+d^{r})}\,.
\end{align*}
\end{lemma}

\begin{proof}
     Fix any edge-set $\cA$. We can sample from the conditional distribution $\Pcm (\cdot \mid \cA)$ by defining the adaptive scheme $f$ in Definition~\ref{def:revealing-graph} so that it first matches the half-edges belonging to $\cA$, yielding the set $\cA_{|\cA|/2} = \cA$ after $|\cA|/2$ steps, then setting $f$ to do a breadth-first search (BFS) of $B_r^{out}(v)$: this latter part is done by choosing $f$ so that it first exhausts $v$, then exhausts each of the neighbors of $v$, and so on. 
     
     Revealing the entire set $B_r^{out}(v)$ takes at most $|E(\cT_r)|$ many steps beyond $|\cA|/2$. If for every $m\in \{|\mathcal A|/2+1,...,|\cA|/2+d^r\}$ the half-edge from $f(\mathcal A_{m-1})$ is not matched to a half-edge belonging to a vertex in $V(\mathcal A_{m-1})$, then evidently $B_r^{out}(v)\cap \cA = \emptyset$ and $B_r^{out}(v)$ is a tree. 
     
     Since on each of these steps, the half-edge $f(\cA_m)$ is being matched to a u.a.r.\ un-matched half-edge, uniformly over the at most $|E(\cT_r)|$ steps it takes to reveal $B_r^{out}(v)$, the probability that the half-edge it is matched to belongs to $\cA_{m-1}$ is at most
     $$\frac{d (|V(\cA)|+d^r)}{\Delta n - (d |V(\cA)|+d^r)}\le \frac{|V(\cA)|+d^r}{n-(|V(\cA)|+d^r)}\,.$$ 
     (The first inequality here uses the fact that in the BFS of $B_r^{out}(v)$, there are at most $d^r$ vertices of the ball that have been discovered but not exhausted.)
     Union bounding over the at most $|E(\cT_r)|\le 2\Delta d^r$ such attempts yields the desired bound. 
\end{proof}

We can use a similar reasoning as the proof above to deduce a proof of Fact~\ref{fact:random-graph-treelike} as follows. 

\begin{proof}[\textbf{\emph{Proof of Fact~\ref{fact:random-graph-treelike}}}]
Fix any $v$ and choose $f$ so that the revealing scheme performs a BFS revealing of $B_R(v)$. In order for $B_R(v)$ to not be $L$-$\treelike$, it must be the case that for more than $L$ different $m$'s in the first $|E(\cT_R)|$ steps, the half-edge $f(\cA_{m-1})$ is being matched to a half-edge belonging to $\cA_{m-1}$. (If there were at most $L$ such steps, then the removal of the at-most $L$ edges formed by those at-most $L$ matchings in the revealing scheme, evidently leaves a tree, so that $B_R(v)$ would be $L$-$\treelike$.)
Uniformly over $\cA_{m-1}$, the probability of this in the $m$'th step is at most $d^{R+1}/(\Delta (n - m))$. Summing over the at most $|E(\cT_R)|$ many such attempts while revealing $B_R(v)$, we find that for every $\ell\ge 1$, 
\begin{align}\label{eq:B-r-v-not-treelike}
    \Pcm ( B_R(v) \mbox{ is not $\ell$-$\treelike$}) \le \bbP\Big(\bin\big(|E(\cT_R)|, \frac{d^R}{n - |E(\cT_R)|}\big) >\ell \Big)\,.
\end{align}
Recall that the standard Chernoff bound applied to a Poisson binomial distribution with mean $\mu = N p$ says that for every $s \ge \mu$, 
\begin{align}\label{eq:Chernoff-Poisson-binomial}
    \mathbb P\big(\bin(N,p) \ge s\big) \le e^{ s - \mu} \Big(\frac{s}{\mu}\Big)^{-s}\,.
\end{align}
With the choice $R = (\frac{1}{2} - \delta) \log_d n$, so that $d^R = n^{\frac 12 - \delta}$ and $|E(\cT_R)|\le 2\Delta d^R$,~\eqref{eq:Chernoff-Poisson-binomial} implies that the right-hand side of~\eqref{eq:B-r-v-not-treelike} is at most $(C n^{-\delta})^\ell$ for some $C(\Delta)$ and large enough $n$. As a consequence, choosing $L > 2\delta^{-1}$, we would find 
\begin{align}\label{eq:CM-treelike}
    \sup_{v\in \{1,...,n\}} \Pcm (B_R(v) \mbox{ is not $L$-$\treelike$}) \le o(n^{-2})\,.
\end{align}
It remains to translate this to a bound under $\Prrg$. This follows by the following standard comparison argument. Let $\Gamma_{\textsc{rrg}}$ be the event that the graph $\cG\sim\Pcm$ has no self-loops or double edges (i.e., it is a simple graph). 
Taking $\Gamma= \{\cG: \cG \mbox{ is not $(L,R)$-$\treelike$}\}$ in~\eqref{eq:CM-to-RRG} and union bounding~\eqref{eq:CM-treelike} over the $n$ vertices yields the desired bound. 
\end{proof}

\subsubsection{Properties of the coupling of localized Markov chains}

The following lemma is the key fact about the construction of the
grand coupling of FK dynamics, Definition~\ref{def:coupling-Glauber-chains}, whereby after revealing some $X_{A,t}^{1}$, we can control
the influence that revealing has on $X_{B,t}^{1}$ for $A\cap B = \emptyset$. In this manner,
through the revealing procedure of Definition~\ref{def:revealing-FK-on-graph}, which reveals different localized configurations
$X_{A_m,t}^{1}$ iteratively, as long as $t\ge \Tburn$ these
are each close to their respective stationary distributions of $\pi_{A_m}^{1}$,
so that it is approximately a concatenation of localized FK models
on treelike graphs, inducing an exponential decay of connectivities.

\begin{lemma}\label{lem:revealing-process-law}
Recall the coupling of Definitions~\ref{def:coupling-Glauber-chains}--\ref{def:coupling-Glauber-chains-average-graph} of the distributions $(\mathbb{P}_{A,t}^{1})_{A\subset \mathfrak M_n,t\ge1}$. Suppose we have revealed edges in $E({\mathcal G})$ showing $E({\mathcal G})\cap A = A$. 
\begin{enumerate}
\item The configuration $X_{A,t}^{1}(A)$ is measurable with respect to $\kappa_{A,t}$ (the number of edge-updates in $A$), the edges chosen to update $\mathcal O_{A,\kappa_{A,t}}$, and the uniform random variables $(U_{e,s})_{e\in A,s\le t}$ on those edges.
The number $\kappa_{A,t}$ 
is distributed as $\bin(t,2|A|/\Delta n)$; the sequence $\cO_{A,\kappa_{A,t}}$ is distributed as $(\tilde{e}_{j})_{j\le \kappa_{A,t}}$ drawn i.i.d. from $A$. The values $(U_{e,s})_{e\in A,s\le t}$ are distributed as i.i.d.\ $\mbox{Unif}[0,1]$.   
\smallskip
\item Suppose $B$ is such that $E({\mathcal G})\cap B= B$ and $A\cap B=\emptyset$.  Conditionally
on $\kappa_{A,t}$, $\mathcal O_{A, \kappa_{A,t}}$, and $(U_{e,s})_{e\in A,s\le t}$, the
distribution of $X_{B,t}^{1}(B)$ is given as follows. The number of updates $\kappa_{B,t}$  
is drawn from $\bin(t-k_{A,t},2|B|/(\Delta n - 2|A|))$, and the edges chosen are distributed as $(\tilde e_{j})_{j\le \kappa_{B,t}}$ drawn i.i.d.\ amongst $B$. The random variables $(U_{e,s})_{e\in B, s\le t}$ are distributed as i.i.d.\ $\mbox{Unif}[0,1]$.   
\end{enumerate}
\end{lemma}

\begin{proof}
Let $\mathcal G$ be any graph having $E({\mathcal G}) \cap A = A$. We claim that uniformly over $\mathcal G$, items (1)--(2) above hold. Observe first that $|E({\mathcal G})|= \Delta n/2$ necessarily, and therefore uniformly over such $\mathcal G$, the number of updates on edges in $A$ by time $t$ in the update sequence $(e_s)_{s\le t}$ is distributed as $\bin(t, 2 |A|/(\Delta n))$. Evidently, the distribution of $\mathcal O_{A, \kappa_{A,t}}$ only depends on $\kappa_{A,t}$ and not on the times these updates were; in particular, given that $e_j \in A$ for some $j$, the law of $e_j$ is clearly uniform at random on $A$. Finally, notice that for every $e$, the sequence $(U_{e,s})_{s\le t}$ is independent of all other sources of randomness, implying the desired item (1). 

Turning to item (2), we fix a $\kappa_{A,t}, \mathcal O_{A,\kappa_{A,t}}$ and family $(U_{e,s})_{e\in A, s\le t}$. We can condition further on the exact times of the updates in $A$, i.e., $(e_s)_{s\le t} \cap A$. Conditionally on that set of updates, the distribution on the remaining updates is evidently $t-\kappa_{A,t}$ i.i.d.\ draws from $E({\mathcal G})\setminus A$. It is then clear that $\kappa_{B,t}$ counts the number of times, amongst these remaining draws, that the update is in $B$. As in item (1), the induced distribution on $\mathcal O_{B,\kappa_{B,t}}$ is then the same as $\kappa_{B,t}$ i.i.d.\ draws from the edges of $B$. Finally, for every $e\in B$, the uniform random variables $(U_{e,s})_{s\le t}$ are independent of all other sources of randomness. 
\end{proof}

\subsection{Domination by the modified branching process $(Z_k)$}\label{subsec:domination-by-the-branching-process}
In this section, we establish the stochastic domination of the sequence $(\cV_k)_{k\ge 0}$ from Definition~\ref{def:revealing-FK-on-graph} by the branching process $(Z_k)$ of Definition~\ref{def:branching-process}. 

\begin{proof}[\textbf{\emph{Proof of Lemma~\ref{lem:branching-process-domination}}}]
We prove the desired stochastic domination by induction over $\ell$. The base case, $Z_0 = |\cV_0|$, is by construction.  Now fix $\ell \ge 1$ and suppose by way of induction that the following stochastic domination holds: $$(|\cV_j|\mathbf 1_{\mathfrak E_t^{\mathfrak m_j}} \mathbf 1_{\{\mathfrak m_{j-1} \le n^{1/2 - \delta/2}\}})_{j\le \ell-1}\preccurlyeq (Z_j)_{j\le \ell-1}\,.$$
Thus there exists a monotone coupling of the sequence on the left-hand side, such that it is below the sequence $(Z_j)_{j\le \ell-1}$ in the natural element-wise ordering on the sequence.  Working on that coupling, it suffices for us to then show that on the intersection $\mathfrak E_t^{\mathfrak m_\ell} \cap \{\mathfrak m_{\ell-1}\le n^{1/2 - \delta/2}\}$, for every  $m\in \{\mathfrak m_{\ell-1}+1,...,\mathfrak m_{\ell}\}$, the distribution of the children of $v_m$ is stochastically below the progeny distribution of Definition~\ref{def:branching-process}. 

Observe, first of all, that for every $m\in \{\mathfrak m_{\ell-1} +1,...,\mathfrak m_{\ell}\}$, on $\mathfrak E_t^{\mathfrak m_\ell} \cap \{\mathfrak m_{\ell-1}\le n^{1/2 - \delta/2}\}$,  deterministically the number of children of $v_m$ is bounded by 
$$|V(\cA_m\setminus \cA_0)|\le |V(\cT_r)| \mathfrak m_{\ell-1} \le |V(\cT_r)| \sum_{j\le \ell-1} |\cV_{j}| \le |V(\cT_r)| \sum_{j\le \ell-1} Z_j\,,$$
where the last inequality is by the inductive hypothesis, and the fact that $\mathfrak E_t^{\mathfrak m_\ell}\cap \{\mathfrak m_{\ell-1}\leq n^{1/2 - \delta/2}\}$ implies $\mathfrak E_t^{\mathfrak m_j}\cap \{\mathfrak m_{j-1}\leq n^{1/2- \delta/2}\}$ for all $j<\ell$. 

Now, for every set of revealed edges $(A_l)_{l \le m-1}$, define the following events on $\cF_{m-1}$ consisting of $(\kappa_{A_l,t})_{l\le m-1}$, edge-values $(\mathcal O_{A_l,\kappa_{A_l,t}})_{l\le m-1}$, uniform random variables $((U_{e,s})_{e\in A_l,s\le t})_{l\le m-1}$: 
\begin{enumerate}
    \item Let $\Gamma_{\textsc{tree},m}$ be the event that $B_r^{out}(v_m)\cap \mathcal A_{m-1} = \emptyset$ and $A_m= \mathcal A_{m}\setminus \mathcal A_{m-1}$ is a tree. 
    \item Let $\Gamma_{\textsc{upd},m}$ be the event that $\kappa_{A_m,t} \ge d^r \Tburn/(2\Delta n)$
\end{enumerate}
We first claim that these two events each happen with probability $1-\frac 13 n^{-1/2}$, uniformly over $(A_l)_{l\le m-1}$ and elements of $\mathcal F_{m-1}$ such that $\mathfrak E_t^{\ell}$ holds and $\mathfrak m_{\ell-1}\le n^{1/2 - \delta/2}$. Given $v_m$, the law of $A_m$ is independent of $\mathcal F_{m-1}$, and only depends on $\mathcal A_{m-1}$. (This can be seen from the explicit construction of the law of $\cF_{m-1}$ in Lemma~\ref{lem:revealing-process-law} as independent of $E(\cG)\setminus \cA_{m-1}$.) Notice that if $\mathfrak m_{\ell-1}\le n^{1/2 - \delta/2}$ and $|\cA_0|\le n^{1/2 - \delta}$ then $|V(\cA_{m-1})|\le (1+|V(\cT_r)|)n^{1/2 - \delta/2}$. 
As such, by Lemma~\ref{lem:ball-intersection-probability}, for every $\cA_0,\cV_0$ such that $|\cA_0|\le n^{1/2 - \delta}$,  
\begin{align*}
    & \sup_{(\mathcal A_{m-1} ,\mathcal F_{m-1})\in \mathfrak E_t^{\mathfrak m_\ell}\cap \{\mathfrak m_{\ell-1}\le n^{\frac 12 - \frac \delta2}\}}  \Pcm (\Gamma_{\textsc{tree},m}^c\mid \mathcal A_{m-1},\mathcal F_{m-1}) \\
    & \!\!\qquad \qquad \le  \sup_{\substack{\mathcal A_{m-1}: \\ V(\mathcal A_{m-1})\le  2|V(\cT_r)| n^{\frac 12 - \frac \delta2}}} \sup_{v_m} \Pcm (B_{r}^{out}(v_m) \cap \mathcal A_{m-1}\ne \emptyset \mbox{ or } B_{r}^{out}(v_m) \mbox{ is not a tree} \mid \mathcal A_{m-1}) \\
    & \!\!\qquad \qquad \le  \sup_{\substack{\mathcal A_{m-1}: \\ V(\mathcal A_{m-1})\le 4\Delta d^r n^{\frac 12 - \frac \delta2}}} \frac{2\Delta d^{r}(|V(\mathcal A_{m-1})|+ d^r)}{n-(|V(\mathcal A_{m-1})|+d^{r})} \le \frac{10\Delta^2 d^{2r} n^{\frac 12 - \frac \delta 2}}{n- 5\Delta d^r n^{\frac 12 - \frac \delta 2}}\,.
\end{align*}
Thus, for $n$ large enough and $r=o(\log n)$, the above is at most $\frac 13 n^{-1/2}$ as desired.

We next turn to the probability of $\Gamma_{\textsc{upd},m}^c \cap \Gamma_{\textsc{tree},m}$. Recall from item (2) of  Lemma~\ref{lem:revealing-process-law} that conditionally on $\mathcal F_{m-1}$, the distribution of $\kappa_{A_m,t}$ is 
$$\kappa_{A_m,t}\sim \bin\Big(t- \sum_{l\le m-1} \kappa_{A_l,t}, \frac{2|A_m|}{\Delta n}\Big)\,.$$
Since we are on the event $\mathfrak E_t^{\mathfrak m_\ell}$ and thus $\mathfrak E_t^{m-1}$, we have that $\sum_{l \le m-1} \kappa_{A_l,t} \le 4 m |E(\cT_r)| t/(\Delta n)$, from which we deduce, using $m\le \mathfrak m_{\ell}\le |V(\cT_r)| \mathfrak m_{\ell-1} \le |V(\cT_r)| n^{1/2 - \delta/2}$, that the number of trials in the binomial is at least 
$$t(1- 16 \Delta d^{2r}   n^{-\frac 12  - \frac \delta 2}) \ge t/2\,,$$
as long as $r = o(\log n)$. 
Since we are on the event $\Gamma_{\textsc{tree},m}$, we have $d^{r} \le |A_m|\le |E(\cT_r)|\le 2\Delta d^{r}$, and we see from lower tail estimates on binomial random variables that 
\begin{align*}
    \sup_{(\mathcal A_{m-1},\mathcal F_{m-1})\in \mathfrak E_t^{m-1}\cap \{\mathfrak m_{\ell-1}\le n^{1/2 - \delta/2}\}} &  \mathbb P \big(\Gamma_{\textsc{upd},m}^c \cap \Gamma_{\textsc{tree},m}\mid \mathcal A_{m-1}, \mathcal F_{m-1}\big)   \\ 
    &  \!\!\!\!\!\!\!\!\!\!\!\!\!\!\!\le \mathbb P \big(\bin(\Tburn/2, 2d^{r}/(\Delta n))\le d^r  \Tburn/(2\Delta n)\big)  
    \le \frac 13 n^{-1/2}\,,
\end{align*} 
as long as $C_0$ in~\eqref{eq:T-burn} is sufficiently large (depending on $r,\Delta$).

By item (2) of Lemma~\ref{lem:revealing-process-law}, conditionally on \emph{any} $(\mathcal A_{l})_{l \le m-1}$ and $\mathcal F_{m-1}$, and any $A_m\in \Gamma_{\textsc{tree},m}$ and $\kappa_{A_m,t}\in \Gamma_{\textsc{upd},m}$, the conditional distribution of $X_{A_m,t}^1 (A_m)$ is equivalent (up to relabeling of edges) to that of $\kappa_{A_m,t}$ updates of a heat-bath chain $(Y_{s}^1)_s$ on a subtree $\hat \cT_r$ of the complete tree $\mathcal T_{r}$ with $(1,\circlearrowleft)$-wired boundary conditions, initialized from $Y_0^1 \equiv 1$. Notice that the equivalent sub-tree $\hat \cT_r$ consists of some $k\le d$ of the children of the root, together with their complete sub-trees. In particular, the random-cluster model on $A_m$ with wired boundary conditions is stochastically below the FK model on the corresponding subset of $\cT_r$ with its $(1,\circlearrowleft)$ boundary conditions. In particular, the number of leaves in the FK cluster of the root under $\pi_{\hat \cT_r}^{(1,\circlearrowleft)}$ is stochastically below the same quantity under $\pi_{\cT_r}^{(1,\circlearrowleft)}$. 
It therefore suffices for us to show that as long as $A_m$ is a tree disjoint from $\mathcal A_{m-1}$ and $\kappa_{A_m,t}\ge d^r \Tburn/(2 \Delta n)$,  we have
\begin{align*}
     \big\|\mathbb P \big(Y_{\kappa_{A_m,t}}^{(1,\circlearrowleft)} \in \cdot \big) - \pi_{\hat \cT_{r}}^{(1,\circlearrowleft)}\big\|_\tv \le \frac 13 n^{-1/2}\,.
\end{align*}
This follows as long as $C_0$ is sufficiently large (depending on $\Delta$), from the fact that
$$\frac{d^r \Tburn}{2\Delta n} \ge \frac{C_0 d^r}{2\Delta |E(\cT_r)|} \log n\cdot \tmix(\mathcal T_{r}, (1,\circlearrowleft)) \ge \frac{C_0 d^r}{2\Delta |E(\cT_r)|} \log n \cdot \tmix(\hat \cT_r, (1,\circlearrowleft))\,,$$
and $|E(\cT_r)| \le 2\Delta d^r$, 
together with the sub-multiplicativity of total-variation distance.
\end{proof}

\subsection{Sub-criticality and tail bounds for the dominating branching process}\label{subsec:analysis-of-branching-process}
We now analyze the process $(Z_k)$ of Definition~\ref{def:branching-process}, and show that it indeed is sub-critical, and satisfies good tails on its total population. For ease of notation, let $\cP_k= \sum_{\ell \le k} Z_\ell$ be the total population after $k$ generations.

\begin{proof}[\textbf{\emph{Proof of Lemma~\ref{lem:modified-branching-process-tail-bound}}}]
    Since $(Z_k)$ is a size-dependent, non-Markov process, we cannot directly use results on branching processes to control its growth. Instead, to control the population of the process $(Z_k)$, we compare it to a sum of branching processes in the following manner. Consider the stopping generation $\kappa_\lambda$ for exceeding population $K_0 Z_0 + \lambda$, i.e.,  
    $$\kappa_\lambda = \inf\{k: \cP_\kappa >K_0 Z_0+ \lambda\}\,.
    $$
    Our aim is to control the probability that $\kappa_\lambda<\infty$. 
    Let $\Gamma_{M,k}$ be the event that no more than $M$ of the progeny counts $((\chi_{i,\ell})_{i\le Z_\ell})_{\ell \le k-1}$ were $\mathsf {Bad}$. By~\eqref{eq:Chernoff-Poisson-binomial}, we get
    \begin{align*}
        \mathbb P (\Gamma_{M,\kappa}^c) \le \bbP\big(\bin(K_0 Z_0 + \lambda , n^{-\frac 12})>M\big) \le \exp\big( M- \mu - M\log \tfrac{M}{\mu}\big)
    \end{align*}
    where $\mu$ is the mean of the Binomial, i.e., $\mu = 
    (K_0 Z_0 + \lambda)n^{ - \frac 12}$. As 
    long as $n$ is sufficiently large and $Z_0, \lambda\le n^{\frac 12 - \delta}$ for $\delta>0$, so that $\mu \le 2 K_0 n^{ - \delta}$, this implies for some $C>0$, 
    \begin{align*}
         \mathbb P (\Gamma_{M,\kappa}^c) \le C e^{M} n^{ - \delta M}\,.
    \end{align*}

    Next consider the event that $\kappa<\infty$ on the event $\Gamma_{M,k}$.  On $\Gamma_{M,k}$, we dominate the population $\cP_k$ by the following sum of sub-critical branching processes with bounded progeny distributions. 

    Define $(\tilde Z_{k}^{(1)})_{k}$ to be the branching process initialized at $\tilde Z_0^{(1)} = Z_0$ with progeny $(\tilde \chi_{i,k}^{(1)})$, distributed i.i.d.\ from the distribution of the
number of leaves connected to the root, in a sample from $\pi_{\mathcal{T}_{r}}^{(1,\circlearrowleft)}$,
    i.e., the distribution of $(\chi_{i,k})$ conditionally on the progeny number not being $\mathsf{Bad}$. Let $\tilde \cP^{(1)}_k = \sum_{\ell \le k} \tilde Z_{k}^{(1)}$. For each $1\le j\le M$, iteratively let $\tilde Z_{k}^{(j)}$ be an independent branching process with the same progeny distribution, initialized from $\tilde Z_0^{(j)}= |V(\cT_r)| \tilde \cP^{(j-1)}_\infty$, where we recall $|V(\cT_r)|\le 2\Delta d^r$. 
    
    The following stochastic domination is clear by construction if we decompose the process $(Z_k)$ revealed in a breadth-first manner, into its excursions between the at most $M$ times (on the event $\Gamma_{M,k}$) when the progeny number $\chi_{i,k}$ was $\mathsf{Bad}$.
    \begin{claim}
    Fix any $k\ge 1$. We have the stochastic domination 
    \begin{align*}
        \cP_k \mathbf 1{\{\Gamma_{M,k}\}} \preccurlyeq \sum_{j\le M} \tilde \cP_\infty^{(j)}\,.
    \end{align*}
    \end{claim}
    
    With this domination in hand, notice that in order for $\kappa<\infty$ while $\Gamma_{M,\kappa}$ holds, there must exist some $k\le K_0 Z_0 + \lambda$ such that $\Gamma_{M,k}$ holds and $\cP_k\ge K_0 Z_0+ \lambda$. Therefore, by a union bound, 
\begin{align*}
\bbP \Big(\cP_\infty \ge K_0 Z_0 + \lambda , \Gamma_{M,\kappa}\Big) & \le \sum_{k\le K_0 Z_0+ \lambda} \mathbb P \Big(\cP_{k} \ge K_0 Z_0 + \lambda , \Gamma_{M,k} \Big) \\
& \le (K_0 Z_0+\lambda) \mathbb P \Big( \sum_{j\le M}\tilde \cP_\infty^{(j)} \ge K_0 Z_0+ \lambda\Big)\,.
\end{align*}
We claim that if $\sum_{j\le M}\tilde \cP_\infty^{(j)} \ge \! K_0 Z_0 + \lambda$ holds, there must exist $j\le M$ for which 
$\tilde Z_0^{(j)} \le  K_0 Z_0+ \lambda$,
and 
$$\tilde \cP_\infty^{(j)}\ge |V(\cT_r)|^{-1} \Big( \frac{K_0}{M}\Big)^{1/M}  (\tilde Z_0^{(j)} + K_0^{-1} M^{-2} \lambda) =: C_{r,\Delta,M} K_0^{1/M} (\tilde Z_0^{(j)} + K_0^{-1} M^{-2} \lambda)\,.$$
Indeed, if no such $j$ existed, as long as $K_0$ is sufficiently large, we could bound $\sum_{j\le M} \tilde \cP_\infty^{(j)}$ by 
\begin{align*}
\sum_{j\le M}  C_{r,\Delta,M} K_0^{1/M} (\tilde Z_0^{(j)} + \lambda) \le M \Big[\frac{K_0}{M} \tilde Z_0^{(1)} + K_0^{-1} M^{-2} \lambda (1 + \cdots + \frac{K_0}{M})\Big] \le K_0 Z_0 + \lambda\,.
\end{align*}

Now fix any $j\le M$, any $\tilde Z_0^{(j)}$ and consider the branching process $\tilde Z_{k}^{(j)}$. This is a branching process with progeny distribution having mean $\mathbf m = A (\hat p d)^r$ for some $A(p,q)$ per Lemma~\ref{lem:exp-decay-wired-tree}. Since $\hat p <d^{-1}$ when $p<p_u(q,\Delta)$, as long as $r$ is greater than some $r_0 (p,q,\Delta)$, for $n$ sufficiently large we have $\mathbf m <1$, and $\tilde Z_{k}^{(j)}$ is sub-critical.  
Additionally, the progeny distribution of $\tilde Z_k^{(j)}$ is almost surely bounded by $|\partial \cT_r|\le \Delta d^{r-1}$. As such, using the standard breadth-first exploration of the total population of the branching process $\tilde Z_k^{(j)}$ (through which $\tilde \cP_k^{(j)}$ is expressed as the random walk $\tilde Z_0^{(j)} + \sum_{\ell\le k}\sum_{i\le \tilde Z_k^{(j)}} (\tilde \chi_{i,k}^{(j)} - 1)$), we can bound 
\begin{align*}
\mathbb P \Big(\tilde \cP_\infty^{(j)} & \ge N^{(j)}_\lambda \Big) \le \mathbb P \Big( \sum_{i\le N^{(j)}_\lambda}  \tilde \chi_i > N^{(j)}_\lambda - \tilde Z_0^{(j)}\Big)\quad \mbox{ for }\quad N_\lambda^{(j)} : = C_{r,\Delta,M}K_0^{1/M} (\tilde Z_0^{(j)} + K_0^{-1}M^{-2}) \lambda\,,
\end{align*}
where $\tilde \chi_i$ are i.i.d.\ copies of $\tilde \chi_{i,k}^{(j)}$. 
Now observe that if $K_0 (p,q,\Delta,M)$ is sufficiently large, the right-hand in the probability above exceeds the mean $\mathbf m N^{(j)}_\lambda$ by some $c N^{(j)}_\lambda$ for $c  = c(p,q,\Delta,M,K_0)>0$.  
As this is a tail probability for a sum of i.i.d.\ random variables, by Hoeffding's inequality, it is at most 
\begin{align*}
\exp \big( - (cN^{(j)}_\lambda)^2 / (4N^{(j)}_\lambda d^{2r})\big) \le \exp ( - c' N_\lambda^{(j)})\,,
\end{align*} 
for some $c'(r,\Delta,M)>0$. Taking a union bound over the $M$ possible values of $j\le M$, we altogether find  
\begin{align*}
\sum_{k\le K_0 Z_0+ \lambda} \mathbb P \Big(\cP_{k} \ge K_0 Z_0+ \lambda , \Gamma_{M,k} \Big) \le (K_0 Z_0 + \lambda) M \exp \big( - c' N_{\lambda}^{(j)}\big)\,. 
\end{align*}
 It follows from this, and the definition of $N_{\lambda}^{(j)}$, that for some $C(p,q,\Delta,M,K_0)$ large enough,  
 \begin{align*}
 \mathbb P ( \kappa<\infty) \le \mathbb P (\Gamma_{M,\kappa}^c) + \sum_{k\le K_0 Z_0+ \lambda} \mathbb P \Big(\cP_{k} \ge K_0 Z_0+ \lambda, \Gamma_{M,k}\Big) \le C n^{- \delta M} + C\exp \big( - (Z_0 + \lambda)/C\big)\,,
 \end{align*}
 concluding the proof. 
\end{proof}

\subsection{Proof of exponential tail on cluster sizes and shattering}\label{subsec:main-revealing-proof}
We are now in position to conclude the proof of the exponential tail bound on clusters of $X_{\cG,t}^1$, and use that to deduce that $X_{\cG,t}^1$ is $(K,R)$-$\sparse$, except with probability $o(n^{-2})$.  We begin by using Lemmas~\ref{lem:branching-process-domination}--\ref{lem:modified-branching-process-tail-bound} to prove the following tail bound on the sequence $(\cV_k)$, which are the roots of the balls revealed through the revealing process of Definition~\ref{def:revealing-FK-on-graph}. 

\begin{lemma}\label{lem:revealing-procedure-tail-bounds}
Fix $\delta>0$ and consider the revealing procedure for any initial subsets $\cA_0$ and $\cV_0$ having $|\cA_0|,|\cV_0|\le n^{\frac 12 - \delta}$. For every $M\ge 1$, there exist $C (p,q,\Delta,M)$, $K_0(p,q,\Delta,M)$ and $C_0 (p,q,\Delta,\delta,M)$ $r (p,q,\Delta)$ in the definition of~$\Tburn$ in~\eqref{eq:T-burn} such that for all $t\ge \Tburn$ and all $0 \le \lambda \le n^{ \frac 12 -\delta}$,  
\begin{align*}
\mathbb P\Big(\sum_{0 \le k\le k_\emptyset} |\mathcal{V}_{k}|\ge K_0 |\mathcal{V}_{0}| + \lambda \Big) \le C\exp ( - \lambda/C) + C n^{ - \delta M}\,,
\end{align*}
\end{lemma}

\begin{proof}
Fix $K_0$ large to be chosen later, and define the following stopping generation  
\[
\varsigma= \inf\{\ell: \mathfrak m_{\ell - 1} > K_0 |\mathcal V_0| + \lambda \}\,.
\]
Recall $\mathfrak E_t$ from~\eqref{eq:mathfrak-E}. Since for every $\ell \le \varsigma$, we have from Lemma~\ref{lem:branching-process-domination}, that $(|\mathcal V_{\ell}| \mathbf 1{\{\mathfrak E_t\}})_{j\le \ell} \preccurlyeq (Z_{j})_{j\le \ell}$, we have that if $C_0$ in~\eqref{eq:T-burn} is sufficiently large, the probability of  $\{\varsigma <\infty\}$ is bounded by the probability of $\cP_\infty = \sum_{k\ge 0} Z_k\ge K_0 Z_0 + \lambda$. By Lemma~\ref{lem:branching-process-domination}, we obtain 
\begin{align*}
    \mathbb P \Big( \sum_{k\le k_\emptyset} |\mathcal V_k | \ge K_0  |\mathcal V_0| + \lambda\Big) \le \mathbb P \Big( \sum_{k\ge 0} Z_k \ge K_0 Z_0 + \lambda  \Big) + \mathbb P (\mathfrak E_t^c)\,.
\end{align*}
Lemma~\ref{lem:modified-branching-process-tail-bound} implies the existence of $r(p,q,\Delta)$ such that the first-term above is at most $C_1\exp(- \lambda/C_1) + C_1n^{-\delta M}$ for some $C_1(p,q,\Delta,M)$.  

Next, consider $\mathbb P (\mathfrak E_t^c)$. By a union bound and item (1) of Lemma~\ref{lem:revealing-process-law}, with the trivial observations that $\mathfrak m_{k_\emptyset}\le n$ and $|A_m|\le |E(\cT_r)| \le 2\Delta d^r$ necessarily, we get for every $t\ge \Tburn$,
\begin{align*}
    \mathbb P ( \mathfrak E_t^c ) \le n  \mathbb P \Big(\bin\big(t, \frac{2|E(\cT_r)|}{\Delta n}\big) > \frac{4|E(\cT_r)|}{\Delta n} t \Big)\,.
\end{align*}
The above entails a deviation of at least $4 t d^r n^{-1}$ from its mean; as such, by standard tail estimates for binomials, for every $t\ge \Tburn$, 
\begin{align}\label{eq:E-t-complement}
    \mathbb P (\mathfrak E_t^c)\le n \exp ( - t d^r n^{-1})\,,
\end{align}
which is at most $n^{ - \delta M}$ for $n$ large, as long as $C_0$ in~\eqref{eq:T-burn} is sufficiently large (depending on $\delta M$). The desired bound then follows up to a change of the constant $C$. 
\end{proof}

Before proceeding to prove Proposition~\ref{prop:nontrivial-bdy-component-bound}, let us translate the tail bound of Lemma~\ref{lem:revealing-procedure-tail-bounds} on $\sum_k |\cV_k|$ to a tail bound on the FK cluster of a single vertex under $X_{\cG,t}^1$ and $\pi_{\cG}$. Notice that towards the proofs of Theorem~\ref{thm:pi-exponential-decay} and Proposition~\ref{prop:nontrivial-bdy-component-bound}, it suffices to show these for $t\ge \Tburn$ for some fixed choices of $C_0, r$ in~\eqref{eq:T-burn} depending on $p,q,\Delta$ (as $\tmix (\cT_r,(1,\circlearrowleft))$ is of course independent of $n$).

\begin{proof}[\textbf{\emph{Proof of Theorem~\ref{thm:pi-exponential-decay}}}]
    Fix any $v\in \{1,...,n\}$, let $\cA_0 = \emptyset$ and let $\cV_0=\{v\}$ in Definition~\ref{def:revealing-FK-on-graph}. By Observation~\ref{obs:key-observation}, for each $\cG\sim  \Pcm$, the cluster of $v$ in the configuration $X_{\cG,t}^1$, denoted  $\cC_{v}(X_{\cG,t}^1)$ is a subset of $\cC_v(\tilde \omega_t)$, which in turn is a subset of $V(\cA_{\mathfrak m_{\emptyset}})$, so that 
    \begin{align*}
    |\cC_{v}(X_{\cG,t}^1)|\le |\cC_v(\tilde \omega_t)| \le |V(\cT_r)| \sum_{k\le k_\emptyset} |\cV_k|\le 2\Delta d^r \sum_{k\le k_\emptyset} |\cV_k|\,.
    \end{align*}
    By Lemma~\ref{lem:revealing-procedure-tail-bounds} and the above, we find that for each $M$, there exists $C(p,q,\Delta,M)$ such that
    \begin{align*}
        \mathbb P \big((\cG,X_{\cG,t}^1): |\cC_v(X_{\cG,t}^1)|\ge 2\Delta d^r (1+ \lambda)\big) \le C\exp ( - \lambda/C) + Cn^{ - \delta M}\,.
    \end{align*}
    Observing that $\mathbb P ( X_{\cG,t}^1\in \cdot ) = \Ecm[P(X_{\cG,t}^1\in \cdot)]$, we can use Markov's inequality to write 
    \begin{align*}
        \Pcm\bigg( \cG: P\Big(X_{\cG,t}^1: |\cC_v(X_{\cG,t}^1)| \ge 2\Delta d^r (1+\lambda)\Big)\ge & \sqrt{C e^{ - \lambda/C}   + Cn^{ - \delta M}}\bigg) \\
        & \qquad \qquad \le \sqrt{C e^{ - \lambda/C} + Cn^{ - \delta M}}\,.
    \end{align*}
    We can obtain the same bound for $\Prrg$ by~\eqref{eq:CM-to-RRG}, up to a multiplicative $c(\Delta)^{-1}$ on the right-hand side. 
    Taking $M$ such that $\delta M> 2K$ and using the fact that $\sqrt{a+b} \le \sqrt a+ \sqrt b$ for all $a,b\ge 0$, we deduce the desired tail bound on $|\cC_v(X_{\cG,t}^1)|$ up to the change of constant $C$ to $2C$. Using the monotonicity  $X_{\cG,t}^1 \succcurlyeq \pi_\cG^1$ implies the analogous bound for $|\cC_v(\omega)|$ where $\omega\sim \pi_\cG$. 
\end{proof}

We now turn to proving that for typical random graphs, the configuration $X_{\cG,t}^1$ is $(K,R)$-$\sparse$ with high probability for all $t \ge \Tburn$. This allows us to localize to treelike balls with sparse boundary conditions. Let us define the following subset of the boundary of a set $H$, which we will apply with the choice $H = B_R(v)$. 

\begin{definition}\label{def:mathfrak-V}
    For a subgraph $H = (V(H),E(H))$ of $\cG$ and a configuration
$\omega$ on $E(\cG)$, let us define $\mathfrak{V}_{H}(\omega)$ as the subset of vertices in $V(H)$ in non-trivial components in the boundary condition induced on $H$
by $\omega(E(\cG)\setminus E(H))$ (a connected component is non-trivial when it has at least two vertices).
\end{definition}


We first prove the following proposition, giving a tail bound on $\mathfrak V_{B_R(v)}(X_{\cG,t}^1)$; after proving this proposition, we straightforwardly use it to conclude $(K,R)$-sparsity of $X_{\cG,t}^1$, i.e., Theorem~\ref{thm:k-R-sparse-whp}. 
 
\begin{proposition}\label{prop:nontrivial-bdy-component-bound}
Let $p,q,\Delta$ be such that
$p<p_{u}(q,\Delta)$. Fix $\delta>0$ and let $R=(\frac{1}{2}-\delta)\log_{d}n$.
There exists $K(p,q,\Delta,\delta)$ such that if $\mathcal{G} \sim \Pcm$, with probability
$1-O(n^{-2})$, $\mathcal{G}$ is such that for all $t\ge \Tburn$
\begin{align*}
\sup_{v\in V(\cG)} P\big(X_{\cG,t}^{1}:|\mathfrak{V}_{B_R(v)}(X_{\cG,t}^{1})|>K\big) & \le O(n^{-2})\,.
\end{align*}
\end{proposition}

\begin{proof}[\textbf{\emph{Proof of Proposition~\ref{prop:nontrivial-bdy-component-bound}}}]
Fix $v\in \{1,...,n\}$ and $\delta>0$, and let $R = (\frac{1}{2} - \delta)\log_d n$. 

We apply the revealing procedure of Definition~\ref{def:revealing-FK-on-graph} with the choices $\cA_0= E(B_R(v))$ and $\cV_0 = \partial B_R(v)$. Recall from Observation~\ref{obs:key-observation} that the FK-clusters of $\cV_0$ induced by  $\tilde \omega_t(E(\cG)\setminus \cA_0)$ ($\tilde \omega_t$ was extended to be all wired off of $\cA_{\mathfrak m_{k_\emptyset}} \setminus \cA_0$) are confined to the set $\cA_{\mathfrak m_{k_\emptyset}}\setminus \cA_0$, and the extended configuration $\tilde \omega_t$ satisfies $\tilde \omega_t \ge X_{\cG,t}^1$. Thus, the sets $\mathfrak V_{B_R(v)}(\tilde \omega_t)$ and in turn $\mathfrak V_{B_R(v)}(X_{\cG,t}^1)$, are below the number of vertices in $\cV_0$ that share a connected component of $\cA_{\mathfrak m_{k_\emptyset}}\setminus\cA_0$ with another vertex of $\cV_0$.

Suppose that through the revealing process of Definition~\ref{def:revealing-FK-on-graph}, for each $m$, the edges of $B_r^{out}(v_m)$ are revealed one at a time per Definition~\ref{def:revealing-graph}. Notice then, that $|\mathfrak  V_{B_R(v)}(\mathcal A_{\mathfrak m_{k_\emptyset}}\setminus \cA_0)|$ is bounded by the number of times through the revealing of $\mathcal A_{\mathfrak m_{k_\emptyset}}$, that a half-edge is matched up to a half-edge belonging to a vertex that has been discovered at that point. Throughout this process, conditionally on an exposed edge-set $\mathcal A$ (and the edge-update sequence, and uniform random variables given by the filtration up to that step of the revealing, but $E(\cG)\setminus \cA$ is independent of these),  the law of the next half-edge to be matched is uniform amongst un-matched half-edges. Thus on any such edge-revealing, uniformly on the history of the revealing, the probability that it matches with a half-edge belonging to a discovered vertex is at most $\frac{|V(\cA_{\mathfrak m_{k_\emptyset}})|}{n- 2|V(\cA_{\mathfrak m_{k_\emptyset}})|}$.

By a union bound, we obtain for $\Lambda$ a sufficiently large constant (depending on $p,q,\Delta,r$), for all $k\ge 1$, 
\begin{align*}
    \mathbb P \Big((\mathcal G, \tilde \omega_t): |\mathfrak V_{v,R}(X_{\cG,t}^1)| > k\Big) & \le \mathbb P \Big( |V(\mathcal A_{\mathfrak m_{k_\emptyset}})|> \Lambda^2 |\mathcal V_0| \Big) + \mathbb P \Big(\bin\Big(\Lambda^2  |\mathcal V_0|, \frac{2 \Lambda^2 |\mathcal V_0|}{n}\Big)>k\Big) \\ 
    & \le \mathbb P \Big( \sum_{k\le k_\emptyset} |\cV_k| \ge \Lambda |\cV_0|\Big) +  \mathbb P \Big(\bin\Big(\Lambda^2  |\mathcal V_0|, \frac{2 \Lambda^2 |\mathcal V_0|}{n}\Big)>k\Big)\,.
\end{align*}
By Lemma~\ref{lem:revealing-procedure-tail-bounds} and the fact that $\Lambda |\cV_0| \le n^{\frac 12 - \frac \delta 2}$ for $n$ large, as long as $\Lambda$ is large enough, the first term is at most $n^{ - 5}$. 
Using the fact that $|\cV_0|\le n^{\frac 12 - \delta}$, we see that the mean of the binomial is at most $n^{-3\delta/2}$, so that by~\eqref{eq:Chernoff-Poisson-binomial}, for every fixed $k\ge 1$, 
\begin{align}\label{eq:non-trivial-bc-tail-X-t}
 \mathbb P \Big((\mathcal G, X_{\cG,t}^1): |\mathfrak V_{v,R}(X_{\cG,t}^1)| > k\Big) \le n^{ - \delta k \wedge 4}\,,
\end{align}
for $n$ large enough. Choosing $k = K$ sufficiently large (depending on $\delta$), we can make the right-hand side at most $n^{-4}$. 
   We deduce the proposition by using Markov's inequality to write 
   $$\Pcm \Big(\cG: P(X_{\cG,t}^1: |\mathfrak V_{B_R(v)} (X_{\cG,t}^1)|\ge K)>n^{ - 2}\Big)\le n^2 \Ecm [P(X_{\cG,t}^1: |\mathfrak V_{B_R(v)} (X_{\cG,t}^1)|\ge K)]\,,$$
   and noticing that the expectation on the right equals the probability bounded in~\eqref{eq:non-trivial-bc-tail-X-t}.
\end{proof}


\begin{proof}[\textbf{\emph{Proof of Theorem~\ref{thm:k-R-sparse-whp}}}]
First of all, a union bound of Proposition~\ref{prop:nontrivial-bdy-component-bound} over $v\in \{1,...,n\}$, with $\Pcm$-probability
$1-O(n^{-1})$, $\cG$ is such that 
\begin{align*}
P\Big(X_{\cG,T}^{1}:\bigcup_{v\in V(\cG)}\big\{|\mathfrak{V}_{B_R(v)}(X_{\cG,t}^{1})|>K\big\}\Big)& \le  n^{-1}\,.
\end{align*}
 We now translate this to a bound under $\Prrg$. Taking $$\Gamma=\Big\{\mathcal{G}:P\Big(X_{\cG,t}^{1}:\bigcup_{v\in V(\cG)}\{|\mathfrak{V}_{B_R(v)}(X_{\cG,t}^{1})|>K\}\Big)> n^{-1}\Big\}\,,$$
in~\eqref{eq:CM-to-RRG}, we deduce that
$\Prrg(\Gamma)\le c^{-1}\Pcm(\Gamma)\le  c^{-1}n^{-1}$ 
for some $c(\Delta)>0$, as needed. 
\end{proof}

\section{Sharp rates of correlation decay in trees and treelike graphs}\label{sec:correlation-decay-treelike}
In this section we establish the precise exponential decay rate of influence from an $O(1)$-$\sparse$ boundary condition on the root of an $O(1)$-$\treelike$ ball. We recall from Section~\ref{sec:proof-strategy}, that getting the right decay rate, (as opposed to e.g., using the decay rate of connectivity from the root to the boundary) is central to pushing our argument through for all $p<p_u$. In particular, the decay rate of influence will be inherited from \emph{twice} the exponential decay rate of the wired tree. 

Recall that the uniqueness point $p_u(q,\Delta)$ is defined by a transition on the wired $\Delta$-regular tree, where the measure $\pi^1_{\T_h}$ transitions between exponentially small (in $h$) probability of a root-to-leaf connection, to giving this event uniformly positive probability. A recursion for this connectivity probability was calculated in~\cite[Lemma 33]{BGGSVY}. A careful examination of this recursion will yield the following identification of the rate of the exponential decay with $\hat p$ of~\eqref{eq:hat-p-inequality}. 

\begin{lemma}\label{lem:exp-decay-wired-tree}
        Let $p<p_u(q,\Delta)$. There exists $C(p,q,\Delta)$ such that for every $h$ and every leaf $u\in \partial \cT_h$, 
        \begin{align*}
            \pi_{\cT_h}^{(1,\circlearrowleft)} (\omega: u\in \cC_\rho(\omega) ) \le C \hat p^h\,.
        \end{align*}
        In particular, the probability of the root being connected to $\partial \cT_h$ in $\omega$ is  at most $C(\hat p d)^h$.  
\end{lemma}
%

In Section~\ref{subsec:wired-tree-exp-decay}, we establish Lemma~\ref{lem:exp-decay-wired-tree}. 
In Section~\ref{subsec:spatial-mixing-treelike}, we show that influence in the random-cluster model travels through the existence of \emph{two} distinct connections; thus on $\treelike$ graphs, influence has twice the exponential decay rate of root-to-leaf connectivities on the wired tree. This will yield Proposition~\ref{prop:influence-probability-new}.

\subsection{Exponential decay rate in the wired {$\Delta$}-regular tree}\label{subsec:wired-tree-exp-decay}
Because of its recursive structure, connectivity properties of the random-cluster measure on the wired tree can be analyzed sharply. In this section, we pursue this and show that in the uniqueness regime of $p<p_u$, the probability of a connection from the root to a leaf at depth $h$ is $O(\hat p^h)$, as one would have for the free tree (corresponding to i.i.d.\ $\ber(\hat p)$ percolation on $\cT_h$). 
We first show that the probability of a root-to-boundary connection decays exponentially in $h$. 

Let $\T_h$ be the complete $\Delta$-regular tree of height $h$ rooted at $\rho$.  The wired ``$1$" boundary conditions on $\cT_h$ are those that wire all leaves of $\cT_h$ (all vertices in $\partial \cT_h$). Define the probability 
$$
\varphi_h := \pi_{\cT_h}^1 (\omega: \cC_\rho(\omega) \cap \partial \cT_h \neq \emptyset)\,,
$$
that the root is connected to a leaf of $\cT_h$. Using the recursive structure of the tree, it was shown in~\cite[Lemma 33]{BGGSVY} that if we define $\mu : = \frac pq + 1-p$, for every $h$, we have
\begin{align}\label{eq:connectivity-recursion}
    \varphi_{h+1} = f(\varphi_h)\,, \qquad \mbox{where}\qquad 	f(x) = \frac{\big(\mu+p(1-\tfrac{1}{q})x\big)^{d}-\big(\mu-\tfrac{p}{q}x\big)^{d}}{\big(\mu+p(1-\tfrac{1}{q})x\big)^{d}+(q-1)\big(\mu-\tfrac{p}{q}x\big)^{d}}\,,
\end{align}
and for every $p<p_u(q,\Delta)$, this satisfies $\lim_{h\to\infty} \varphi_h = 0$. The following lemma establishes that this convergence is exponentially fast.

\begin{lemma}
	\label{lem:exp-decay}
	Let $p<p_u(q,\Delta)$. We have $\lim_{h\to\infty} \frac{\varphi_{h+1}}{\varphi_{h}} = \hat p d$.
	Moreover, $\varphi_{h} \le  (\hat p d)^{h+o(h)}$.
\end{lemma}

\begin{proof}
    Consider the recursion of~\eqref{eq:connectivity-recursion} for $\varphi_h$. Since $\lim_{h\to\infty} \varphi_h =0$, if $\lim_{x\to 0} \frac{f(x)}{x}$ exists, we would have 
	\begin{align}
	\lim_{h\to\infty}\frac{\varphi_{h+1}}{\varphi_{h}}= & \lim_{x\rightarrow0}\frac{f(x)}{x} = \lim_{x\rightarrow0} \frac{\big(\mu+p(1-\tfrac{1}{q})x\big)^{d}-\big(\mu-\tfrac{p}{q}x\big)^{d}}{x\big(\mu+p(1-\tfrac{1}{q})x\big)^{d}+x(q-1)\big(\mu-\tfrac{p}{q}x\big)^{d}} \label{eq:lim}\,.
	\end{align}
	Since both the numerator and denominator of \eqref{eq:lim} are differentiable and have limit $0$ as $x\to0$, using L'H\^opital's rule we get
	\begin{align*}
	\lim_{x\rightarrow0}\frac{f(x)}{x}&=\lim_{x\rightarrow0}\frac{\partial_{x}\Big[\big(\mu+p(1-\tfrac{1}{q})x\big)^{d}-\big(\mu-\tfrac{p}{q}x\big)^{d}\Big]}{\partial_{x}\Big[x\big(\mu+p(1-\tfrac{1}{q})x\big)^{d}+x(q-1)\big(\mu-\tfrac{p}{q}x\big)^{d}\Big]} \\ &=\frac{dp(1-\tfrac{1}{q})\mu^{d-1}+d\tfrac{p}{q}\mu^{d-1}}{\mu^{d}+(q-1)\mu^{d}}
	=\frac{dp}{q\mu} = \frac{dp}{p+q(1-p)} = d\ps\,.
	\end{align*}
	Recall that for every $0<p<p_u$,  we have $0<d\hat p<1$.	
	Thus, there exists a sequence $\{\epsilon_{h}\}$ such that $\lim_{h\to\infty}\varepsilon_{h}=0$ and for every $h$, 
	\begin{align*}
	\varphi_{h}= \varphi_1\cdot \frac{\varphi_{h}}{\varphi_{h-1}} \dots \frac{\varphi_{2}}{\varphi_{1}} = \varphi_{1} \cdot \prod_{i=2}^h (\hat p d+\varepsilon_{i})\,.
	\end{align*}
	Expanding this out, we deduce the desired 
\[
	\varphi_{h} = \varphi_{1} (\hat p d)^h \exp\left( \sum_{i=2}^h \ln \left(1+\frac{\varepsilon_i}{\hat p d}\right)\right) \le \varphi_{1} (\hat p d)^h \exp\left((\hat p d)^{-1} \sum_{i=2}^h \varepsilon_i\right) =  (\hat p d)^{h+o(h)}\,. \qedhere
\]
\end{proof}

Our aim is to now prove Lemma~\ref{lem:exp-decay-wired-tree}, bounding connectivities of the root to a single leaf. 
\begin{proof}[\textbf{\emph{Proof of Lemma~\ref{lem:exp-decay-wired-tree}}}]
To prove Lemma~\ref{lem:exp-decay-wired-tree}, we write a recursion for the root-to-leaf connection probability.
Let $\vartheta_{h}$ be the probability under $\pi_{\cT_h}^1$ that the root is connected to the left-most leaf of depth $h$.
Let $\vartheta_{h}^\circlearrowleft$ be the probability of the same event, under $\pi_{\cT_h}^{(1,\circlearrowleft)}$ where we recall that the $(1,\circlearrowleft)$ boundary conditions additionally wire the leaves of $\cT_h$ to the root. By monotonicity we have $\vartheta_h \le \vartheta_h^{\circlearrowleft}$ and by Lemma~\ref{lemma:simple-rc-bound}, we have $\vartheta_{h}^\circlearrowleft \le q^2 \vartheta_{h}$. 

	Let $(I_i)_{i\le \Delta}$ be the indicator function of the event that there is a root-to-boundary path going through the $i$-th
	child of the root; set $I = \sum_{i=2}^\Delta I_i$.
	Then, we can write
	\begin{align*}
	\vartheta_{h} \leq p  \cdot \pi_{\cT_{h}}^1(I \geq 1) \cdot \vartheta_{h-1}^\circlearrowleft+\ps \vartheta_{h-1} \le \vartheta_{h-1} \left[pq^2 \cdot \pi_{\cT_{h}}^1(I \geq1)+ \ps\right]\,,
	\end{align*}
	where in the first inequality we used the fact that in order for the root to be connected to the left-most leaf,
	it is required that the root 
	is connected to its left-most child $w_1$,
	and that  $w_1$ is connected to the left-most leaf of its sub-tree. 
	The former event occurs with probability $p$ or $\ps$,
	depending on whether or not the root is connected to $\partial \cT_h$ through any child besides $w_1$.
	
    By monotonicity, for every $i=2,...,\Delta$, the law of $I_i$ under $\pi_{\cT_h}^{1}$ is below its law under $\pi_{\cT_h}^{(1,\circlearrowleft)}$ and the same holds for  $I$. Since, by Lemma~\ref{lemma:simple-rc-bound} a single external wiring may distort the distribution by at most a $q^2$ factor, we get 
	$
	\pi_{\cT_h}^{(1,\circlearrowleft)}(I_i = 1) \le p q^2  \varphi_{h}
	$ for all $i$. 
	Hence, under $\pi_{\cT_h}^{(1,\circlearrowleft)}$, $I$  is stochastically below $Q$, where $Q\sim \bin(d, p q^2\varphi_h)$. 
	A union bound and Lemma~\ref{lem:exp-decay} then imply
	$$
	\pi_{\cT_h}^{(1,\circlearrowleft)} (I \ge 1) \le \mathbb P (Q \ge 1) \le d p q^2  (\hat p d)^{h-o(h)} \le C (\hat p d)^{(1-\varepsilon)h}\,,
	$$
	for all $h$; note that $\varepsilon$ can be chosen as small as needed provided the constant $C(p,q,\Delta,\varepsilon)$ is large enough. 
	Thus, setting $a = Cpq^2 \ps^{-1}$
	we
	obtain 
	\begin{align*}
	\vartheta_{h}& \leq \ps \vartheta_{h-1} \left[1 + a(\ps d)^{(1-\varepsilon)h}\right] \le \ps^{h} \prod_{i=1}^{h} \left[1 + a(\ps d)^{(1-\varepsilon)i}\right]\,,
	\end{align*}
	by continuing the recursion.
	Now, observe that since $\ps d<1$ when $p<p_u$,
	\begin{align*}
	\prod_{i=1}^{h} \left[1 + a (\hat pd)^{(1-\varepsilon)i}\right]
	&=  \exp\left[ \sum_{i=1}^h \log \left(1 + a (\hat p d)^{(1-\varepsilon)i} \right)\right] \le  \exp\left[ a \sum_{i=1}^h (\hat pd)^{(1-\varepsilon)i}\right] \le  \exp\left[\frac{a}{(\ps d)^{1-\varepsilon}}\right].\end{align*}
	Combining the above two bounds, there exists an absolute constant $A = A(p,q,\Delta)$ such that for every
	$h$ we have $\vartheta_h \leq A \ps^{h}$ and thus $\vartheta_h^{\circlearrowleft} \le A q^2 \ps^h$. The first inequality in the lemma follows by noticing that all the leaves in $\cT_h$ are equivalent, and the second follows from a union bound over the $\Delta d^{h-1}$. 
\end{proof}

\subsection{Exponential decay rate in $(L,R)$-$\treelike$ graphs}\label{subsec:spatial-mixing-treelike}
Let $G=(V,E)$ be an $(L,R)$-$\treelike$ graph.
For $v \in V$, let $B := B_R(v)$ denote the ball of radius $R$ around the vertex $v$.
Recall that we use $\mathcal N_v \subseteq E$ for the set of edges incident to $v$. 
For each $1\le \ell \le R$, let $Q_\ell = \{u \in B: d(u,v) \ge \ell\}$.

For a boundary condition $\xi$ on $\partial B$, recall the set $\mathfrak V_{B,\xi}$ of vertices in non-trivial boundary components of $\xi$ from Definition~\ref{def:mathfrak-V}. For any $u \in B$ such that $d(u,v) = \ell$,
let $u \stackrel{Q_\ell}\longleftrightarrow \mathfrak V_{B,\xi}$ denote the event that $u$
 is connected to $\mathfrak V_{B,\xi}$ by a path of open edges fully contained in $Q_\ell$: i.e., 
 \[
 \{u\stackrel{Q_\ell}\longleftrightarrow \mathfrak V_{B,\xi}\} := \{\omega: \cC_u (\omega(Q_\ell)) \cap \mathfrak V_{B,\xi} \neq \emptyset\}\,.
 \]
Define the event 
\[
\Upsilon_{B,\xi} : = \{\omega \in \{0,1\}^{E(B)}: |\{u \in B:d(u,v)= \ell\,,\, u \stackrel{Q_\ell}\longleftrightarrow \mathfrak V_{B,\xi}\}| \ge 2 \mbox{ for all $1 \le \ell \le R$}\}\,.
\]
Notice that $\Upsilon_{B,\xi}$ is an increasing event. We claim that $\Upsilon_{B,\xi}$ controls the propagation of influence from $\partial B$. 

\begin{lemma}\label{lem:influence-event}
	Fix a graph $G = (V,E)$, a vertex $v\in V$ and consider the ball $B_R(v)$; let $\xi \ge \tau$ denote two boundary conditions on $\partial B_R(v) = \{w\in B_R(v): d(v,w) = R\}$. Then, 
	\begin{align*}
	\|\pi_{B_R(v)}^\xi (\omega(\mathcal N_v)\in \cdot)- \pi_{B_R(v)}^\tau(\omega(\mathcal N_v)\in \cdot ) \|_\tv \le \pi_{B_R(v)}^\xi(\Upsilon_{B_R(v),\xi})\,.
	\end{align*}
\end{lemma}

\begin{proof}[{\textbf{\emph{Proof of Lemma~\ref{lem:influence-event}}}}]
	For ease of notation let $B := B_R(v)$, $\Upsilon_\xi = \Upsilon_{B,\xi}$ and $\mathfrak V_\xi = \mathfrak V_{B,\xi}$. 
	We construct a monotone coupling $\P$ of $\omega^\xi\sim\pi_{B}^\xi$ and $\omega^\tau\sim\pi_B^\tau$.
	The coupling $\P$ reveals the configurations $\omega^\xi \sim \pi_{B}^\xi$ and $\omega^\tau \sim \pi_B^\tau$
	on $B$ one edge at a time using i.i.d.\ uniform random variables $U_e \in [0,1]$
	for each  $e \in E(B)$. The same $U_e$ is used to reveal the values $\omega^\xi(e)$ and $\omega^\tau(e)$ from the corresponding conditional measures. 
	The order in which the uniform variables are revealed is irrelevant and can be adaptive; this will allow us to reveal the boundary components.
	(For more details on the process of revealing random-cluster components under the monotone coupling, see below, as well as e.g.,~\cite{BS,BGVfull}.) 
	
	We construct an adaptive revealing scheme that ensures that on the event $\Upsilon_\xi^c$ for the top sample $\omega^\xi$, the samples $\omega^\xi$ and $\omega^\tau$ agree on $\mathcal N_v$. This implies the desired result as one would then have by the definition of total-variation distance, 
	$$
		\|\pi_{B}^\xi (\omega(\mathcal N_v)\in \cdot)- \pi_{B}^\tau(\omega(\mathcal N_v)\in \cdot ) \|_\tv 
		\le \P(\omega^\xi(\mathcal N_v) \neq \omega^\tau(\mathcal N_v)) \le \pi_{B}^\xi(\Upsilon_{B,\xi})\,.
	$$
	We construct $\P$ with the following iterative scheme which proceeds level-by-level
	from the leaves of $B$.
	Recall that for each $\ell \ge 1$, we let $Q_\ell = \{u \in B: d(u,v) \ge \ell\}$ and  $E(Q_\ell)$ is the set of edges with both endpoints in $Q_\ell$.
	At any time in the revealment process, we say that a vertex $u \in Q_\ell$ is unsaturated in $Q_\ell$
	if there exists $w\in Q_\ell$ such that the edge-values $(\omega^{\xi}(uw),\omega^{\tau}(uw))$ have not been revealed. 
	Let $(U_e)_{e\in E(B)}$ be a family of i.i.d.\ uniform random variables on $[0,1]$ and reveal the configuration $\omega^{\xi}$ as follows:	
\begin{center}
	\fbox{
		\parbox{0.88\textwidth}{
		    \begin{definition}\label{def:correlation-decay-revealing} 
			\textbf{Initialize $\cV_\xi = \mathfrak V_\xi$ and $\cE_\xi = \emptyset$};
			
			\smallskip
			\textbf{for $i=1,2,...,R$ do}
			
			\smallskip
			\quad \textbf{while} $\exists u \in \cV_\xi $ such that $u$ is unsaturated in $Q_{R-i}$
			
			\smallskip
			\quad\quad \textbf{for each} vertex $w \in Q_{R-i}:~uw \in E(Q_{R-i})$

			\begin{enumerate}[\quad\quad\quad~1.]\setlength{\itemsep}{3mm}
			\item {Reveal $\omega^{\xi}(uw)$ from $\pi_{B}^{\xi}(\cdot \mid \omega(\cE_{\xi}))$ using $U_{uw}$, i.e., set
			\[
			  \omega^{\xi}(uw) = \begin{cases}
			  1 \qquad \mbox{if } \pi_{B}^{\xi}(\omega(uw)=1 \mid \omega(\mathcal E_{\xi}))\ge U_{uw} \\ 0 \qquad \mbox{else}
			  \end{cases}\,;
			\]}
			\item Add the edge $uw$ to the set $\cE_{\xi}$;
			\item If  $\omega^{\xi}(uw)=1$, add the vertex $w$ to $\cV_\xi$;  
			\end{enumerate}
\end{definition}
	}}
\end{center}
 Note that we can use the same family $(U_e)_{e\in E(B)}$ in this process to generate coupled samples of $\omega^\xi$ and $\omega^\tau$. Notice that this coupling is monotone, so that because $\xi\ge \tau$, $\omega^\xi \ge \omega^\tau$ almost surely. 
Let $\mathcal C_{\mathfrak V}^i(\omega^\xi)$ denote the set of open edges
 revealed up to the $i$-th iteration of the procedure; we observe that $\mathcal C_{\mathfrak V}^i(\omega^\xi)$ is not necessarily equal to the intersection of $\mathcal C_{\mathfrak V}(\omega^\xi)$ with $E(Q_{R-i})$, but it is a subset of $\mathcal C_{\mathfrak V}(\omega^\xi) \cap E(Q_{R-i})$. Refer to Figure~\ref{fig:correlation-decay-revealing} for a depiction of the above revealing procedure.
	
	\begin{figure}
    \centering
    \begin{tikzpicture}
    \node at (-4,2.75) {
    \includegraphics[width = .45\textwidth] {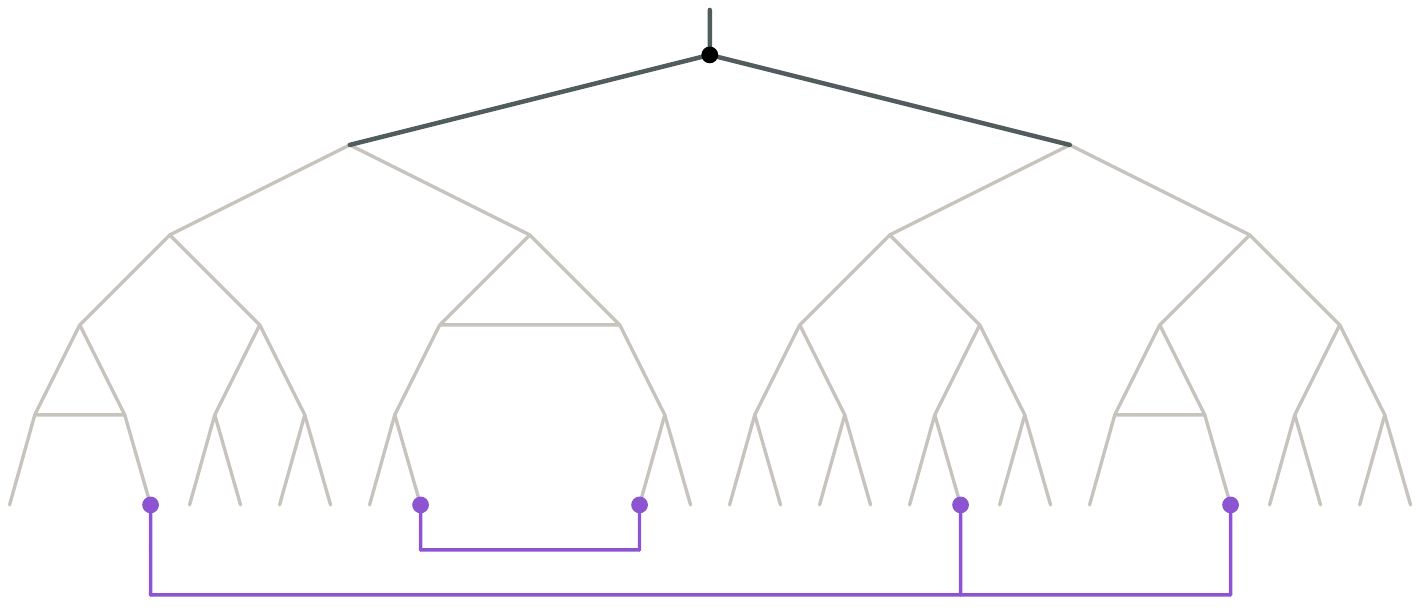}};
    \node at (4,2.99) { \includegraphics[width = .45\textwidth] {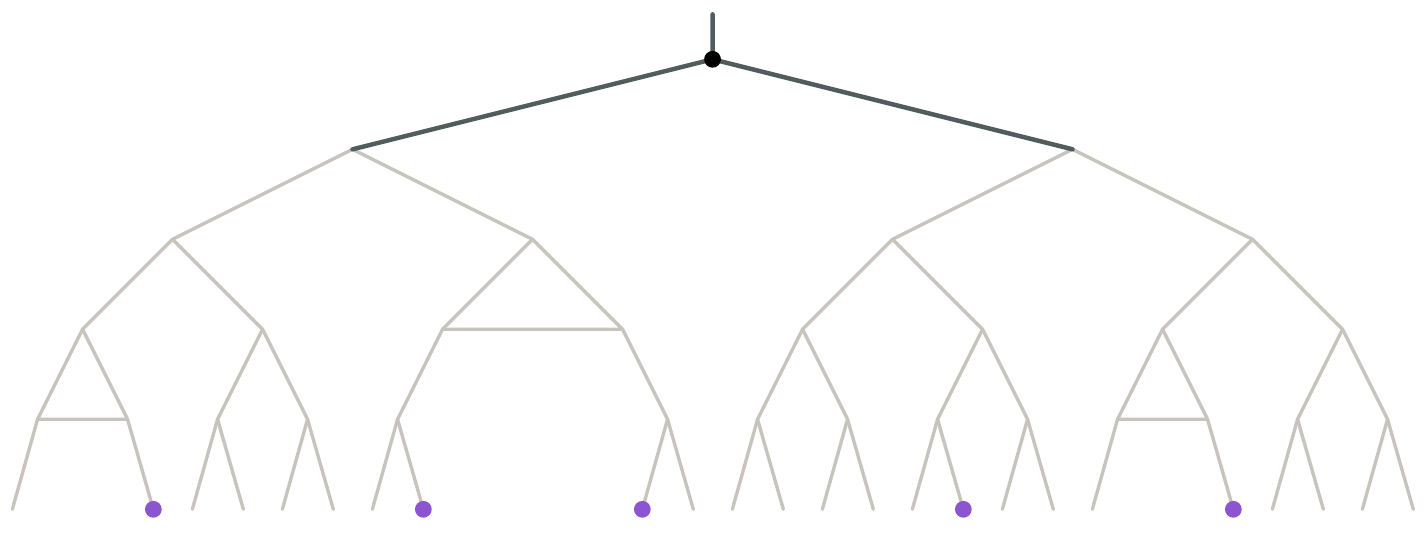}};

    \node[font = \tiny] at (-4.02,4.65) {$\vdots$};
    \node[font = \tiny] at (-4+.125,3.85) {$v$};
        \node[font = \tiny] at (3.98,4.65) {$\vdots$};
    \node[font = \tiny] at (4+.125,3.85) {$v$};

    \node at (-4,-1.75) {
    \includegraphics[width = .45\textwidth] {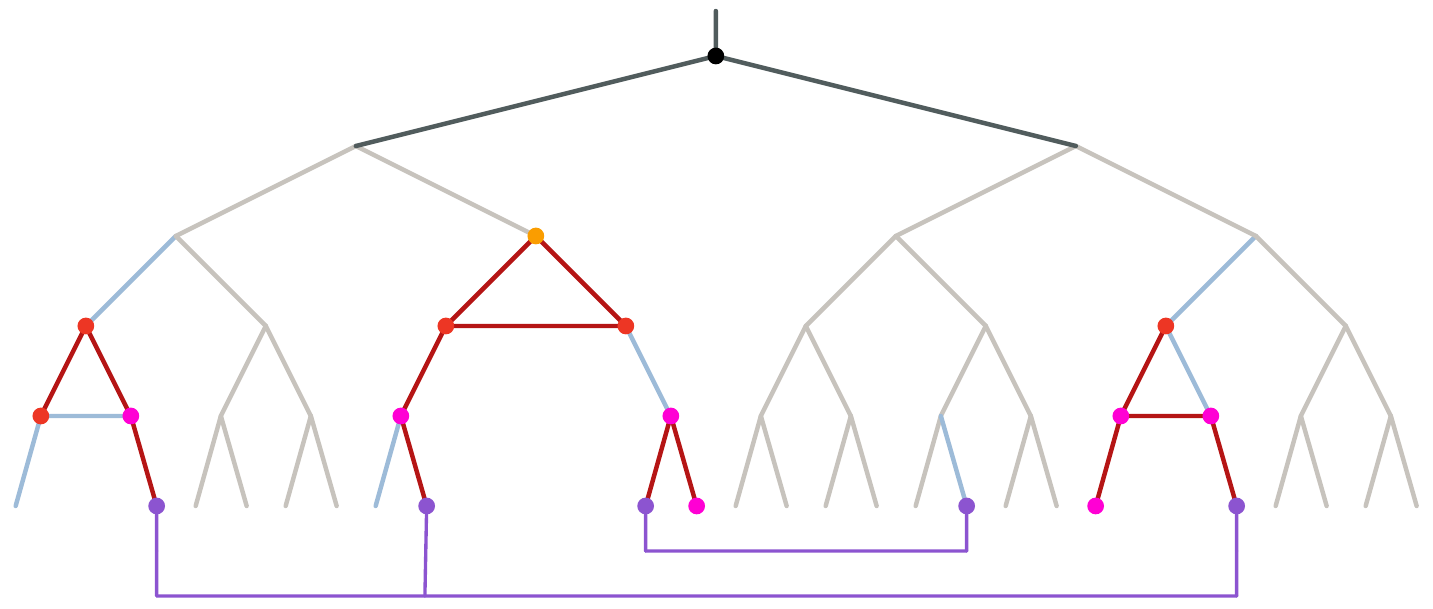}};
        \node at (4,-1.55) {
    \includegraphics[width = .45\textwidth] {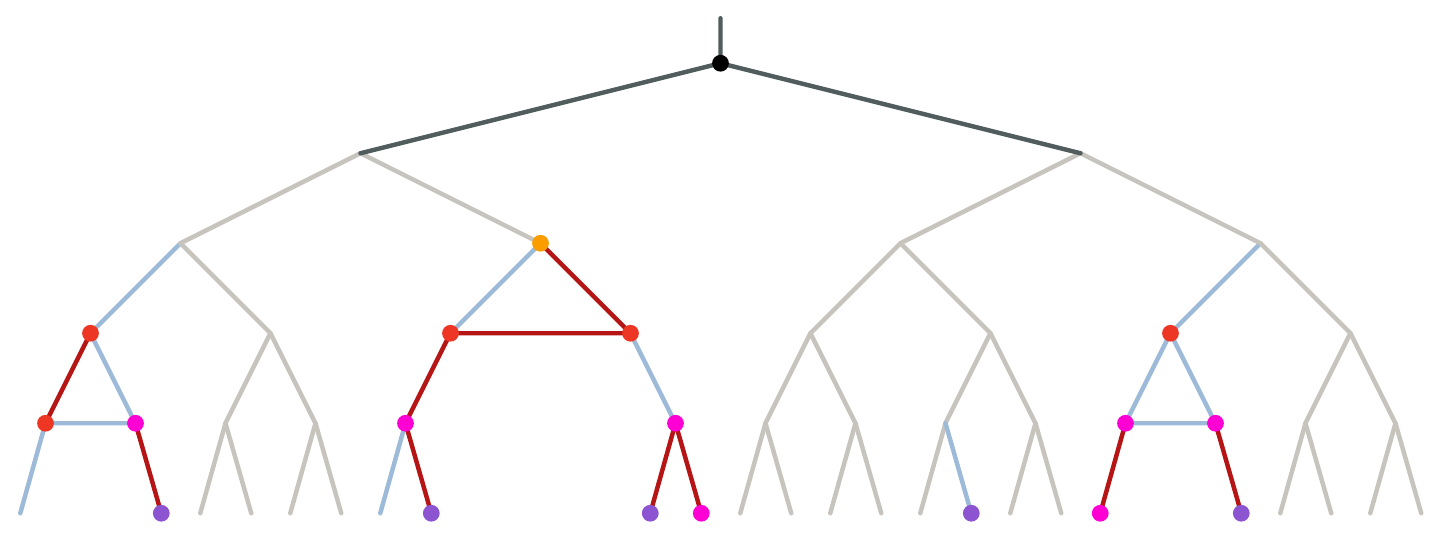}};

    \node[font = \tiny] at (-4.02,-4.5 + 4.65) {$\vdots$};
    \node[font = \tiny] at (-4+.125,-4.5+ 3.85) {$v$};
        \node[font = \tiny] at (3.98,-4.5+4.65) {$\vdots$};
    \node[font = \tiny] at (4+.125,-4.5+3.85) {$v$};
    
    \draw[dotted, color = gray] (-8, -1.4)--(8,-1.4);
    \end{tikzpicture}
    \vspace{-.1cm}
    \caption{Top: The ball $B_R(v)$ for $R = 5$, with a $K$-sparse boundary condition $\tau$ for $K = 4$ (left), and the free boundary condition $\xi = 0$ (right). Bottom: The configurations revealed by the procedure of Definition~\ref{def:correlation-decay-revealing}, showing $\cC_{\mathfrak V}^{i_0}(\omega^{\tau})$ (red, left) along with its outer edge boundary in $Q_{i_0}$ (light blue), revealing the dotted line (depth $i_0$) to be the largest $i$ for which the set $\cV_{i}$ is a singleton. The vertices that would have been exposed for larger values of $i$ are colored in different colors. The coupled edge configuration $\omega^0(\cC_{\mathfrak V}^{i_0}(\omega^{\tau}))$ is depicted on the right (open edges in red, closed edges in blue). The exposed configurations on $\cC_{\mathfrak V}^{i_0}(\omega^{\tau})$ induce free boundary conditions on $E(B)\setminus \bar \cC_{\mathfrak V}^{i_0}(\omega^{\tau})$. } 
    \label{fig:correlation-decay-revealing}
\end{figure}
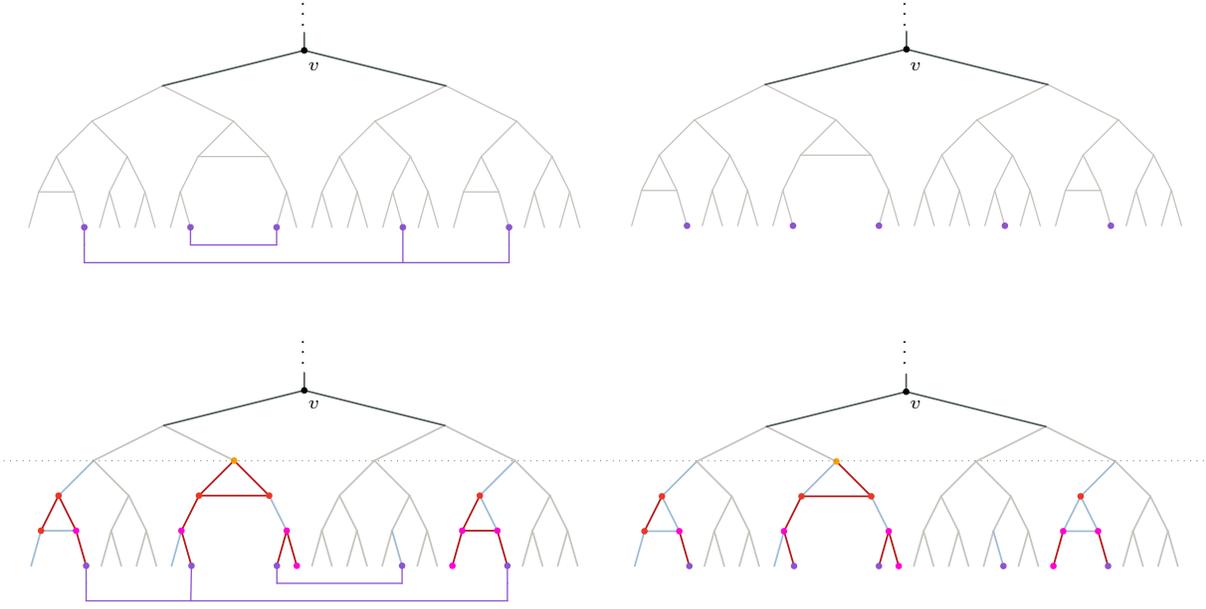

	Through this revealing process, we see that $\omega^{\xi}$ is open on the edges in the random set $\mathcal C_{\mathfrak V}^i(\omega^\xi)$ and free on the edges in its outer (edge) boundary in $Q_{R-i}$. Let $\bar {\cC}_{\mathfrak V}^i(\omega^\xi)$ be the union 
	$\mathcal C_{\mathfrak V}^i(\omega^\xi)$ with its outer (edge) boundary in $Q_{R-i}$, and note that this corresponds to the state of $\cE_\xi$ after the $i$'th iteration. 
	The random set $\mathcal C_{\mathfrak V}^i(\omega^\xi)$ is measurable with respect to the uniform random variables assigned to edges of $\bar {\cC}_{\mathfrak V}^i(\omega^\xi)$. 
	
	For each $\mathcal C_{\mathfrak V}^i(\omega^\xi)$, let $\cV^i(\omega^\xi)$ be the  vertices in $\mathcal C_{\mathfrak V}^i(\omega^\xi)$ at distance exactly $R-i$ from $v$. Then,
	\[
	\mathcal V^i(\omega^\xi) \subseteq \mathcal C_{\mathfrak V}(\omega^\xi)\cap \{w:d(w,v) = R-i\}\,.
	\] 
	On  $\Upsilon_B^c$, there must be some $i$ for which $|\mathcal C_{\mathfrak V}(\omega^\xi)\cap \{w:d(w,v) = R-i\}|\le 1$, and therefore $|\mathcal V^i(\omega^\xi)|\le 1$. Let $i_0$ be the first  $i$ for which $|\mathcal V^{i_0}(\omega^\xi)|\le 1$, and for ease of notation set $\mathcal V_0 = \mathcal V^{i_0}(\omega^\xi)$, $\bar \cC_{\mathfrak V,0} = \bar {\cC}_{\mathfrak V}^{i_0}(\omega^\xi)$ and set $\bar \cC_{\mathfrak V,0}^c = E(B) \setminus \bar \cC_{\mathfrak V,0}$. Notice the inclusion 
	\begin{align*}
	    \bar{\cC}_{\mathfrak V}^{i}(\omega^\xi) \subset \bar{\cC}_{\mathfrak V}^{i+1}(\omega^\xi)\,,
	\end{align*}
	and from that deduce that $i_0$ is measurable with respect to the uniform random variables assigned to edges of $\bar{\cC}_{\mathfrak V}^{i_0}(\omega^\xi)$. By the domain Markov property (see e.g.,~\cite{Grimmett}), conditionally on $\omega^\xi(\bar \cC_{\mathfrak V,0})$, the configuration $\omega^\xi(\bar \cC_{\mathfrak V,0}^c)$ (respectively,  $\omega^\tau(\bar \cC_{\mathfrak V,0}^c)$) is distributed according to the random-cluster distribution on $\bar \cC_{\mathfrak V,0}^c$ with boundary conditions induced by $\xi$ and $\omega^\xi(\bar \cC_{\mathfrak V,0})$, respectively $\tau$ and $\omega^\tau(\bar \cC_{\mathfrak V,0})$. 
	
	To conclude the proof, it suffices to see that because $|\mathcal V_0|\le 1$, both $\omega^\xi(\bar \cC_{\mathfrak V,0})$ and $\omega^\tau(\bar \cC_{\mathfrak V,0})$ induce the free boundary conditions on $\bar \cC_{\mathfrak V,0}^c$. In that case $\omega^\xi$ and $\omega^\tau$ would agree on $\bar \cC_{\mathfrak V,0}^c$ and in particular on $\mathcal N_v$. 
	By monotonicity, it suffices for us to show that the boundary conditions induced by $\xi$ and $\omega^\xi(\bar \cC_{\mathfrak V,0})$ on $\bar \cC_{\mathfrak V,0}^c$ are free. 
	Since the wirings of $\xi$ are only on vertices of $\mathfrak V_\xi \subset \bar \cC_{\mathfrak V,0}$, the only way for the boundary conditions on $\bar \cC_{\mathfrak V,0}^c$ to be not free is if multiple vertices on its boundary are incident to open edges of $\omega^\xi(\bar \cC_{\mathfrak V,0})$. By construction, the only vertices in $\bar \cC_{\mathfrak V,0}$ which can be incident to an open edge of $\omega^\xi(\bar \cC_{\mathfrak V,0})$ must be at distance exactly $R-i_0$ from $v$. By the assumption that $|\mathcal V_0|\le 1$, there can be at most one such vertex, and therefore there are no non-trivial (i.e., non-singleton) boundary components induced on $\bar \cC_{\mathfrak V,0}^c$ by the boundary condition $(\xi,\omega^\xi(\bar \cC_{\mathfrak V,0}))$, implying the desired conclusion. 
\end{proof}

\begin{proof}[{\textbf{\emph{Proof of Proposition~\ref{prop:influence-probability-new}}}}]
With Lemma~\ref{lem:influence-event} in hand, it suffices for us to prove the following: 
	there exists $C(p,q,K,L)>0$ such that if $G=(V,E)$ is an $(L,R)$-$\treelike$ graph and $\xi$ is a $K$-$\sparse$ boundary condition for the $L$-$\treelike$ ball
	$B := B_R(v)$ about some $v\in V$, we have 
	\begin{align}\label{eq:wts-spatial-mixing}
	\pi_{B}^\xi(\Upsilon_{B,\xi}) \le C  \ps^{2R}\,. 
	\end{align}

	Let $H\subset E(B)$ be a set of at most $L$ edges such that the subgraph $(V, E(B)\setminus H)$ is a tree; the existence of such a set is guaranteed by the fact that $B_R(v)$ is $L$-$\treelike$.  
	Let $\mathcal Z= \{d_1,...,d_k\}$ be the subset of distances (from $v$) which $H$ intersects, i.e., $\mathcal Z= \{1\le \ell < R: \exists w\in V(H): d(w,v)=\ell\}$. See Figure~\ref{fig:treelike-balls} for a depiction. Observe that each edge of $H$ intersects either one or two consecutive depths in $\mathcal Z$. Since $B$ Is $L$-$\treelike$, we clearly have $|\mathcal Z|\le 2L$. 
	Letting $d_0 =0$ and $d_{k+1}=R$,
	for $i=0,\dots,k$ we define:
	$$
	\mathcal F_i := \{u \in B:  d_i < d(u,v) < d_{i+1}\}\,.
	$$ 
	For each $0\le i \le k$, the graph 
	$\cF_i = (\mathcal F_i,E(\mathcal F_i))$ is a forest.
	 For each $i$, let $\mathcal T_{ij} = (\cT_{ij}, E(\cT_{ij}))$ for $j = 0,1,\dots$ denote the distinct connected components (subtrees) of $\mathcal F_i$ so that $\mathcal F_i = \bigcup_{j \ge 0} \mathcal T_{ij}$. (For some $i$, this may be empty, and for other $i$, this may be a single vertex.)
		
	Now, in order for $\Upsilon_{B,\xi}$ to hold, it must be the case that in each $\cF_i$, every depth $\ell$ is intersected by at least two sites in the FK cluster of $\mathfrak V_{B,\xi}$ in $Q_\ell$.
	Specifically, for each $i$, at distance $d_i+1$ from $v$ there must be at least two distinct vertices connected to 
	$\mathfrak V_{B,\xi}$ with paths in $Q_{d_i+1}$. Thus,
    for each $i$ 
	there must exist 
	two open monotone paths (each intersecting each height in $\cF_i$ at exactly one vertex), $\gamma_i \subset E(\mathcal T_{ij})$ and $\gamma_i' \subset E(\mathcal T_{ij'})$ with $j \neq j'$ such that $\gamma_i$ (resp., $\gamma_i'$) connects the root of $\mathcal T_{ij}$ (resp., $\mathcal T_{ij'}$) 
	to one of its leaves.
	If there are multiple such paths, choose according to some predetermined ordering, and call the sequences of paths $\Gamma = \gamma_{0},\ldots, \gamma_{k}$ and $\Gamma' = \gamma_{0}',\ldots, \gamma_{k}'$. See Figure~\ref{fig:treelike-balls-influence} for a depiction.

	We enumerate over the choices of such sequences of paths and then show that for any two fixed sequences of paths, the probability that they are both open is bounded by $C \ps^{2R}$ for some $C(p,q,\Delta,K,L)$. (We say that a sequence of paths is open if all of its paths are.)
	
	In order to enumerate over the choices of sequences of paths, for each monotone path $\gamma_i$, let $x_i$ be its bottom endpoint, and define $x_i'$ for $\gamma_i'$ similarly. Since $\xi$ is $K$-$\sparse$, there are evidently at most $K$ many choices of $x_0$, and $K$ choices of $x_0'$. 
	Now observe that since $\gamma_i$ is a monotone path on a tree, for each $i$, the bottom endpoint $x_i$ determines the entire path $\gamma_i$. Since these paths form parts of the connections to $\mathfrak V_{B,\xi}$ the sequence of paths can be required to have endpoints at depths $d_{i+1}-1$ that are either an ancestor of $x_0$, or an ancestor of $V(H)$ . Here, at each height $h\notin \mathcal Z$ an \emph{ancestor} of a vertex $u$ at height $h$ is a vertex along the geodesic from $v$ to $u$. We make the following observation. 
	\begin{claim}
	If $B_R(v)$ is $L$-$\treelike$, if $u$ is such that $d(u,v) = h$, for every $h'<h$, $u$ has at most  $2^L$ many ancestors at height $h'$. 
	\end{claim}
	Indeed, except along the edges in $H$, every vertex has a unique \emph{parent} which is an ancestor of that vertex at one smaller depth. Thus, the geodesics of $B$ are uniquely determined by their endpoints together, possibly, with a subset of edges of $H$ traversed along the geodesic, yielding the at most $2^L$ available choices.  
	
	Returning to the enumeration over $\Gamma, \Gamma'$, the heights of the endpoints $x_i,x_i'$ are predetermined by $i$, and therefore, having chosen $x_0, x_0'$ for each $i$, there are at most $2L$ many choices of bottom end-point $x_i$, and likewise of $x_i'$, and therefore at most $2L\cdot 2^L$ many choices of $\gamma_i$ and $\gamma_i'$. 
	
	\begin{figure}
    \centering
    \begin{tikzpicture}
    \node at (-4.2,0) {
    \includegraphics[width = .47\textwidth] {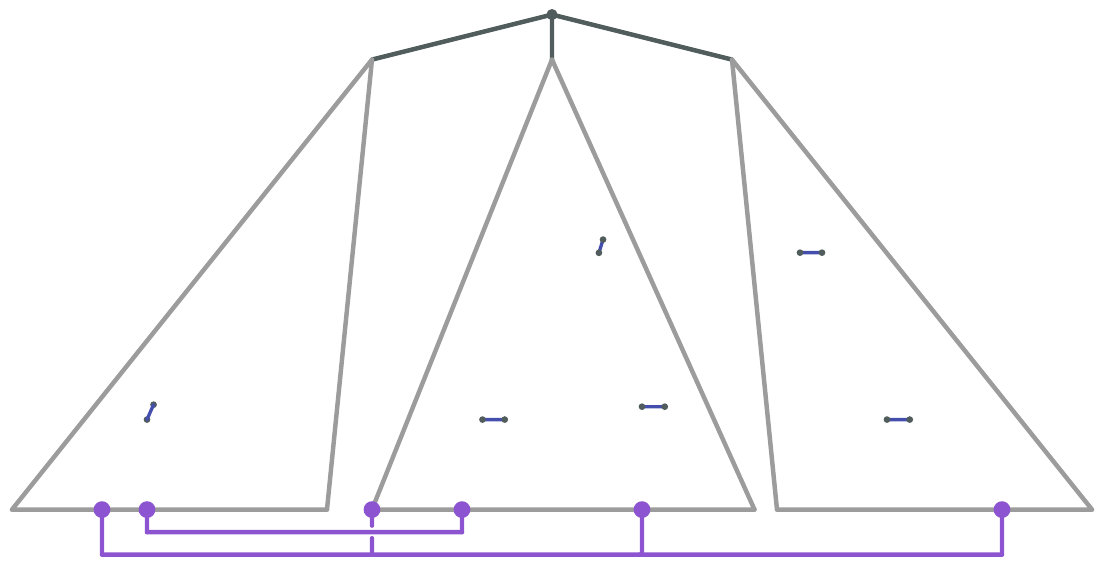}};

    \node at (4.2,.03) {
    \includegraphics[width = .47\textwidth] {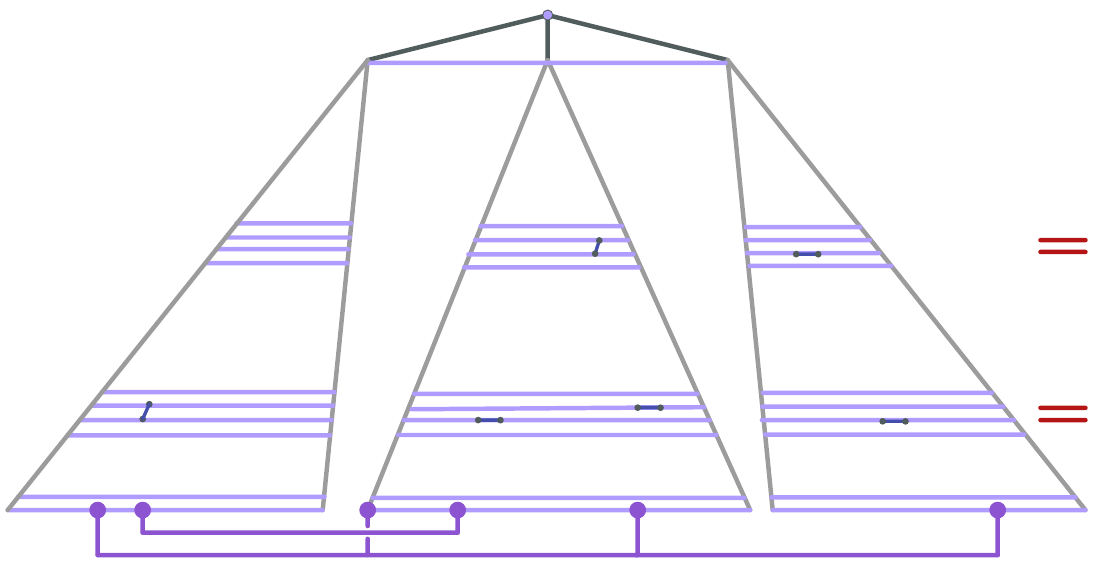}};
    \end{tikzpicture}
    \vspace{-.1cm}
    \caption{Left: An $L$-$\treelike$ ball with $K$-$\sparse$ boundary conditions is depicted for $L = K = 6$: the $L$ edges that need to be removed to leave a tree are indicated in blue. Right: We modify the boundary conditions to be all wired (the wired component is depicted in purple) at or one away from heights in $\mathcal Z$ (marked by red dashes).}
    \label{fig:treelike-balls}
\end{figure}

	Hence, a union bound implies
	\begin{align}
	\label{eq:event-bound}
	\pi_{B}^\xi (\Upsilon_{B,\xi}) \le  K^2(2L)^{2L} (2^{L})^{2L}  \sup_{\Gamma,\Gamma'}\, \pi_{B}^\xi(\omega(\Gamma \cup \Gamma')=1)\,. 
	\end{align}
	
	Now fix any two such sequences of paths $\Gamma, \Gamma'$, and consider the probability that $\omega(\Gamma \cup \Gamma') =1$. 
	Observe that $\Gamma$ and $\Gamma'$ are vertex-disjoint by construction.
	Our aim is to make the events that $\Gamma$ and $\Gamma'$ are open in $\omega$ independent. 
	For this, let $\rho_i$ be the set of roots of the trees in $\mathcal F_i$. We introduce auxiliary wirings (as shown in Figures~\ref{fig:treelike-balls}--\ref{fig:treelike-balls-influence}) for all vertices at depths $\{d: \min_{i=0 ,...,k+1} |d-d_i|\le 1\}$.
	Call the resulting distribution $\tilde \pi_{B}$; by monotonicity,  
	\begin{align}\label{eq:monotonicity-pi-tilde-pi}
	\pi_{B}^\xi(\omega(\Gamma \cup \Gamma')=1) \le \tilde \pi_{B}(\omega(\Gamma \cup \Gamma')=1)\,.
	\end{align}
	The distribution $\tilde \pi_{B}$ is a product measure over the $\mathcal T_{ij}$'s
	with boundary condition $(1,\circlearrowleft)$ in each  $\mathcal T_{ij}$ (recall that this boundary condition wires all leaves $\partial \cT_{ij}$ together with the root of $\mathcal T_{ij}$). 
	Hence, since $\Gamma$ and $\Gamma'$ are such that, for each $i \ge 0$,
	$\gamma_i$ and $\gamma_i'$ belong to distinct subtrees $\mathcal T_{ij_i}$, $\mathcal T_{ij_i'}$ of the forest $\mathcal F_i$, and	
	we have
	$$
		\tilde \pi_{B}(\omega(\Gamma \cup \Gamma')=1) = 
		\prod_{i=0}^{k} \pi_{\cT_{ij_i}}^{(1,\circlearrowleft)}(\gamma_i) 
		\prod_{i=0}^{k} \pi_{\cT_{ij'_i}}^{(1,\circlearrowleft)}(\gamma_i')\,.
	$$
	
		\begin{figure}
    \centering
    \begin{tikzpicture}
    \node at (-4.2,0) {
    \includegraphics[width = .47\textwidth] {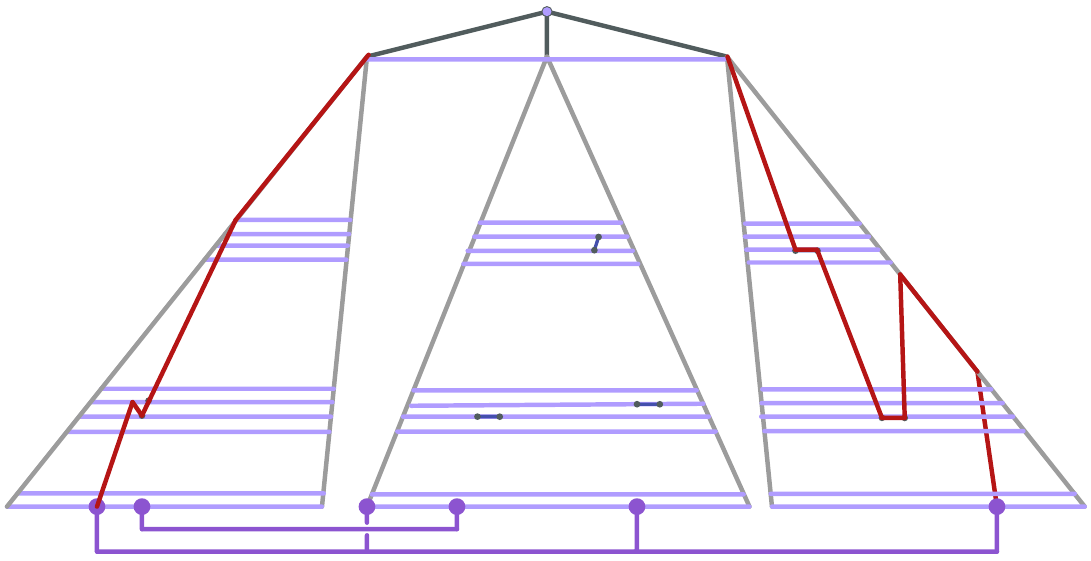}};

    \node at (4.2,0) {
    \includegraphics[width = .47\textwidth] {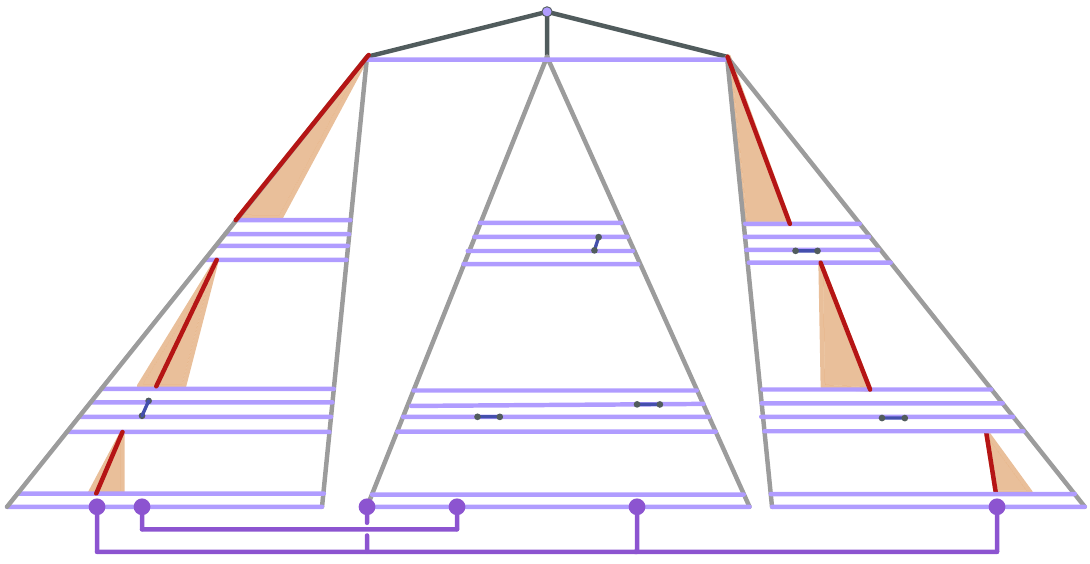}};
    \end{tikzpicture}
    \vspace{-.1cm}
    \caption{Left: Two disjoint components (red) of the vertices in $\mathfrak V_{B,\xi}$, together intersect every depth in the ball and satisfy the event $\Upsilon_{B,\xi}$. Right: The two components contain corresponding sequences of open leaf-to-root paths (red) in independent wired subtrees (shaded, orange) whose endpoints are amongst the ancestors of vertices of $H$ or $\mathfrak V_{B,\xi}$.}
    \label{fig:treelike-balls-influence}
\end{figure}

	Let $h_i= d_{i+1} - d_i$ be the height of the trees in $\mathcal F_i$.
	We deduce from Lemma~\ref{lem:exp-decay-wired-tree} that there exists a constant  $A (p,q,\Delta) > 0$ such that uniformly over $\Gamma, \Gamma'$, 
	\begin{align*}
	\tilde \pi_{B}(\omega(\Gamma \cup \Gamma')=1)  \le A^{2L} \prod_{i=0}^{k} \ps^{2h_i} \le A^{2L} \ps^{2(R-4L)}\,.
	\end{align*}
	Plugging this bound into~\eqref{eq:event-bound}--\eqref{eq:monotonicity-pi-tilde-pi}, we obtain
	$$
		\pi_{B}^\xi (\Upsilon_{B,\xi}) \le K^2((2L)(2^L) \ps^{-4}  A)^{2L} \ps^{2R}\,,
	$$
	from which the required~\eqref{eq:wts-spatial-mixing} follows.
\end{proof}

\begin{remark}
\label{rmk:dr:lb}
A matching lower bound of $\Omega(\hat p^{2R})$ 
for the decay rate in Proposition~\ref{prop:influence-probability-new}
is easy to construct by e.g., taking the $K$-$\sparse$ boundary conditions $\xi$ that wires two leaves $w_1,w_2$ on distinct sub-trees of $v$, and the free boundary conditions $\xi'=0$ on $\cT_R$. The event that
the root is connected to $w_1$ and its corresponding child is connected to $w_2$ has probability at least $C \hat p^{2R}$ by Lemma~\ref{lem:exp-decay-wired-tree} and the FKG inequality (see e.g.,~\cite{Grimmett}). On this event, the probability that the edge incident $v$ down towards $w_2$ is open is $p$ under the boundary condition $\xi$ and $\hat p$ under $\xi'=0$.
\end{remark}

\section{Proof of fast mixing}\label{sec:proof-main-theorem}
In this section, we combine the results of Sections~\ref{sec:shattering}--\ref{sec:correlation-decay-treelike} to conclude the proof of Theorem~\ref{thm:main-fk}. As indicated in Section~\ref{sec:proof-strategy}, the analysis of Sections~\ref{sec:shattering}--\ref{sec:correlation-decay-treelike} reduce the mixing time of the FK-dynamics on a random graph to understanding the convergence to equilibrium on $O(1)$-$\treelike$ balls of volume $O(n^{\frac 12 - \delta})$ with $O(1)$-$\sparse$ boundary conditions. In Section~\ref{subsec:mixing-time-prelim}, we recall the log-Sobolev inequality and comparison bounds for the log-Sobolev constant under different boundary conditions. In Section~\ref{subsec:local-mixing}, we bound this log-Sobolev constant via straightforward comparison to a product chain. Then in Section~\ref{subsec:mainthm-proof}, we proceed to combine all of the above ingredients to deduce the proof of Theorem~\ref{thm:main-fk} using the censoring inequalities of~\cite{PWcensoring}.  

\subsection{Mixing time preliminaries}\label{subsec:mixing-time-prelim}
Let us recall some standard tools to help us bound the rate of convergence to equilibrium of the FK-dynamics on treelike balls with sparse boundary conditions.

\subsubsection*{Log-Sobolev inequalities}
Recall, for a Markov chain with transition matrix $P$, the Dirichlet form 
\begin{align}\label{eq:Dirichlet-form}
    \mathcal E(f,f) := \frac{1}{2}\sum_{\omega,\omega'\in \{0,1\}^E} \pi(\omega) P (\omega, \omega') (f(\omega) - f(\omega'))^2\,,
\end{align}
for $f:\Omega \to \mathbb R$. Then the \emph{log-Sobolev constant} is given by 
\begin{align}\label{eq:lsi-constant}
    \alpha(P) := \min_{f: \mbox{Ent}_\pi[f^2]\ne 0} \frac{\mathcal E(f,f)}{\mbox{Ent}_\pi[f^2]}\,, \qquad \mbox{where} \qquad \mbox{Ent}_{\pi} [f^2] = \mathbb E_{\pi}\Big[f^2 \log \frac{f^2}{\mathbb E_{\pi}[f^2]}\Big]\,.
\end{align}

As such, a \emph{log-Sobolev inequality} takes the form $\mathcal E(f,f) \ge \gamma \cdot \mbox{Ent}_\pi[f^2]$ for all $f$. A log-Sobolev inequality is stronger than a mixing time bound, in the sense that it implies exponential convergence with rate $\gamma$ in total-variation distance from the stationary distribution. This is captured by the following standard fact (e.g., a proof in the discrete time setting we consider follows immediately from Lemma 2.8 and Eq.~(2.10) of~\cite{BCPSV}). 

\begin{fact}\label{fact:log-sobolev-tv-distance}
    Consider an ergodic finite state Markov chain $(X_t)_{t \ge 0}$ with transition matrix  $P$ reversible with respect to stationary measure $\pi$. If the chain has a log-Sobolev constant $\alpha = \alpha(P)$, then for every $\gamma <\alpha$, 
    \begin{align*}
        \max_{x_0\in \Omega} \|P (X_t^{x_0} \in \cdot) - \pi\|_\tv\le e^{-\gamma t/2} \Big(\log \frac{1}{\min_{x\in \Omega} \pi(x)}\Big)^{1/2}\,.
    \end{align*}
\end{fact}

\subsubsection*{Boundary condition comparisons for the FK-dynamics}
The following formalizes the notion that sparse boundary conditions are ``close to free", and allows us to compare the induced mixing time on balls with sparse boundary to those with free boundary. 

\begin{definition}[Definition 2.1  of~\cite{BGVfull}]
	 For two boundary conditions (partitions) $\phi \leq \phi'$, define $D(\phi,\phi') := c(\phi) - c(\phi')$ where $c(\phi)$ is the number of components in $\phi$. For two partitions $\phi, \phi'$ that are not comparable, let $\phi''$ be the smallest partition such that $\phi'' \geq \phi$ and $\phi'' \geq \phi'$ and set $D(\phi,\phi') = c(\phi) - c(\phi'')+ c(\phi')- c(\phi'')$. 
\end{definition}

\begin{lemma}[Lemma 2.2 of~\cite{BGVfull}]
	\label{lemma:simple-rc-bound}
	Let $G=(V,E)$ be an arbitrary graph, $p \in (0,1)$ and $q > 0$. Let $\phi$ and $\phi'$ be two partitions of $V$ encoding two distinct external wirings on the vertices of $G$.
	Let $\pi_{G}^\phi$, $\pi_{G}^{\phi'}$ be the resulting random-cluster measures. Then, for all FK configurations $\omega\in \{0,1\}^E$, we have
	$$
	q^{-2D(\phi,\phi')} {\pi_{G}^{\phi'}(\omega)} \le \pi_{G}^{\phi}(\omega) \le q ^{2D(\phi,\phi')} \pi_{G}^{\phi'}(\omega)\,.
	$$  
\end{lemma}

From Lemma~\ref{lemma:simple-rc-bound}, and the definition of the Dirichlet form,~\eqref{eq:Dirichlet-form}, we deduce the following.

\begin{cor}\label{cor:comparison-df}
	Let $G=(V,E)$ be an arbitrary graph, $p \in (0,1)$ and $q > 0$.
	Consider the FK-dynamics on $G$ with boundary conditions $\phi$ and $\phi'$,
	and let $\df_{G^\phi}$, $\df_{G^{\phi'}}$ denote their Dirichlet forms, respectively.
	Then 
	$$
	q^{-4 D(\phi,\phi'))} \df_{G^{\phi'}}(f,f) \le \df_{G^\phi}(f,f)\leq   q^{ 4 D(\phi,\phi'))} \df_{G^{\phi'}}(f,f)\,,\qquad \mbox{for all }f:\{0,1\}^E \to \mathbb R\,.
	$$
\end{cor}
Together with Corollary~\ref{cor:comparison-df} and Lemma~\ref{lemma:simple-rc-bound} again, this controls the change in both \emph{log-Sobolev constant}~\eqref{eq:lsi-constant}, and mixing time, under two boundary conditions with distance $D(\phi,\phi')$.

\subsection{Local mixing: fast mixing on treelike graphs with sparse boundary conditions}\label{subsec:local-mixing}
In this section we establish a bound for the speed of convergence of the FK-dynamics on $L$-$\treelike$ balls 
with $K$-$\sparse$ boundary conditions (see Definitions~\ref{def:k-treelike} and~\ref{def:k-sparse-bc}). 
Our goal is to prove the following lemma.

\begin{lemma}
	\label{lemma:main-tree-mixing}
	Suppose $B_R(v)$ is $L$-$\treelike$. Let $\xi$ be a $K$-$\sparse$ boundary condition on $\partial B_R(v)$. For every $p\in (0,1)$ and $q>0$, the log-Sobolev constant of the FK-dynamics on $B_R(v)$ with boundary condition $\xi$ is $\Omega(|E(B_R(v))|^{-1}) = \Omega(d^{-R})$. 
\end{lemma}

Lemma~\ref{lemma:main-tree-mixing} follows by comparing log-Sobolev on an $L$-$\treelike$ ball with $K$-$\sparse$ boundary to a tree with $K$-$\sparse$ boundary conditions, whose log-Sobolev constant is bounded by comparison to a product chain. We first note the following bound on the log-Sobolev constant on trees with sparse boundaries.
\begin{corollary}
	\label{cor:tree:sht:mixing}
	There exists $c(p,q)>0$ such that the following holds. For every rooted (not necessarily complete) tree $\hat \cT_h = (V(\hat \cT_h),E(\hat \cT_h))$ of depth $h$ and degree at most $\Delta$, and every $K$-$\sparse$ boundary condition $\phi$ on $\hat \cT_h$,  
	the log-Sobolev constant of the FK-dynamics on $\hat{\T}_h$ with boundary conditions $\phi$ is at least $c q^{6K}{|E(\hat{\T}_h)|}^{-1}$.
\end{corollary}	

\begin{proof}
	Consider the FK-dynamics on $\hat{\T}_h$ under the free boundary conditions.
	In this case, the random-cluster measure is a $\ber(\ps)$ product measure
	and thus the log-Sobolev constant of the FK-dynamics is $c {|E(\hat \cT_h)|}^{-1}$ for some $c(p,q)>0$; see, e.g.,~\cite{Diaconis96}.
	The result then follows from Lemma~\ref{lemma:simple-rc-bound} and Corollary~\ref{cor:comparison-df}.
\end{proof}

To move from mixing on an $L$-$\treelike$ ball to mixing on a tree, the following fact will be useful.

\begin{fact}
	\label{fact:generic:prod}
	Let $G$ be a subgraph of $G'$ such that $V(G) = V(G')$ and $E(G) \subset E(G')$; let $H = E(G') \setminus E(G)$. Suppose $\phi$ is a boundary condition on $G,G'$ such that for every $e\in H$, the endpoints of $e$ are wired in $\phi$. 
	For every $p\in (0,1)$ and $q>0$, let $P_{G}$ and $P_{G'}$ be the transition matrices of the FK-dynamics on $G$ and $G'$, respectively, with boundary conditions $\phi$, and let $\alpha (P_G)$ and $\alpha(P_{G'})$ be their log-Sobolev constants. 
 There exists a constant $c(p) > 0$ such that
	\begin{align*}
	\alpha(P_{G'}) &\ge \min\left\{\frac{|E(G)| }{|E(G)|+|H|} \cdot  \alpha(P_{G}), \frac{c|H|}{|E(G)|+|H|}\right\}\,.
	\end{align*}
\end{fact}

\begin{proof}
	The FK-dynamics on $G'$ is a product Markov chain on $\{0,1\}^{E(G)} \times \{0,1\}^H$
	with stationary distribution $\pi_{G'}^\phi = \pi_{G}^\phi \otimes \prod_{i=1}^{|H|} \nu_i$, where $(\nu_i)_{1\le i\le |H|}$ are independent $\ber(p)$ distributions over edges in $H$. 
	The result then follows from the tensorization of the log-Sobolev inequality (e.g.,~\cite[Lemma 2.2.11]{SClecture-notes}).	
\end{proof}

We can now combine the above ingredients to deduce the bound of Lemma~\ref{lemma:main-tree-mixing}.

\begin{proof}[\textbf{\emph{Proof of Lemma~\ref{lemma:main-tree-mixing}}}]
	Let $B = B_R(v)$ and let $H\subset E(B)$ be a set of at most $L$ edges such that $(B, E(B)\setminus H)$ is a tree. 
	Consider the tree $\hat \cT_R = (V(B),E(B) \setminus H)$ 
	and let $\phi$ be the boundary condition 
	that includes all the connections from $\xi$ and adds wirings between 
	$w$ and $w'$ for every edge $ww' \in H$. 

	Corollary~\ref{cor:tree:sht:mixing} implies that the log-Sobolev constant for the FK-dynamics on $\hat \cT_R$ with boundary condition $\phi$ is at least ${c q^{6(K+L)}}{|E(\hat \cT_R)|}^{-1}$ for some $c(p,q)>0$. 
	We then get from Fact~\ref{fact:generic:prod} that the log-Sobolev constant for the FK-dynamics on $B$ with boundary condition $\phi$ is at least ${cq^{6(k+L)}}{|E(B)|}^{-1}$.
	Lemma~\ref{lemma:simple-rc-bound} and Corollary~\ref{cor:comparison-df}
	then imply that the log-Sobolev constant on $B$ with boundary conditions $\xi$ is at least  ${c q^{6 K+12 L}}{|E(B)|}^{-1}$. 
\end{proof}

\subsection{Proof of Theorem~\ref{thm:main-fk}: upper bound}\label{subsec:mainthm-proof}
    Fix $p<p_u(q,\Delta)$, let $\epsilon= 1-\hat p d$ (positive when $\hat p <p_u$) and fix $\delta>0$ small enough (depending on $\epsilon, \Delta$) such that
    \begin{align*}
        2\delta + (1-2\delta) \log_d (1-\epsilon) <0\,,
    \end{align*}
    in which case the following is polynomially decaying in $n$:
    \begin{align}\label{eq:choice-of-delta}
     n \hat p^{(1 - 2 \delta)\log_d n}  = n d^{ (-1+  2\delta)\log_d n} (1-\epsilon)^{(1-2\delta)\log_d n} = n^{2\delta} (1-\epsilon)^{(1-2\delta)\log_d n}\,.
    \end{align}
    Let $R = (\frac 12 - \delta)\log_d n$ and let $K$ be a constant sufficiently large (depending on $p,q,\Delta$) that both  Fact~\ref{fact:random-graph-treelike} and Theorem~\ref{thm:k-R-sparse-whp} hold for $(K,R)$. For each $t$, let $\Gamma_t$ be the set of $\Delta$-regular graphs on $n$ vertices having  
    \[                          
    \Gamma_{t} = \{\cG: \cG \mbox{ is }(K,R)\mbox{-}\treelike\}\cap \{ \cG : P(X_{\cG,t}^1 \mbox{ is }(K,R)\mbox{-}\sparse) \ge 1-n^{-2} \}\,.
    \]
    By Fact~\ref{fact:random-graph-treelike} and Theorem~\ref{thm:k-R-sparse-whp}, there exists $C_0(p,q,\Delta)$ such that if $T= C_0 n \log n$, then $\Prrg (\Gamma_T^c ) \le o(1)$. It suffices for us to prove that the mixing time of the FK-dynamics on any $\cG\in \Gamma_T$ is $O(n \log n)$.
    
    Fix any $\cG \in \Gamma_T$ and for every configuration $\omega$ on $E(\cG)$, let $X_t^\omega = X_{\cG,t}^\omega$ be the FK-dynamics chain on $\cG$ initialized from $X_0^\omega = \omega$. Couple the family of chains $((X_t^\omega)_{t\ge 0})_{\omega \in \{0,1\}^{E(\cG)}}$ using the grand coupling as in Definition~\ref{def:coupling-Glauber-chains}: recall that this is the coupling that in each step picks the same random $e\in E(\cG)$ to update, and the same uniform random variable $U_{e,t}$ on $[0,1]$ to decide the next state on the edge $e$. As mentioned earlier, this coupling is monotone when $q>1$ so that for every $t\ge 0$, if $X_t^{\omega} \le X_t^{\omega'}$, then $X_{t+1}^{\omega}\le X_{t+1}^{\omega'}$. 
    
    It follows from the definition of $\tmix$ and monotonicity of the grand coupling (see e.g.,~\cite{LP}), that it suffices for us to show that there exists $\hat C(p,q,\Delta)$ such that if $\hat T = T + \hat C n \log n$, 
    \[
    \mathbb P \big(X_{\hat T}^1 \ne X_{\hat T}^0\big) \le\frac 14\,.
    \]
    By a union bound over the $n$ edge-neighborhoods $\cN_v$ (edges of $\cG$ incident $v$), this reduces to showing 
    \begin{align}\label{eq:wts-main-theorem}
         \sup_{v\in V(\cG)}\mathbb P \big(X_{\hat T}^1 (\cN_v) \ne X_{\hat T}^0 (\cN_v)\big) \le \frac 1{4n}\,.
    \end{align}
    Now fix any such $v$ and consider the probability above. For ease of notation, let $B_v = B_R(v)$ and $B_v^c = E(\cG)\setminus B_v$. 
    Introduce two new Markov chains $Y_t^1$ and $Y_t^0$ that are coupled via the grand coupling to $X_t^1, X_t^0$ except that they censor (ignore) all updates on edges of $B_v^c$ after time $T =  C_0 n\log n$. The censoring inequality~\cite[Lemma 2.3]{PWcensoring} implies the stochastic relations $Y_t^1 \succcurlyeq X_t^1$ and $Y_t^0 \preccurlyeq X_t^0$ for all $t\ge 0$ and thus 
\[
    \mathbb P \big (X_{t}^1(\cN_v)\ne X_{t}^0(\cN_v) \big) \le \Delta\sup_{e\in \cN_v} \mathbb P \big(X_{t}^1(e) \ne X_{t}^0 (e)\big) \le \Delta \sup_{e\in \cN_v}[\mathbb P \big(Y_{t}^1(e) = 1\big) - \mathbb P \big(Y_{t}^0(e) = 1\big)]\,.
\]
Fix any $e\in \cN_v$ and consider the difference in probabilities on the right-hand side. Let $\cE_{T}$ be the event (measurable with respect to the first $T$ steps of the Markov chain) that the boundary conditions induced by $X_{T}^1(B_v^c)$ are $K$-$\sparse$. Observe that $K$-sparsity of a boundary condition is a decreasing event, so that on $\cE_{T}$, the boundary conditions induced by $X_T^0(B_v^c)$ are also $K$-$\sparse$.
As such, for all $t\ge T$, 
\begin{align}\label{eq:Y-wts}
    \mathbb P (Y_t^1 & (e) = 1)-  \mathbb P ( Y_t^0 (e) = 1) \\
    & \le \mathbb P (\cE_{T}^c)+ \sup_{\substack{\phi^0,\phi^1\in \{0,1\}^{B_v^c} \\ \phi^1\textrm{ is }K-\sparse\,;\, \phi^0 \le \phi^1}} \mathbb P ( Y_t^1 (e) =1 \mid Y_{T}^1(B_v^c)= \phi^1) - \mathbb P (Y_t^0(e) = 1 \mid Y_T^0 (B_v^c) = \phi^0)\,. \nonumber
\end{align}
Since $\cG\in \Gamma_T$, and $Y_T^1 = X_T^1$, the first term is at most $n^{-2}$. Turning to the second term, fix any two configurations $\phi^1, \phi^0$ on $B_v^c$ such that $\phi^0 \le \phi^1$ and $\phi^1$ (and therefore also $\phi^0$) induce $K$-$\sparse$ boundary conditions on $B_v$, and consider the difference 
\begin{align}
    \mathbb P (Y_{T+s}^1(e) = 1  \mid Y_T^1 (B_v^c) & = \phi^1)  - \mathbb P (Y_{T+s}^0(e) =1 \mid Y_T^0(B_v^c) = \phi^0 ) \nonumber \\ 
    & \le |\mathbb P (Y_{T+s}^1(e) = 1 \mid Y_T^1 (B_v^c) = \phi^1) - \pi_{\cG}(\omega(e) =1 \mid \omega(B_v^c) =\phi^1)| \label{mixing:plus-bound} \\
    & \quad + |\pi_{\cG}(\omega(e) =1 \mid \omega(B_v^c) =\phi^1) - \pi_{\cG}(\omega(e) =1 \mid \omega(B_v^c) =\phi^0)| \label{mixing:smp-bound} \\ 
    & \quad + |\mathbb P (Y_{T+s}^0(e) = 1 \mid Y_T^0 (B_v^c) = \phi^0) - \pi_{\cG}(\omega(e) =1 \mid \omega(B_v^c) =\phi^0)|\,. \label{mixing:minus-bound}
\end{align}
Observe that $Y_{T+s}^1(B_v)$ is distributed as a \emph{lazy} FK-dynamics chain $Z_s^1$ on $B_v$ with boundary conditions induced by $\phi^1$, initialized from the random configuration $Z_0^1(B_v) = Y_{T}^1(B_v)$: the laziness is in the choice that at each step, $Z_s^1$ makes an FK-dynamics update on $B_v$ with probability $|E(B_v)|/|E(\cG)|$ and makes no update otherwise. The analogous statement holds for $Y_{T+s}^0(B_v)$ with respect to some lazy chain $Z_s^0$. The invariant measure of $Z_s^1$ is easily seen to be 
$$\pi_{\cG}(\omega(B_v)\in \cdot \mid \omega(B_v^c) = \phi^1) = \pi_{B_v}^{\phi^1}\,,$$
and the analogous statement holds for $Z_s^0$. 

Now let $\hat T = T+\hat S$ where $\hat S =  \hat C n\log n$ for a constant $\hat C$ to be chosen sufficiently large depending on $p,q,\Delta$. The expected number of updates in $B_v$ between time $T$ and $T+\hat S$ is 
\begin{align*}
    \hat C n \log n \cdot \frac{|E(B_v)|}{|E(\cG)|}  \ge  2\Delta^{-1} \hat C |E(B_v)| \log n\,.
 \end{align*}
 Let $C_1(p,q,\Delta,K)$ be a constant such that the $\Omega(d^{-R})$ bound on the log-Sobolev constant guaranteed by Lemma~\ref{lemma:main-tree-mixing} (with the choice of $L = K$) is at least $C_1^{-1} d^{-R}$. For any $C_2$, if $\hat C$ is sufficiently large, a Chernoff bound (namely~\eqref{eq:Chernoff-Poisson-binomial} applied to $\bin(\hat S, |E(B_v)|/|E(\cG)|)$) implies that with probability $1-o(n^{-2})$, at least $C_1 C_2 |E(B_v)| \log n$ updates are made in $E(B_v)$ between times $T$ and $\hat T$. By $K$-sparsity of $\phi^1$, Lemma~\ref{lemma:main-tree-mixing}, and Fact~\ref{fact:log-sobolev-tv-distance}, the term in~\eqref{mixing:plus-bound} is bounded by 
 \begin{align*}
     \|\mathbb P (Z_{\hat S}^1(B_v)\in \cdot )  - \pi_{B_v}^{\phi^1}\|_\tv  &\le \frac{1}{\sqrt{\log \min_\omega \pi(\omega)}} \exp \Big( - \frac{C_1 C_2 |E(B_v)|\log n}{2C_3 |E(B_v)|}\Big) +o(n^{-2}) \\
     & \le O(\sqrt n)\cdot e^{- C_1 C_2 \log n/(2C_3)} + o(n^{-2})\,,
 \end{align*}
for some $C_3(p,q,\Delta)$; we thus have for $C_2$ sufficiently large (and therefore $\hat C$ sufficiently large), that this is at most $o(n^{-2})$.
 By the same reasoning, by $K$-sparsity of $\phi^0$, the same bound applies to~\eqref{mixing:minus-bound}.

	Finally, 
	since both $\phi^1$ and $\phi^0$ induce $K$-$\sparse$ boundary conditions on $B_v$,
	by Proposition~\ref{prop:influence-probability-new} there exists a constant $C(p,q,\Delta,K) > 0$ such that~\eqref{mixing:smp-bound} is at most
	$$
	\|\pi_{B_v}^{\phi^1} (\omega(\cN_v)\in \cdot) - \pi_{B_v}^{\phi^0}(\omega(\cN_v)\in \cdot )\|_\tv \le C\ps^{2R}\,,
	$$ 
    which is $o(n^{-1})$ by our choice of $\delta$ and~\eqref{eq:choice-of-delta}.
	 Putting these three bounds together we see that as long as $\hat C$ is sufficiently large (depending on $p,q,\Delta$) the difference in~\eqref{eq:Y-wts} is $o(n^{-1})$, from which the bound of~\eqref{eq:wts-main-theorem} follows for $n$ sufficiently large, concluding the proof.  \qed
	 
\section{Matching lower bound on the mixing time}\label{sec:proof-lower-bound}

In this section, we show a matching $\Omega(n \log n)$ lower bound on the mixing time of the FK-dynamics on a random $\Delta$-regular graph and thus complete the proof of Theorem~\ref{thm:main-fk} from the introduction. A general lower bound for the mixing time of the Glauber dynamics on spin systems
was show in~\cite{HayesSinclair}. However, the non-locality of the FK-dynamics complicates extending the ideas from~\cite{HayesSinclair} to the random-cluster setting. In~\cite{BS}, the argument from \cite{HayesSinclair} was adapted to the random-cluster model on $\mathbb Z^2$ when $p\ne p_c(q)$, but the amenability of $\mathbb Z^2$ together with the exponential decay of connectivities at $p<p_c$ was key to this extension. 

In our setting, the non-amenability of the random $\Delta$-regular graph prevents us from bounding the speed of disagreement percolation under couplings of the FK-dynamics and implementing the argument of~\cite{HayesSinclair} directly. Instead, we use the locally treelike structure of the random $\Delta$-regular graph to directly couple a projection of the model on a certain set of $n^\epsilon$ edges to a product measure on $n^\epsilon$ edges, for which the coupon collector problem gives an immediate lower bound.   

\begin{claim}\label{clm:rg-many-trees}
With $\Prrg$-probability $1-o(1)$, $\cG$ is such that there exist $n^{1/5}$ vertices whose balls of radius $\frac 15 \log_d n$ are disjoint, and are trees. 
\end{claim}
\begin{proof}
Per~\eqref{eq:CM-to-RRG}, it suffices to prove the above under $\Pcm$. We prove the claim by repeated application of Lemma~\ref{lem:ball-intersection-probability}. Namely, consider the procedure where we repeatedly take an arbitrary vertex $v$ that has not been discovered yet, and reveal its ball of radius $R$. Let $v_i$ be the $i$'th vertex to be selected in this procedure, and let $\cA_i$ be $\bigcup_{j\le i} E(B_R(v_j))$. Then, for integer $m \le n$ the probability that $(B_R(v_1),...,B_R(v_m))$ are disjoint trees, is at least
\begin{align*}
    1 - \sum_{i=1}^m \Pcm \big(B_R(v_i)\cap \cA_{i-1} = \emptyset \mbox{ or } B_R(v_i) \mbox{ is not a tree} \mid \cA_{i-1}\big)\,.
\end{align*}
By Lemma~\ref{lem:ball-intersection-probability} (using the fact that each $v_i \notin V(\cA_{i-1})$ so that $B_R^{out}(v_i) = B_R(v_i)$) each of the summands is at most $O(m d^{2R}/(n-O(md^R)))$. Taking $R= \delta \log_d n$ and $m = n^\delta$, we see that the sum above is at most $O(n^{4\delta}/(n- O(n^{2\delta}))$ which is $o(1)$ as long as  $\delta<\frac 14$. 
\end{proof}

Fix $\epsilon \in (0,1/5)$ to be taken sufficiently small later. For every $\cG$ having $n^{1/5}$ many vertices whose balls of radius $\frac 15 \log_d n$ are disjoint trees, choose arbitrarily some $n^\epsilon$ vertices amongst the $n^{1/5}$ of Claim~\ref{clm:rg-many-trees}, and for each vertex collect a representative edge incident to it to form the set $\cC = \cC_\epsilon(\cG)$. 
Our proof will rely on a coupling of the restrictions of $X_{t,\cG}$ and $\pi_{\cG}$ to $\cC$ to $\ber (\hat p )$ product chains. For this, let: 
\begin{enumerate}
\item $X_t= X_{t,\cG}$ be a realization of the FK-dynamics;
\item $Y_t= Y_{t,\cG}$ be a realization of the FK-dynamics that censors all updates in $E(\cG)\setminus \cC$;
\item $\nu$ as the product measure over $|\cC|$ many $\ber(\hat p)$ random variables. 
\end{enumerate}
As before, let $Y_t^0$ be the chain $Y_t$ initialized from the all-$0$ configuration. 
  
\begin{lemma}\label{lem:couplings-to-product-chain}
Let $\cG$ be a graph with at least $n^{1/5}$ vertices whose balls of radius $\frac 15 \log_d n$ are disjoint trees.
For every $q > 1$, integer $\Delta \ge 3$, and $p<p_u(q,\Delta)$, there exists $\epsilon>0$ sufficiently small such that we have the following for $\cC = \cC_{\epsilon}(\cG)$:
\begin{enumerate}
\item For all $t\le T= O(n \log n)$,
\begin{align*}
	\|P(X_{t}^0(\cC) \in \cdot ) - P(Y_t^0(\cC)\in \cdot)\|_\tv \le  o(1)\,.
\end{align*}
\item $\|\pi_{\cG}(\omega(\cC) \in \cdot ) - \nu\|_\tv \le  o(1)\,.$
\end{enumerate}
\end{lemma}
\begin{proof}
We start with part (1).
Our aim is to show that under the grand coupling of $X_t^0$ and $Y_t^0$, for every $t\le T = O(n \log n)$, we have $\mathbb P (X_t^0 \ne Y_t^0) \le o(1)$. Under the grand coupling, let $\mathscr T_T = (t_1,t_2,...,t_{s(T)})$ denote the sequence of times on which the updated edge is in $\cC$, so that $s(T)$ counts the number of updates in $\cC$ by time $T$.  We can then bound
\begin{align*}
    \mathbb P (X_t^0 \ne Y_t^0) \le   \mathbb P (s(T) > n^{2\epsilon}) + \mathbb P (X_t^0 \ne Y_t^0, s(T)\le n^{2\epsilon})\,.
\end{align*}
The first term on the right-hand side is at most the probability that $\mbox{Binom}(T, |\cC|/|E(\cG)|)\ge n^{2\epsilon}$ which is $o(1)$ by the Chernoff bound~\eqref{eq:Chernoff-Poisson-binomial}. It thus suffices to work on the event $s(T)\le n^{2\epsilon}$. 

 Let $R : = \frac{1}{6} \log_d n$ and let $Z_t$ be the FK-dynamics chain (coupled to $X_t, Y_t$ through the grand coupling) that freezes the configuration on $\cC  \cup (E(\cG)\setminus \bigcup_{e\in \cC} E(B_R(e)))$ to be all-$1$. Let $Z_t^0$ be the chain $Z_t$ initialized from the configuration that is all-$0$ on  $\bigcup_{e\in \cC} E(B_R(e))$ (but all-$1$ on the frozen edges). 
Observe, trivially, that $X_t^0 \le Z_t^0$ for all $t \ge 0$. Also, observe that the updates of $Z_t^0$ are stochastically dominated by Glauber updates on the union of $2|\cC|$ many $d$-ary trees $(\cT_{e,1},\cT_{e,2})_{e\in \cC}$ of depth $R$,  rooted at the endpoints of the edges of $\cC$, and each having $(1,\circlearrowleft)$ boundary conditions. By the monotonicity of the FK-dynamics, for all $t\ge 0$, we have that
\begin{align}\label{eq:Z-t-stochastic-domination}
    \mathbb P \bigg(Z_t^0 \Big(\bigcup_{e\in \cC} \big\{E(B_R(e))\setminus \{e\}\big\}\Big) \in \cdot\bigg) \preceq \bigotimes_{e\in \cC} \bigotimes_{i\in \{1,2\}} \pi_{\cT_{e,i}}^{(1,\circlearrowleft)}\,.
\end{align}

For each time $t_i \in \sT_T$, when an edge  $e_{t_i}\in \cC$ is updated, $Y_{t_i}^0(e_{t_i})$ is drawn from $\ber(\hat p)$. At the same time, $X_{t_i}^0(e_{t_i})$ is drawn from $\ber(\hat p)$ if the endpoints of $e_{t_i}$ are not connected in $X_{t_i}^0$, which in turn must occur if none of $(\cT_{e,1},\cT_{e,2})_{e\in \cC}$ have an open root-to-leaf path in $Z_t^0$, as $X_t^0 \le Z_t^0$. 

By the stochastic domination of~\eqref{eq:Z-t-stochastic-domination} on $Z_t^0$, and Lemma~\ref{lem:exp-decay-wired-tree}, the probability that the endpoints of $e_{t_i}$ are connected in $Z_{t_i}^0$ is at most $2  C (\hat p d)^{R}$; 
for $\epsilon$ sufficiently small (depending on $p,q,d$), the above is $O(n^{ - 3\epsilon})$. 
 On the event that $\{s(T) \le n^{2\epsilon}\}$, we can union bound the above probability over the $s(T)$ times in $\sT_T$, to find that $\mathbb P(X_t^0 \ne Y_t^0, s(T)\le n^{2\epsilon})$ is at most $O(n^{-\epsilon})= o(1)$ as desired.

For part (2), consider the $2|\cC|$ many $d$-ary trees $(\cT_{e,1}, \cT_{e,2})_{e\in \cC}$ emanating from the endpoints of the edges of $\cC$. 
Notice that if none of $(\cT_{e,1}, \cT_{e,2})_{e\in \cC}$ have an open root-to-leaf path, then the values $\omega(\cC)$ are conditionally distributed as a product of $\ber(\hat p)$ random variables, i.e., $\omega(\cC)$ would conditionally be distributed as $\nu(A)$. 

As such, the total-variation distance $\|\pi_{\cG}(\omega(\cC)\in \cdot) - \nu\|_\tv$ is bounded by the $\pi_{\cG}$-probability that one of $(\cT_{e,1}, \cT_{e,2})_{e\in \cC}$ has an open root-to-leaf path. By the stochastic domination 
\begin{align*}
    \pi_{\cG}(\omega(\bigcup_{e\in \cC} \cT_{e,1}\cup \cT_{e,2})\in \cdot) \preceq \bigotimes_{e\in \cC} \bigotimes_{i\in \{1,2\}} \pi_{\cT_{e,i}}^{(1,\circlearrowleft)}\,.
\end{align*}
By a union bound, the $\pi_{\cG}$-probability that one of $(\cT_{e,1}, \cT_{e,2})_{e\in \cC}$ has an open root-to-leaf path is at most 
\begin{align*}
    \sum_{e\in \cC} \sum_{i\in \{1,2\}} \pi_{\cT_{e,i}}^{(1,\circlearrowleft)}(e\leftrightarrow \partial \cT_{e,i})\,,
\end{align*}
which by Lemma~\ref{lem:exp-decay-wired-tree} is at most $2n^{\epsilon} \cdot C(\hat p d)^R$. For $\epsilon$ sufficiently small (depending on $p,q,d$) this is $o(1)$. 
\end{proof}

\begin{proof}[\textbf{\emph{Proof of Theorem~\ref{thm:main-fk}: lower bound.}}]
Take any $n$-vertex graph $\cG$ with $n^{1/5}$ many vertices whose balls of radius $\frac 15 \log_d n$ are disjoint trees.
Note that by Claim~\ref{clm:rg-many-trees}, such graphs have $\Prrg$-probability $1-o(1)$. Take $\epsilon$ sufficiently small per Lemma~\ref{lem:couplings-to-product-chain}. Consider the event $A^+ \subset \{0,1\}^{\cC}$ that at least $\hat p n^{\epsilon} - n^{2\epsilon/3}$ of the edges in $\cC$ are open. 
 Let $(\overline Y_s)$ be the standard product chain over $|\cC| = n^{\epsilon}$ many i.i.d.\ $\ber(\hat p)$ random variables, coupled to $Y_t(\cC)$ via $\overline Y_{s(t)} = Y_t(\cC)$ for all $t$,
 where $s(t)$ counts the number of updates in $\cC$ by time $t$. 
 By item (1) of Lemma~\ref{lem:couplings-to-product-chain}, for every $T = O(n\log n)$,
 \begin{align*}
	\mathbb P(X_T^0(\cC) \in A^+) & \le \mathbb P(s(T) >cn^{\epsilon} \log n^{\epsilon}) + \mathbb P(Y_T^0 \in A^+, s(T)\le cn^{\epsilon} \log n^{\epsilon})  + o(1) \\
	& \le \mathbb P (s(T) > c n^{\epsilon} \log n) + \sup_{s \le c n^{\epsilon} \log n} \mathbb P (\overline Y_s^0 \in A^+) + o(1)\,.
\end{align*}
Taking $T := c^2 n \log n^\epsilon = \Theta(n\log n)$ for $c>0$ sufficiently small, the probability that $s(T)$ is more than $c n^{\epsilon} \log n^\epsilon$ is $o(1)$ by the Chernoff bound~\eqref{eq:Chernoff-Poisson-binomial}. Turning to the middle term above, by the standard coupon collector bound, for every $c>0$ sufficiently small, 
\begin{align*}
	\sup_{s\le  c  n^{\epsilon} \log n^\epsilon} \mathbb P(\overline Y_s^0 \in A^+) \le o(1)\,.
\end{align*}
Combining the above, we obtain 
\[
\mathbb P(X_T^0(\cC) \in A^+) = o(1)\,.
\]
At the same time, by a Chernoff bound, $$\nu(A^+) = \mathbb P(\bin(n^\epsilon,\hat p)\ge \hat p n^\epsilon - n^{2\epsilon/3}) = 1-o(1)\,,$$ so that by item (2) of Lemma~\ref{lem:couplings-to-product-chain}, we have $\pi_{\cG}(A^+) = 1-o(1)$. This implies the mixing time is at least $T = \Omega(n\log n)$ as claimed.
\end{proof}

\subsection*{Acknowledgements} The authors thank the anonymous referee for detailed comments on the manuscript. 
The research of A.B.\ was supported in part by NSF grant CCF-1850443. R.G.\ thanks the Miller Institute for Basic Research in Science for its support.

\bibliographystyle{abbrv}
\bibliography{references}

\end{document}